\documentclass[10pt]{amsart}
\usepackage{amsmath,amssymb,amsthm,amsrefs,a4wide}
\usepackage{latexsym}
\usepackage{indentfirst}
\usepackage{graphicx}
\usepackage{placeins}
\usepackage{booktabs}
\usepackage{algorithm}
\usepackage{algorithmic}
\usepackage{multirow}
\usepackage{subfig}
\usepackage{graphicx,color}
\usepackage{diagbox}
\usepackage[normalem]{ulem}
\usepackage{hyperref}
\hypersetup{
    colorlinks=true,
}

\theoremstyle{plain}
\newtheorem{theorem}{Theorem}
\newtheorem{lemma}[theorem]{Lemma}

\newtheorem{prop}[theorem]{Proposition}
\newtheorem{cor}[theorem]{Corollary}

\theoremstyle{definition}
\newtheorem{defn}[theorem]{Definition}

\theoremstyle{remark}
\newtheorem{remark}{Remark}
\newtheorem{condition}{Condition}

\newtheoremstyle{cited}%
  {3pt}
  {3pt}
  {\itshape}
  {}
  {\bfseries}
  {.}
  {.5em}
  {\thmname{#1} \thmnumber{#2} \thmnote{\normalfont#3}}

\theoremstyle{cited}

\renewcommand{\tilde}{\widetilde}

\renewcommand{\ell}{l}

\newcommand{\bbm}{\begin{bmatrix}}
\newcommand{\ebm}{\end{bmatrix}}

\newcommand{\bR}{\mathbb{R}}
\newcommand{\bN}{\mathbb{N}}

\definecolor{fern}{RGB}{79,143,0}

\usepackage{amsmath}
\usepackage{mathtools}
\usepackage{etoolbox}
\DeclareFontFamily{U}{matha}{\hyphenchar\font45}
\DeclareFontShape{U}{matha}{m}{n}{
      <5> <6> <7> <8> <9> <10> gen * matha
      <10.95> matha10 <12> <14.4> <17.28> <20.74> <24.88> matha12
      }{}
\DeclareSymbolFont{matha}{U}{matha}{m}{n}
\DeclareFontSubstitution{U}{matha}{m}{n}

\DeclareFontFamily{U}{mathx}{\hyphenchar\font45}
\DeclareFontShape{U}{mathx}{m}{n}{
      <5> <6> <7> <8> <9> <10>
      <10.95> <12> <14.4> <17.28> <20.74> <24.88>
      mathx10
      }{}
\DeclareSymbolFont{mathx}{U}{mathx}{m}{n}
\DeclareFontSubstitution{U}{mathx}{m}{n}

\DeclareMathDelimiter{\vvvert}{0}{matha}{"7E}{mathx}{"17}
\DeclarePairedDelimiterX{\normi}[1]
  {\vvvert}
  {\vvvert}
  {\ifblank{#1}{\:\cdot\:}{#1}}

\title{Generative modeling via tensor train sketching}

\author[]{YoonHaeng Hur\textsuperscript{$\ast$}
\and
Jeremy G Hoskins\textsuperscript{$\ast$}
\and
Michael Lindsey\textsuperscript{$\dagger$}
\and 
E.M. Stoudenmire\textsuperscript{$\ddagger$}
\and
Yuehaw Khoo\textsuperscript{$\ast$}
}

\email{yoonhaenghur@uchicago.edu}
\email{jeremyhoskins@uchicago.edu}
\email{michael.lindsey@cims.nyu.edu}
\email{mstoudenmire@flatironinstitute.org}
\email{ykhoo@uchicago.edu}

\begin{document}

\maketitle
\vspace{-15pt}
\begin{center}
    $^\ast$ \textit{Department of Statistics, University of Chicago}
\end{center}
\begin{center}
    $^\dagger$ \textit{Department of Mathematics, Courant Institute of Mathematical Sciences, New York University}
\end{center}
\begin{center}
    $^\ddagger$ \textit{Center for Computational Quantum Physics, Flatiron Institute}
\end{center}

\begin{abstract}
In this paper, we introduce a sketching algorithm for constructing a tensor train representation of a probability density from its samples. Our method deviates from the standard recursive SVD-based procedure for constructing a tensor train. Instead, we formulate and solve a sequence of small linear systems for the individual tensor train cores. This approach can avoid the curse of dimensionality that threatens both the algorithmic and sample complexities of the recovery problem. Specifically, for Markov models under natural conditions, we prove that the tensor cores can be recovered with a sample complexity that scales logarithmically in the dimensionality. Finally, we illustrate the performance of the method with several numerical experiments.
\end{abstract}

\section{Introduction}
Given independent samples from a probability distribution, learning a \emph{generative model} \cite{ruthotto2021generative} that can produce additional samples is a task of fundamental importance in machine learning and data science. The generative modeling of high-dimensional probability distributions has seen significant recent progress, particularly due to the use of neural-network based parametrizations within both old and new paradigms such as generative adversarial networks (GANs) \cite{goodfellow2014generative}, variational autoencoders (VAE) \cite{kingma2013auto}, and normalizing flows \cites{tabak2010density,rezende2015variational}. Among these three major paradigms, only normalizing flows furnish an analytic formula for the probability density function, and in all cases the computation of downstream quantities of interest can only be achieved via Monte Carlo sampling-based approaches with a relatively low order of convergence. 

More precisely, suppose we are given $N$ independent samples 
\[
(y_1^{(1)}, \ldots, y_d^{(1)}), \ \ldots \ , (y_1^{(N)}, \ldots , y_d^{(N)}) \sim p^\star
\]
drawn from an underlying probability density $p^\star \colon \mathbb{R}^d\rightarrow \mathbb{R}$, our goal is to estimate $p^\star$ from the empirical distribution
\begin{equation}
    \widehat p(x_1,\ldots,x_d) = \frac{1}{N} \sum_{i = 1}^{N} \delta_{(y_1^{(i)}, \ldots, y_d^{(i)})}(x_1,\ldots,x_d),
\end{equation}
where $\delta_{(y_1,\ldots,y_d)}$ is the $\delta$-measure supported on $(y_1,\ldots,y_d) \in \mathbb{R}^d$. In this paper, assuming that the underlying density $p^\star$ takes a low-rank tensor train (TT) \cite{oseledets2011tensor} format (known as a matrix product state (MPS) in the physics literature \cites{white1992density, perez2006matrix}), we propose and analyze an algorithm that outputs a TT format of $\hat{p}$ to estimate $p^\star$. Such a TT ansatz has found applications in generative modeling; for instance, \cite{PhysRevX.8.031012} (and its extension \cite{cheng2019tree}) utilizes it to learn the distribution of handwritten digit images. In particular, the TT ansatz offers several benefits. First, generating independent and identically distributed (i.i.d.) samples can be done efficiently by applying conditional distribution sampling \cite{dolgov2020approximation} to the obtained TT format; it can also be used for other downstream tasks, such as direct (deterministic) computation of the moments. However, in order to exploit these benefits, we need to be able to determine the TT representation efficiently. Our algorithm, which we name Tensor Train via Recursive Sketching (TT-RS), provides computationally/statistically efficient estimation of $p^\star$, making the following contributions. 
\begin{itemize}
    \item By a sketching technique, we can estimate the tensor components of the TT via a sequence of linear systems, with a complexity that is linear in both the dimension $d$ and the sample size $N$. 
    \item In the setting of a Markovian density with dimension-independent transition kernels, we prove that the tensor cores can be estimated from a number of samples that scales as $\log(d)$. 
\end{itemize}

\subsection{Prior work}\label{sec:prior_work}
In the literature, generally two types of input data are considered for the recovery of low-rank TTs. In the first case, one assumes that one has the ability to evaluate a $d$-dimensional function $p$ at arbitrary points and seeks to recover $p$ in a TT format with a limited (in particular, polynomial in $d$) number of evaluations. In this context, various methods such as TT-cross \cite{oseledets2010tt}, DMRG-cross \cite{oseledetstt}, and TT completion \cite{steinlechner2016riemannian} have been considered. Furthermore, generalizations such as \cites{wang2017efficient,khoo2021efficient} have been developed to treat densities which have a tensor ring structure. In the second case, which is the case of this paper, one only has access to a fixed collection of empirical samples from the density. Importantly, one does not have access to the value of the density at the given samples. In this case, the ideas of the TT methods that we mentioned earlier cannot be applied directly.

In order to understand how the proposed method differs from the previous methods, we first show that in generative modeling, the nature of the problem is different. More precisely, we are mainly dealing with an estimation problem rather than an approximation problem, where we want to estimate the underlying density $p^\star$ that gives the empirical distribution $\hat{p}$, in terms of a TT. In such a generative modeling setting, suppose one designs an algorithm $\mathcal{A}$ that takes any $d$-dimensional function $p$ and gives $\mathcal{A}(p)$ as a TT, then one would like such $\mathcal{A}$ to minimize the following differences 
\begin{equation*}
    p^\star - \mathcal{A}(\hat{p}) = \underbrace{p^\star - \mathcal{A}(p^\star)}_{\text{approximation error}} + \quad \underbrace{\mathcal{A}(p^\star) - \mathcal{A}(\hat{p})}_{\text{estimation error}}.
\end{equation*}
In generative modeling, $\hat{p}$ suffers from sample variance, which leads to variance in $\mathcal{A}(\hat{p})$ and hence the estimation error. Our focus is to reduce such an error so that there is no curse of dimensionality in estimating $p^\star$. While our method is inspired by sketching ideas from randomized linear algebra \cites{nakatsukasa2021fast,rokhlin2008fast}, which have found applications in the tensor computation field \cites{che2019randomized,daas2021randomized,sun2020low}, there are several notable differences with the current literature.

\begin{itemize}
    \item \textbf{In relation to TT-compression algorithms}: Algorithms based on singular value decomposition (SVD) \cite{oseledets2011tensor} and randomized linear algebra \cites{oseledets2010tt, oseledetstt, shi2021parallel} aim to compress the input function $p$ as a TT such that $\mathcal{A}(p)\approx p$. If such a compression is successful, the above approximation error can be made small, that is, $p^\star - \mathcal{A}(p^\star) \approx 0$, and we also have $\mathcal{A}(\hat{p})\approx \hat{p}$; accordingly, the estimation error becomes $\mathcal{A}(p^\star) - \mathcal{A}(\hat{p}) \approx p^\star - \hat{p}$. Such an estimation error, however, grows exponentially in $d$ when having a fixed number of samples. In this paper, we focus on developing methods that reduce the estimation error due to sample variance such that there is no curse of dimensionality, and such a setting has not been considered in the previous TT-compression literature. 
    
    A recent work \cite{shi2021parallel} determines a TT from values of a high-dimensional function in a computationally distributed fashion. In particular, \cite{shi2021parallel} forms an independent set of equations with sketching techniques from randomized linear algebra to determine the tensor cores in a parallel way. While our method has similarities with \cite{shi2021parallel}, our goal, which is to estimate a TT based on empirical samples of a density, is different from \cite{shi2021parallel}. Therefore, the purpose and means of sketching are fundamentally different. We apply sketching such that each equation in the independent system of equations has size that is constant with respect to the dimension of the problem (unlike the case in \cite{shi2021parallel}), and hence we can estimate the coefficient matrices of the linear system in a statistically efficient way. Furthermore, our use of parallelism in setting up the system is mainly to prevent error accumulation in the estimation of tensor cores.
    
    \item \textbf{In relation to optimization-based algorithms}: A more principled approach for estimating the underlying density $p^\star$ is to perform maximum likelihood estimation, i.e. minimizing the Kullback-Leibler (KL) divergence between the TT ansatz and the empirical distribution \cites{PhysRevX.8.031012, bradley2020modeling, novikov2021tensor}. Although maximum likelihood estimation is statistically efficient in terms of having a low-variance estimator, due to the non-convex nature of the minimization, these methods can suffer from local minima. Furthermore, these iterative procedures require multiple passes over $N$ data points. In contrast, the method described in this paper recovers the cores with a single sweep across all tensor cores.
\end{itemize}

\subsection{Organization}
The paper is organized as follows. First, we briefly describe the main idea of our algorithm in Section~\ref{section:motivation}. Details of the algorithm are presented in Section~\ref{section:main alg} and conditions for the algorithm to work are discussed in Section~\ref{section:exact recovery}. In Section~\ref{section:discrete markov}, we examine how the conditions in Section~\ref{section:exact recovery} lead to exact and stable recovery of tensor cores under a Markov model assumption of the density. In Section~\ref{section:numerical}, we illustrate the performance of our algorithm with several numerical examples. We conclude in Section~\ref{section:conclusion}.

\subsection{Notations}
For an integer $n \in \bN$, we define $[n] = \{1, \ldots, n\}$. Note that for $m, n \in \bN$, a function $c \colon [m] \times [n] \to \bR$ may also be viewed as a matrix of size $m \times n$. We alternate between these two viewpoints often throughout the paper. For any $a, b \in \bR$, we define $a \vee b := \max(a, b)$ and $a \wedge b := \min(a, b)$. For $a,b\in \mathbb{N}$ where $b\geq a$, we may use the ``$\textsf{MATLAB}$ notation'' $a:b$ to denote the set $\{a,a+1\ldots,b\}$.

Our primary objective in this paper is to obtain a TT representation of a $d$-dimensional function. Throughout the remainder of this paper, we fix a $d$-dimensional function $p \colon X_1 \times \cdots \times X_{d} \to \bR$, where $X_1, \ldots, X_{d} \subset \bR$. Unless stated otherwise, $p$ may not be a density, that is, it can take negative values or its integral may not be 1. Whenever we are interested in a density, we will mention explicitly that $p$ is a density or use $p^\star$ instead.

\begin{defn}\label{def:TT}
    We say that $p$ admits a TT representation of rank $(r_1, \ldots, r_{d - 1})$ if there exist $G_1 \colon X_1 \times [r_1] \to \bR$, $G_k \colon [r_{k - 1}] \times X_{k} \times [r_{k}] \to \bR$ for $k = 2, \ldots, d - 1$, and $G_{d} \colon [r_{d - 1}] \times X_{d} \to \bR$ such that 
    \begin{equation*}
        p(x_1, \ldots, x_{d}) = \sum_{\alpha_{1} = 1}^{r_1} \cdots \sum_{\alpha_{d - 1} = 1}^{r_{d - 1}} G_1(x_1, \alpha_1) G_2(\alpha_1, x_2, \alpha_2) \cdots G_{d-1} (\alpha_{d-2},x_{d-1},\alpha_{d-1}) G_d(\alpha_{d - 1}, x_{d})
    \end{equation*}
    for all $(x_1, \ldots, x_d) \in X_1 \times \cdots \times X_d$. In this case, we call $G_1, \ldots, G_{d}$ the {\it cores} of $p.$ For notational simplicity, in the following we often replace the right-hand side of the above equation (and similar expressions involving contractions of several tensors) with $G_1 \circ \cdots \circ G_d$, where `$\circ$' represents the contraction of the cores. We will also sometimes express the TT representation of $p$ diagrammatically as shown in Figure \ref{fig:TT}.
\end{defn}

\begin{figure}[ht]
    \centering
    \includegraphics[width=0.6\columnwidth]{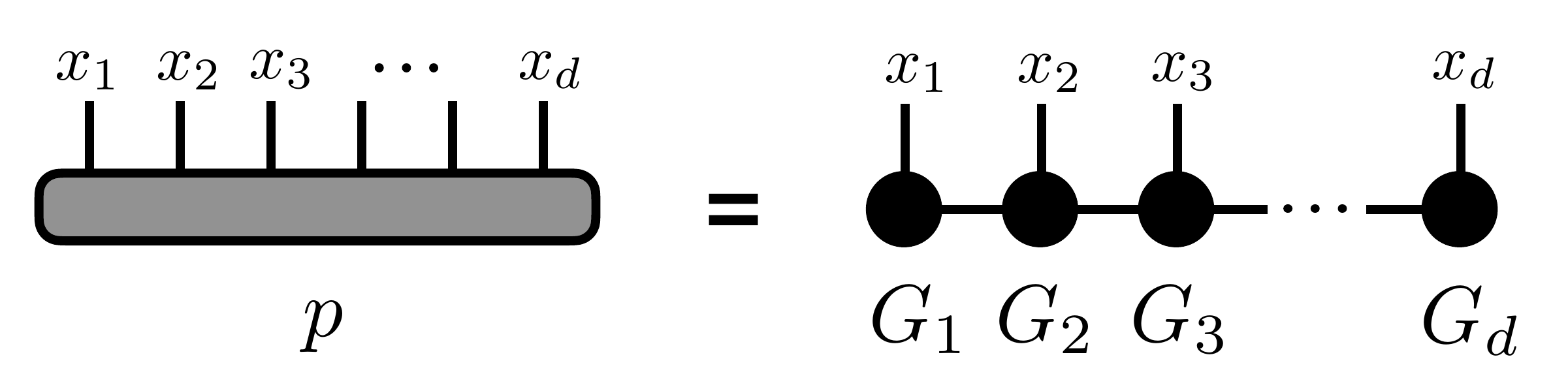}
    \caption{Tensor diagram illustrating the TT representation of $p$ (Definition \ref{def:TT}). The variables $x_i$'s correspond to the outward solid lines in both sides. The solid lines between two adjacent cores on the right-hand side depict the contraction. See \cite{chen2021committor} for a detailed introduction to tensor network diagram notation.}
    \label{fig:TT}
\end{figure}

\begin{remark}
In Definition \ref{def:TT}, the sets $X_1, \ldots, X_d \subset \bR$ may be infinite; in such a case, the representation in Definition \ref{def:TT} is also called a functional TT representation \cites{gorodetsky2019continuous,bigoni2016spectral}.
\end{remark}

Finally, when working with high-dimensional functions, it is often convenient to group the variables into two subsets and think of the resulting object as a matrix. We call these matrices {\it unfolding matrices}. In particular, for $k = 1 , \ldots, d - 1$, we define the $k$-th unfolding matrix by $p(x_1, \ldots, x_{k}; x_{k + 1}, \ldots, x_{d})$; namely, group the first $k$ and the last $d - k$ variables to form rows and columns, respectively. In certain situations, for ease of exposition we write $x_\mathcal{S}$ to denote the joint variable $(x_{i_1},\ldots,x_{i_k})$, where $\mathcal{S} = \{i_1 ,\ldots , i_k \}$ and $1\leq i_1 \leq \cdots \leq i_k \leq d$. For example, we may write $p(x_1, \ldots, x_{k}; x_{k + 1}, \ldots, x_{d})$ as $p(x_{1:k}; x_{k + 1:d})$.

\section{Main idea of the algorithm}
\label{section:motivation}
In this section, we sketch the main idea of the TT-RS algorithm. We start with the following simple observation in the discrete case, i.e., the case where $p \colon [n_1] \times \cdots \times [n_d] \to \bR$ for $n_1, \ldots, n_d \in \bN$. Supposing that $p$ is representable in a TT format with rank $(r, \ldots, r),$ then the $k$-th unfolding matrix $p(x_{1 : k}; x_{k + 1 : d})$ is low-rank. Indeed, we can write 
\begin{equation*}
    p(x_{1 : k}; x_{k + 1 : d}) = \sum_{\alpha_k = 1}^{r} \Phi_k(x_{1 : k}; \alpha_{k}) \Psi_k(\alpha_{k}; x_{k + 1 : d})
\end{equation*}
for some $\Phi_k \colon [n_1] \times \cdots \times [n_k] \times [r] \to \bR$ and $\Psi_k \colon [r] \times [n_{k + 1}] \times \cdots \times [n_{d}] \to \bR$. On the other hand, the TT-format assumption on $p$ implies that there exist $G_1,\dots,G_d$ such that
\begin{align*} 
    \Phi_k(x_{1 : k}, \alpha_{k}) := \sum_{\alpha_1 = 1}^{r} \cdots \sum_{\alpha_{k - 1} = 1}^{r} G_1(x_1, \alpha_1) \cdots G_{k}(\alpha_{k - 1}, x_{k}, \alpha_{k}), \\
    \Psi_k(\alpha_{k}, x_{k + 1 : d}) := \sum_{\alpha_{k + 1} = 1}^{r} \cdots \sum_{\alpha_{d - 1} = 1}^{r} G_{k + 1}(\alpha_k, x_{k + 1}, \alpha_{k + 1}) \cdots G_{d}(\alpha_{d - 1}, x_{d}),
\end{align*}
so that $p = G_1\circ\cdots\circ G_d$. In other words, contractions of the first $k$ and the last $d - k $ cores of $G_1,\ldots,G_d$ yield spanning vectors for the $r$-dimensional column and the row spaces, respectively, of the $k$-th unfolding matrix.

This observation motivates the following procedure to obtain the cores. Suppose that the rank of the $k$-th unfolding matrix of $p$ is $r$. We consider $\Phi_k \colon [n_1] \times \cdots \times [n_k] \times [r] \to \bR$ such that the column space of $\Phi_k(x_{1 : k}; \alpha_k)$ is the same as that of the $k$-th unfolding matrix; for instance, a suitable $\Phi_k$ can be constructed by forming the SVD of the $k$-th unfolding matrix $p(x_{1 : k}; x_{k + 1 : d})$ and setting $\Phi_k(x_{1:k}; \alpha_k)$ to be the matrix of left-singular vectors. Next, we attempt to find cores $G_1, \ldots, G_{d - 1}$ such that 
\begin{equation}
\label{eq:previous}
    \Phi_k(x_{1 : k}, \alpha_{k}) = \sum_{\alpha_1 = 1}^{r} \cdots \sum_{\alpha_{k - 1} = 1}^{r} G_1(x_1, \alpha_1) \cdots G_{k}(\alpha_{k - 1}, x_{k}, \alpha_{k})
\end{equation}
for $k = 1,\ldots ,  d - 1$. Equivalently, we let $G_1 = \Phi_1$ and solve the following equations for the cores $G_k \colon [r_{k - 1}] \times [n_k] \times [r_{k}] \to \bR$ for $k = 2, \ldots,  d - 1$:
\begin{equation}
\label{eq:ours}
    \Phi_k(x_{1 : k}, \alpha_{k}) = \sum_{\alpha_{k - 1} = 1}^{r} \Phi_{k - 1}(x_{1 : k - 1}, \alpha_{k - 1}) G_{k}(\alpha_{k - 1}, x_{k}, \alpha_{k}).
\end{equation}
The above discussion has also been studied in \cites{shi2021parallel,daas2022parallel}. For completeness, we formally state it as follows.
\begin{prop}
\label{prop:CDEs}
    For each $k = 1, \ldots, d - 1$, suppose that the rank of the $k$-th unfolding matrix of $p$ is $r_k$ and define $\Phi_k \colon [n_1] \times \cdots \times [n_k] \times [r_k] \to \bR$ so that the column space of $\Phi_k(x_{1 : k}; \alpha_k)$ is the same as that of the $k$-th unfolding matrix of $p$. Consider the following $d$ matrix equations with unknowns $G_1 \colon [n_1] \times [r_1] \to \bR$, $G_k \colon [r_{k - 1}] \times [n_k] \times [r_{k}] \to \bR$ for $k = 2, \ldots, d - 1$, and $G_{d} \colon [r_{d - 1}] \times [n_d] \to \bR$:
    \begin{equation}
        \label{eq:CDEs}
        \begin{aligned}
            G_1(x_1; \alpha_1) &= \Phi_1(x_1; \alpha_1), \\
            \sum_{\alpha_{k - 1} = 1}^{r_{k - 1}} \Phi_{k - 1}(x_{1 : k - 1}; \alpha_{k - 1}) G_{k}(\alpha_{k - 1}; x_{k}, \alpha_{k}) &= \Phi_k(x_{1 : k - 1}; x_k, \alpha_{k}) \quad k = 2, \ldots, d - 1, \\
            \sum_{\alpha_{d-1} = 1}^{r_{d - 1}} \Phi_{d-1}(x_{1 : d - 1}; \alpha_{d-1}) G_d(\alpha_{d-1}; x_d) &= p(x_{1 : d - 1}; x_d).
        \end{aligned}
    \end{equation}
    Then, each equation of \eqref{eq:CDEs} has a unique solution, and the solutions $G_1, \ldots, G_{d}$ satisfy 
    \begin{equation}
    \label{eq:prop-core}
        p(x_1, \ldots, x_{d}) = \sum_{\alpha_1 = 1}^{r_1} \cdots \sum_{\alpha_{d - 1} = 1}^{r_{d - 1}} G_1(x_1, \alpha_1) \cdots G_{d}(\alpha_{d - 1}, x_{d}).
    \end{equation}
    Hence, by solving these equations we obtain a TT representation of $p$ with cores $G_1, \ldots, G_d$. We call \eqref{eq:CDEs} the Core Determining Equations (CDEs) formed by $\Phi_1, \ldots, \Phi_{d - 1}$.
\end{prop}
Proposition \ref{prop:CDEs}, which we prove in Appendix~\ref{app:CDE}, implies that the cores $G_k$ can be obtained by solving matrix equations. That said, it should be noted that the coefficient matrices of the CDEs, $\Phi_k(x_{1:k}; \alpha_{k})$ for $k=1,\ldots ,d - 1$, are exponentially sized in the dimension $d$. 

In what follows, we take an approach that is similar in spirit to the ``sketching'' techniques commonly employed in the randomized SVD literature \cite{halko2011finding}, which are used to dramatically reduce the computational cost of computing the SVD of several broad classes of matrices. In this paper, however, sketching plays a fundamentally different role. Here, sketching is crucial for the stability of the algorithm, though it also yields an improvement in computational complexity. For our problem, i.e., to determine a TT from samples, the most important function of sketching is to reduce the size of CDEs such that the reduced coefficient matrices can be estimated efficiently with a small sample size $N$. Furthermore, the choice of sketches cannot be arbitrary (e.g., Gaussian random matrices) but must be chosen carefully to reduce the variance of the coefficient matrices as much as possible. The features and requirements of this sketching strategy are particularly apparent in the case of their application to Markov models, which is treated in Section~\ref{section:discrete markov}. More concretely, in order to reduce the size of the CDEs, for some function $S_{k-1} \colon [m_{k-1}] \times [n_1] \times \cdots [n_{k-1}]\rightarrow \bR,$ contracting $S_{k-1}$ against \eqref{eq:CDEs} (i.e., multiplying both sides by $S_{k-1}$ and summing over $x_1,\dots,x_{k-1}$) we find: 
\begin{equation}\label{eq:CDEs_reduced_rough}
    \begin{split}
        & \sum_{\alpha_{k - 1} = 1}^{r_{k - 1}} \left(\sum_{x_1 = 1}^{n_1} \cdots \sum_{x_{k - 1} = 1}^{n_{k - 1}}  S_{k-1}(\beta_{k-1};x_{1:k-1}) \Phi_{k - 1}(x_{1:k-1}; \alpha_{k - 1})\right) G_{k}(\alpha_{k - 1}; x_{k}, \alpha_{k}) \\
        & = \sum_{x_1 = 1}^{n_1} \cdots \sum_{x_{k - 1} = 1}^{n_{k - 1}}  S_{k-1}(\beta_{k-1};x_{1:k-1}) \Phi_k(x_{1:k-1}; x_{k}, \alpha_{k}).
    \end{split}
\end{equation}
Note that the number of rows of the new coefficient matrix on the left-hand side of \eqref{eq:CDEs_reduced_rough} is $m_{k - 1}$. Hence, sketching in this way reduces the number of equations to $m_{k-1} n_k r_k$ when determining each $G_k$. Of course, one must be careful to choose suitable {\it sketch functions} $S_{k-1}$, as mentioned previously. As we shall see, $\Phi_k$'s are also obtained from some right sketching functions $T_{k} \colon [n_{k+1}] \times \cdots [n_d] \times [l_k] \rightarrow \bR$ to be contracted with $p$
over the variables $x_{k+1}, \ldots, x_d$.

In the next section, we present the details of the proposed algorithm, TT-RS, which gives a set of equations of the form \eqref{eq:CDEs_reduced_rough}.

\begin{remark}
We pause here to comment on why we solve \eqref{eq:previous} in the form of \eqref{eq:ours}. To solve \eqref{eq:previous}, one can in principle determine $G_1,\ldots, G_d$ successively, i.e. after determining $G_1,\ldots G_{k-1}$, plug them into \eqref{eq:previous} to solve for $G_k$. In principle, this is the same as solving \eqref{eq:ours} where each $G_1,\ldots,G_d$ is determined independently. But in practice, when $\Phi_k$'s contain noise, determining $G_1,\ldots,G_d$ successively via substitutions leads to noise accumulation. As we will see later, solving the independent set of equations \eqref{eq:ours} is more robust against perturbations on the coefficients $\Phi_k$'s. We again remark that this independent set of equations is similar to the ones presented in a recent work \cite{shi2021parallel}. However, as mentioned in Section \ref{sec:prior_work}, our main algorithm presented in the next section is designed to improve statistical estimation, where it is instrumental to reduce the size of the coefficients $\Phi_k$'s via the sketching using $S_{k - 1}$'s, whereas equations in \cite{shi2021parallel} are exponentially large.
\end{remark}

\section{Description of the main algorithm: TT-RS}\label{section:main alg}

In this section, we present the algorithm TT-RS (Algorithm \ref{alg:1} below) for the case of determining a TT representation of any discrete $d$-dimensional function $p$, where we assume $p \colon [n_1] \times \cdots \times [n_d] \to \bR$ for some $n_1, \ldots, n_d \in \bN$. The stages of Algorithm \ref{alg:1} are depicted in Figure \ref{fig:algorithm}.

\begin{algorithm}
\caption{TT-RS for a discrete function $p$.}
\label{alg:1}
\begin{algorithmic}[1]
\REQUIRE $p \colon [n_1] \times \cdots \times [n_d] \to \bR$ and target ranks $r_1, \ldots, r_{d - 1}$.
\REQUIRE $T_{k} \colon [n_{k}] \times \cdots \times [n_d] \times [\ell_{k - 1}] \to \bR$ with $\ell_{k - 1} \ge r_{k - 1}$ for $k = 2, \ldots, d$.
\REQUIRE $s_1 \colon [m_1] \times [n_1] \to \bR$ and $s_k \colon [m_k] \times [n_k] \times [m_{k - 1}] \to \bR$ for $k = 2, \ldots,  d - 1$.
\STATE $\tilde \Phi_1, \ldots, \tilde \Phi_{d} \leftarrow \textsc{Sketching}(p, T_2, \ldots, T_d, s_1, \ldots, s_{d - 1})$.
\STATE $B_1, \ldots, B_{d} \leftarrow \textsc{Trimming}(\tilde \Phi_1, \ldots, \tilde \Phi_{d}, r_1,\ldots,r_{d-1})$.
\STATE $A_1, \ldots, A_{d - 1} \leftarrow \textsc{SystemForming}(B_1,\ldots,B_{d - 1}, s_1, \ldots, s_{d - 1})$.
\STATE Solve the following $d$ matrix equations via least-squares for the variables $G_1 \colon [n_1] \times [r_1] \to \bR$, $G_k \colon [r_{k - 1}] \times [n_k] \times [r_{k}] \to \bR$ for $k = 2,\ldots,  d - 1$, and $G_{d} \colon [r_{d - 1}] \times [n_d] \to \bR$: \vspace{2mm}
\begin{equation}
    \label{eq:alg-CDEs}
    \begin{aligned}
        G_1 & = B_1, \\
        \sum_{\alpha_{k-1} = 1}^{r_{k - 1}} A_{k-1}(\beta_{k-1}; \alpha_{k-1}) G_{k}(\alpha_{k-1}; x_{k}, \alpha_{k}) & = B_{k}(\beta_{k-1}; x_{k}, \alpha_{k}) \quad k = 2,\ldots,  d - 1, \\
        \sum_{\alpha_{d-1} = 1}^{r_{d - 1}} A_{d-1}(\beta_{d-1}; \alpha_{d-1}) G_d(\alpha_{d-1}; x_d) & = B_d(\beta_{d-1}; x_{d}).
    \end{aligned}
\end{equation}
\RETURN $G_1, \ldots, G_d$ 
\end{algorithmic}
\end{algorithm}

\begin{figure}[ht]
    \centering
    \includegraphics[width=0.8\textwidth]{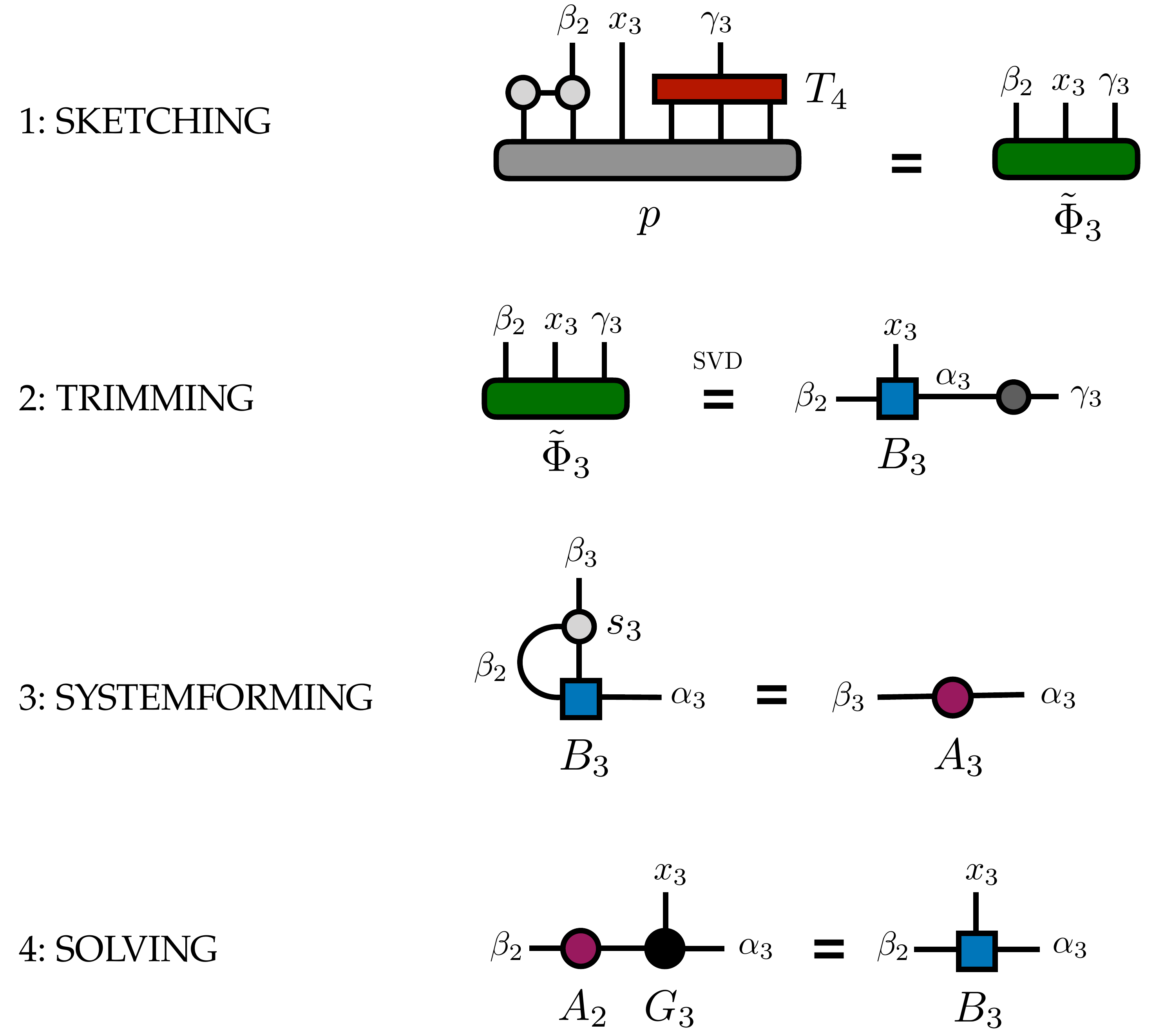}
    \caption{Tensor diagrams illustrating the four steps of the TT-RS algorithm (Algorithm \ref{alg:1}), explicitly showing the case of step $k=3$ for a $d=6$ dimensional distribution. \textsc{Sketching} produces $\tilde \Phi_k$ by applying sketch functions $S_{k - 1}$ and $T_{k + 1}$ to $p$. \textsc{Trimming} generates $B_k$ from $\tilde \Phi_k$ using the SVD. \textsc{SystemForming} outputs $A_k$ based on $B_k$. Lastly, collecting the outputs $A_k$'s and $B_k$'s, we form \eqref{eq:alg-CDEs} and solve for $G_k$'s. See Section \ref{subsec:s+d} for full details.}
    \label{fig:algorithm}
\end{figure}

Algorithm \ref{alg:1} is divided into four parts: \textsc{Sketching} (Algorithm \ref{alg:1-2}), \textsc{Trimming} (Algorithm \ref{alg:1-3}), \textsc{SystemForming} (Algorithm \ref{alg:1-4}), and solving $d$ matrix equations \eqref{eq:alg-CDEs}. As input, algorithm \ref{alg:1} requires functions $T_2, \ldots, T_d$ and $s_1, \ldots, s_{d - 1}$; we call them {\it right} and {\it left sketch functions}, respectively. \textsc{Sketching} applies these sketch functions to $p$ so that $\tilde \Phi_k$ resembles the right-hand side of the reduced CDEs~\eqref{eq:CDEs_reduced_rough}. In particular, if $l_k$ denotes the number of right sketches and we set $l_k = r_k$ for each $k$ where $r_1,\ldots,r_{d-1}$ are the target ranks of the TT, then one could in principle replace the right-hand side of \eqref{eq:CDEs_reduced_rough} with $\tilde \Phi_k$. In practice, we choose $l_k>r_k$, and use \textsc{Trimming} to generate suitable $B_k$'s, to be defined below, from the corresponding $\tilde \Phi_k$'s. These can in turn be used to form a right-hand side in the sense of~\eqref{eq:CDEs_reduced_rough}. Lastly, based on $B_1,\ldots,B_{d-1}$, \textsc{SystemForming} outputs $A_1,\ldots,A_{d-1}$, which resemble the coefficient matrices on the left-hand side of \eqref{eq:CDEs_reduced_rough}. Detailed descriptions of each subroutine are given in the following subsection. In what follows, we constantly refer back to Section~\ref{section:motivation} to motivate the algorithm.

\begin{remark}
The choice of sketch functions is based on two criteria: (i) When $p$ actually has an underlying TT representation, solving the equations \eqref{eq:alg-CDEs} should produce suitable cores $G_1,\ldots,G_d$. A proof of such an exact recovery property is given in Section~\ref{section:exact recovery}, where we also discuss the conditions that the sketch functions have to satisfy. (ii) Let $\hat{G}_1, \ldots, \hat{G}_d$ be the results of TT-RS with $\hat{p}$ as input, where $\hat p$ is an empirical distribution constructed based on i.i.d.\ samples from some density $p^\star$. We would like to have $p^\star \approx \hat{G}_1 \circ \cdots \circ \hat{G}_d$ if $\hat{p}$ is a good approximation to $p^\star$. This requires $A_1, \ldots, A_{d-1}, B_1, \ldots, B_d$ to have a small variance, and the variance of these objects depends on the choice of sketches. We discuss these considerations in Section~\ref{section:discrete markov} for Markov models.
\end{remark}

\begin{remark}
The above algorithm is written only for the case where we consider densities $p$ over a finite state space. However, if $p$ is in $L_2(X_1)\times \dots \times L_2(X_d)$ then one can pass to a suitable tensor product of orthonormal bases in each dimension and work with truncated coefficient tensors instead. We summarize the necessary modifications required for continuous functions in Appendix \ref{app:cttrs}. We call this ``continuous'' version of the algorithm TT-RS-Continuous (TT-RS-C) (Algorithm~\ref{alg:2}). There will, of course, be a new source of error associated with the choice of how to truncate the coefficients. Standard estimates from approximation theory can be used to relate the smoothness of $p$ to the decay of coefficients in each dimension.
\end{remark}

\subsection{Details of the subroutines} 
\label{subsec:s+d}
In this section, we provide details of the three main subroutines used in TT-RS. First, \textsc{Sketching} (Algorithm \ref{alg:1-2}) converts each unfolding matrix of $p$ into a smaller matrix using sketch functions. For each $ k = 2, \ldots,  d - 1$, by contracting the $k$-th unfolding matrix of $p$ with the right and left sketch functions, $T_{k + 1}$ and $S_{k - 1}$, we obtain $\tilde{\Phi}_{k}$, which can be thought of as a three-dimensional tensor of size $\bR^{m_{k - 1} \times n_k \times \ell_{k}}$ as in Step 1 of Figure \ref{fig:algorithm}. This ``sketched'' version of the $k$-th unfolding matrix of $p$ is no longer exponentially large in $d$. In \textsc{Sketching}, each $\bar \Phi_k$ plays the role of $\Phi_k$ in the left-hand side of \eqref{prop:CDEs}, which captures the range of the $k$-th unfolding matrix of $p$. The extra ``bar'' in the notation for $\bar \Phi_k$ is used to distinguish this object from $\Phi_k$, as $\bar \Phi_k(x_{1:k}; \gamma_k)$ has $l_k\geq r_k$ columns, while $\Phi_k(x_{1:k}; \alpha_k)$ only has $r_k$ columns. Such ``oversampling'' \cite{halko2011finding} is standard in randomized linear algebra algorithms for capturing the range of a matrix effectively. Then, as in \eqref{eq:CDEs_reduced_rough}, left sketches $S_k$'s are applied to further reduce $\bar \Phi_k$'s to $\tilde \Phi_k$'s. As mentioned previously, $\tilde \Phi_k$ resembles the right-hand side of \eqref{eq:CDEs_reduced_rough}, though the $\tilde \Phi_k$'s need to be further processed by \textsc{Trimming}. An important remark here is that unlike the right sketch functions $T_2,\ldots T_d$, the left sketch functions $S_2, \ldots, S_{d - 1}$ are constructed sequentially, i.e., $S_k$ is obtained by contracting a small block $s_{k}$ with $S_{k - 1}$; hence, it is a sequential contraction of $s_1, \ldots, s_k$. Such a design is necessary as is shown in \textsc{SystemForming}. Another remark is that Algorithm~\ref{alg:1-2} is presented in a modular fashion for the sake of clarity. In fact, many computations in Algorithm~\ref{alg:1-2} can be re-used by leveraging the fact that $S_k$ is obtained from the contraction of $S_{k-1}$ and $s_k$. Hence, $\tilde \Phi_k$ can be obtained recursively from $\tilde \Phi_{k-1}$.

\begin{algorithm}
\caption{\textsc{Sketching}.}
\label{alg:1-2}
\begin{algorithmic}
\REQUIRE $p$, $T_2, \ldots, T_{d}$, and $s_1, \ldots, s_{d - 1}$ as given in Algorithm \ref{alg:1}.
\FOR{$k = 1$ to $d - 1$}
    \STATE Right sketching: define $\bar{\Phi}_k \colon [n_1] \times \cdots \times [n_k] \times [\ell_k] \to \bR$ as
    \begin{equation*}
        \bar{\Phi}_k(x_{1 : k}, \gamma_{k}) = \sum_{x_{k+1} = 1}^{n_{k+1}} \cdots \sum_{x_{d} = 1}^{n_d} p(x_{1 : k}, x_{k + 1 : d}) T_{k+1}(x_{k+1 : d}, \gamma_{k}).
    \end{equation*}
    \IF {$k > 1$}
        \STATE Left sketching: define $\tilde{\Phi}_k \colon [m_{k - 1}] \times [n_k] \times [\ell_k] \to \bR$ as
        \begin{equation*}
            \tilde{\Phi}_{k}(\beta_{k-1}, x_{k}, \gamma_{k})
            =
            \sum_{x_1 = 1}^{n_1} \cdots \sum_{x_{k - 1} = 1}^{n_{k - 1}} S_{k-1}(\beta_{k-1}, x_{1 : k - 1}) \bar{\Phi}_k(x_{1 : k - 1}, x_{k}, \gamma_{k}).
        \end{equation*}
        \STATE Compute sketch function $S_k \colon [m_k] \times [n_1] \times \cdots \times [n_k] \to \bR$ for the next iteration: 
        \begin{equation*}
            S_k(\beta_k, x_{1 : k}) 
            = 
            \sum_{\beta_{k - 1} = 1}^{m_{k - 1}} s_k(\beta_{k}, x_{k}, \beta_{k - 1}) S_{k - 1}(\beta_{k - 1}, x_{1 : k - 1}).
        \end{equation*}
    \ELSE
        \STATE Define
        \begin{equation*}
            \tilde \Phi_1(x_1,\gamma_1) = \bar \Phi_1(x_1,\gamma_1).
        \end{equation*}
        \STATE Define sketch function
        \begin{equation*}
            S_1(\beta_1, x_1) = s_1(\beta_1, x_1).
        \end{equation*}
    \ENDIF
\ENDFOR
\STATE Left sketching: define $\tilde{\Phi}_{d} \colon [m_{d - 1}] \times [n_d] \to \bR$ as
\begin{equation*}
    \tilde \Phi_{d}(\beta_{d - 1}, x_{d})
    =
    \sum_{x_1 = 1}^{n_1} \cdots \sum_{x_{d - 1} = 1}^{n_{d - 1}} S_{d-1}(\beta_{d-1}, x_{1 : d - 1}) p(x_{1 : d - 1}, x_{d}).
\end{equation*}
\RETURN $\tilde{\Phi}_1, \ldots,  \tilde{\Phi}_{d}$.
\end{algorithmic}
\end{algorithm}

\textsc{Trimming} takes the outputs $\tilde \Phi_1,\ldots,\tilde \Phi_{d-1}$ of \textsc{Sketching} and further process them to have the appropriate rank of the underlying TT using the SVD. This procedure is illustrated in Step 2 in Figure~\ref{fig:algorithm}. It should be noted that this procedure is not necessary if for any $k$, $l_k = r_k$. In this case, one should directly let $B_k = \tilde \Phi_k$ for each $k$.

\begin{algorithm}
\caption{\textsc{Trimming}.}
\label{alg:1-3}
\begin{algorithmic}
\REQUIRE $\tilde{\Phi}_1, \ldots, \tilde{\Phi}_{d}$ from Algorithm \ref{alg:1-2}. 
\REQUIRE Target ranks $r_1, \ldots, r_{d - 1}$ as given in Algorithm \ref{alg:1}.
\FOR{$k = 1$ to $d - 1$}
    \IF {$k = 1$}
        \STATE Compute the first $r_1$ left singular vectors of $\tilde{\Phi}_1(x_1; \gamma_1)$ and define $B_1 \colon [n_1] \times [r_1] \to \bR$ so that these singular vectors are the columns of $B_1(x_1; \alpha_1)$.
    \ELSE
        \STATE Compute the first $r_k$ left singular vectors of $\tilde{\Phi}_{k}(\beta_{k-1}, x_{k}; \gamma_{k})$ and define $B_{k} \colon [m_{k-1}] \times [n_k] \times [r_k] \to \bR$ so that these singular vectors are the columns of $B_{k}(\beta_{k - 1}, x_{k}; \alpha_{k})$.
    \ENDIF 
\ENDFOR
\STATE Let $B_d(\beta_{d-1}, x_d) = \tilde \Phi_d(\beta_{d-1}, x_d)$.
\RETURN $B_1, \ldots, B_{d}$.
\end{algorithmic}
\end{algorithm}

Finally, \textsc{SystemForming} forms the coefficient matrices to solve for $G_1,\ldots,G_d$ from the output $B_1,\ldots, B_{d - 1}$ of \textsc{Trimming} by contracting $s_1, \ldots, s_{d - 1}$ with them, which results in $A_1, \ldots, A_{d - 1}$, respectively, as in Step 3 of Figure~\ref{fig:algorithm}. The matrices $A_1,\ldots,A_{d-1}$ play the role of the coefficient matrices appearing on the left-hand side of \eqref{eq:CDEs_reduced_rough}. As we see in the algorithm, the fact that the sketch functions $S_1,\cdots,S_{d-1}$ are obtained by successive contractions of $s_1,\ldots,s_{d-1}$ allows $A_k$ to be constructed from $B_{k}$. We stress that this is not merely for the sake of efficient computation. In fact, it is important for the correctness of the algorithm, as illustrated in the proof of recovery for Markov models in Section~\ref{section:discrete markov} below.

\begin{algorithm}
\caption{\textsc{SystemForming}.}
\label{alg:1-4}
\begin{algorithmic}
\REQUIRE $B_1, \ldots, B_{d - 1}$ from Algorithm \ref{alg:1-2}. 
\REQUIRE $s_1, \ldots, s_{d - 1}$ as given in Algorithm \ref{alg:1}.
\FOR{$k = 1$ to $d - 1$}
    \IF {$k = 1$}
        \STATE Compute $A_1 \colon [m_1] \times [r_1] \to \bR$:
        \begin{equation*}
            A_1(\beta_1, \alpha_{1}) 
            = \sum_{x_1 = 1}^{n_1} s_1(\beta_1, x_1) B_1(x_1, \alpha_{1}).
        \end{equation*}
    \ELSE
        \STATE Compute $A_k \colon [m_k] \times [r_k] \to \bR$:
        \begin{equation*}
            A_{k}(\beta_{k}, \alpha_{k}) 
            = 
            \sum_{x_{k} = 1}^{n_k} \sum_{\beta_{k - 1} = 1}^{m_{k - 1}} s_{k}(\beta_k, x_{k}, \beta_{k-1}) B_{k}(\beta_{k-1}, x_{k}, \alpha_{k}).
        \end{equation*}
    \ENDIF 
\ENDFOR
\RETURN $A_1,\ldots, A_{d - 1}$.
\end{algorithmic}
\end{algorithm}

\subsection{Complexity}
As noted earlier, we are practically interested in the case where $p$ is an empirical distribution $\hat{p}$ constructed from $N$ i.i.d.\ samples from an underlying density $p^\star$. In such a case, $\hat{p}$ is $N$-sparse. The high-dimensional integrals within TT-RS can be efficiently computed in this case. To see this, suppose that the input $p$ of Algorithm \ref{alg:1} is $N$-sparse, and let $n = \max_{1 \le k \le d} n_k$, $m = \max_{1 \le k \le d - 1} m_k$, $\ell = \max_{1 \le k \le d - 1} \ell_k$, and $r = \max_{1 \le k \le d - 1} r_k$. Note that the complexity of \textsc{Sketching} is $O(m \ell N d)$ since each $\tilde{\Phi}_k$ can be computed in $O(m \ell N)$ time. \textsc{Trimming} requires $O(m n l^2 d)$ operations as each $B_k$ is computed using SVD in $O(m n l^2)$ times. Also, \textsc{SystemForming} is achieved in $O(m^2 n r d)$ time. Lastly, the equations \eqref{eq:alg-CDEs} can be solved in $O(m n r^3 d)$ time. In summary, the total computational cost of TT-RS with $N$-sparse input is 
\begin{equation*}
    O(m \ell N d) + O(m n l^2 d) + O(m^2 n r d) + O(m n r^3 d).
\end{equation*}
Note that this cost is linear in both $n$ and the dimension $d$ of the distribution.

\begin{remark}
    The term ``recursive sketching'' in the name TT-RS is due to the sequential contraction of the left sketch functions $s_1, \ldots, s_{d - 1}$. We remark that it is possible to design an algorithm without such ``recursiveness'', which we call TT-Sketch (TT-S); see Appendix \ref{app:non-recursive sketch} for the details.
\end{remark}

\section{Conditions for exact recovery for TT-RS}\label{section:exact recovery}
The main purpose of this section is to provide sufficient conditions for when TT-RS can recover an underlying TT if the input function $p$ admits a representation by a tensor train. In particular, the following theorem provides a guideline for choosing the sketch functions in TT-RS.
\begin{theorem}
    \label{thm:1}
    Assume the rank (in exact arithmetic) of the $k$-th unfolding matrix of $p$ is $r_k$ for each $k = 1, \ldots, d - 1$. Suppose $T_2, \ldots, T_{d}$ and $s_{1}, \ldots, s_{d-1}$ of Algorithm \ref{alg:1} satisfy the following.
    \begin{itemize}
        \item[(i)] $\bar{\Phi}_k(x_{1:k};\gamma_k)$ and $p(x_{1:k}; x_{k + 1 : d})$ have the same column space for $k = 1, \ldots, d - 1$.
        \item[(ii)] $\tilde{\Phi}_{k}(\beta_{k-1}, x_{k}; \gamma_{k})$ and $\bar{\Phi}_k(x_{1:k}; \gamma_{k})$ have the same row space for $k = 2, \ldots, d - 1$.
        \item[(iii)] $A_k(\beta_k; \alpha_k)$ is rank-$r_k$ for $k = 1, \ldots, d - 1$.
    \end{itemize}
    Then, each equation of \eqref{eq:alg-CDEs} has a unique solution, and the solutions $G_1, \ldots, G_{d}$ are cores of $p$.
\end{theorem}
We first present a lemma showing that \textsc{Sketching} and \textsc{Trimming} give rise to the right-hand side of \eqref{eq:CDEs_reduced_rough} for determining the cores of $p$. 
\begin{lemma}
\label{lem:order-change}
    Under the assumptions of Theorem \ref{thm:1}, consider the results $B_1, \ldots, B_{d - 1}$ produced by Algorithms \ref{alg:1-2} and \ref{alg:1-3}. The column space of $B_1(x_1; \alpha_1)$ is the same as that of the first unfolding matrix of $p$. Also, for each $ k = 2,\ldots ,  d -1$, there exists a $\Phi_k \colon [n_1] \times \cdots \times [n_k] \times [r_k] \to \bR$ such that the column space of $\Phi_k(x_{1 : k}; \alpha_k)$ is the same as that of the $k$-th unfolding matrix and 
    \begin{equation}
        \label{eq:lemma}
        B_{k}(\beta_{k - 1}, x_k, \alpha_k) = \sum_{x_1 = 1}^{n_1} \cdots \sum_{x_{k - 1} = 1}^{n_{k - 1}} S_{k - 1}(\beta_{k - 1}, x_{1 : k - 1}) \Phi_{k}(x_{1 : k - 1}, x_{k}, \alpha_k).
    \end{equation}
\end{lemma}
\begin{proof}
    (i) implies that $\tilde{\Phi}_1(x_1; \gamma_1) = \bar{\Phi}_1(x_1; \gamma_1)$ and $p(x_1; x_{2 : d})$ have the same $r_1$-dimensional column space, which is the same as the column space of $B_1(x_1; \alpha_1)$ by the definition of $B_1$.
    
    For $k = 2, \ldots,  d - 1$, (i) and (ii) imply that $\tilde \Phi_{k}(\beta_{k-1}, x_{k}; \gamma_{k})$ is still rank-$r_k$. Since the columns of $B_k(\beta_{k - 1}, x_{k}; \alpha_{k})$ are the first $r_k$ left singular vectors of $\tilde{\Phi}_{k}(\beta_{k - 1}, x_{k}; \gamma_{k})$, we may write
    \begin{equation*}
        B_{k}(\beta_{k - 1}, x_{k}; \alpha_{k})
        =
        \sum_{\gamma_{k} = 1}^{\ell_k} \tilde \Phi_{k}(\beta_{k - 1}, x_{k}; \gamma_{k}) q_{k + 1}(\gamma_{k}; \alpha_{k})
    \end{equation*}
    for some $q_{k + 1} \colon [\ell_k] \times [r_k] \to \bR$; here, the column space of $q_{k + 1}(\gamma_k; \alpha_k)$ is the same as the row space of $\tilde{\Phi}_{k}(\beta_{k - 1}, x_{k}; \gamma_{k})$. Now, we define $\Phi_k \colon [n_1] \times \cdots [n_k] \times [r_k] \to \bR$ by 
    \begin{equation}
    \label{eq:p_k}
        \Phi_k(x_{1 : k}, \alpha_k) = \sum_{\gamma_{k} = 1}^{\ell_k} \bar{\Phi}_k(x_{1 : k}, \gamma_{k}) q_{k + 1}(\gamma_{k}, \alpha_{k}).
    \end{equation}
    Next, we observe that \eqref{eq:lemma} holds since
    \begin{equation*}
        \begin{split}
            B_{k}(\beta_{k - 1}, x_{k}, \alpha_{k})
            & =
            \sum_{\gamma_{k} = 1}^{\ell_k} \tilde{\Phi}_{k}(\beta_{k - 1}, x_{k},\gamma_{k}) q_{k + 1}(\gamma_{k}, \alpha_{k}) \\
            & =
            \sum_{\gamma_{k} = 1}^{\ell_k}
            \sum_{x_1 = 1}^{n_1} \cdots \sum_{x_{k - 1} = 1}^{n_{k - 1}} S_{k-1}(\beta_{k-1}, x_{1 : k - 1}) \bar{\Phi}_k(x_{1 : k - 1}, x_{k}, \gamma_{k}) q_{k + 1}(\gamma_{k}, \alpha_{k}).
        \end{split}
    \end{equation*}
    We claim that the column space of $\Phi_k(x_{1 : k}; \alpha_k)$ is the same as that of the $k$-th unfolding matrix. Indeed, due to \eqref{eq:p_k}, the column space of $\Phi_k(x_{1 : k}; \alpha_k)$ is contained in that of $\bar{\Phi}_{k}(x_{1 : k}; \gamma_{k})$, which is the column space of the $k$-th unfolding matrix because of (i). Now, it suffices to prove that $\Phi_k(x_{1 : k}; \alpha_k)$ has full column rank. This is true because the column space of $q_{k + 1}(\gamma_k; \alpha_k)$ is the same as the row space of $\tilde{\Phi}_{k}(\beta_{k - 1}, x_{k}; \gamma_{k})$ by construction, which is equivalent to the row space of $\bar{\Phi}_{k}(x_{1 : k}; \gamma_{k})$ due to (ii).
\end{proof}
In Lemma~\ref{lem:order-change}, we showed that \textsc{Sketching} and \textsc{Trimming} give the right-hand sides of \eqref{eq:CDEs_reduced_rough} (i.e., $B_k$ in \eqref{eq:lemma}), without forming the exponentially-sized $\Phi_k$ explicitly. Lastly, by combining \textsc{Sketching} and \textsc{Trimming} with \textsc{SystemForming}, we have a well-defined system of equations for determining $G_1,\ldots,G_d$, as in Algorithm~\ref{alg:1}. This is shown in the following proof for Theorem~\ref{thm:1}. 

\begin{proof}[Proof of Theorem \ref{thm:1}]
    Due to Lemma \ref{lem:order-change}, there exists $\Phi_2, \ldots, \Phi_{d - 1}$ such that \eqref{eq:lemma} holds; also, letting $\Phi_1 = B_1$, we have shown that $\Phi_k(x_{1 : k}; \alpha_k)$ and the $k$-th unfolding matrix have the same column space for $k = 1 , \ldots,  d - 1$. Hence, we can consider CDEs \eqref{eq:CDEs} formed by $\Phi_1, \ldots, \Phi_{d - 1}$. First, we verify that the equations in \eqref{eq:alg-CDEs} are implied by \eqref{eq:CDEs}, obtained by applying sketch functions to both sides of \eqref{eq:CDEs}. The first equation $G_1 = \Phi_1$ is the same in both \eqref{eq:alg-CDEs} and \eqref{eq:CDEs}. For $k = 2 ,\ldots , d - 1$, if we apply $S_{k - 1}$ to both side of the $k$-th equation of \eqref{eq:CDEs}, then 
    \begin{equation}
    \label{eq:tmp}
        \begin{split}
            & \sum_{x_1 = 1}^{n_1} \cdots \sum_{x_{k - 1} = 1}^{n_{k - 1}} S_{k - 1}(\beta_{k - 1}, x_{1 : k - 1}) \sum_{\alpha_{k - 1} = 1}^{r_{k - 1}} \Phi_{k - 1}(x_{1 : k - 1}, \alpha_{k - 1}) G_{k}(\alpha_{k - 1}, x_{k}, \alpha_{k}) \\
            & = \sum_{x_1 = 1}^{n_1} \cdots \sum_{x_{k - 1} = 1}^{n_{k - 1}} S_{k - 1}(\beta_{k - 1}, x_{1 : k - 1}) \Phi_k(x_{1 : k}, \alpha_{k}).
        \end{split}
    \end{equation}
    Note that the right-hand side of \eqref{eq:tmp} is simply $B_{k}(\beta_{k - 1}, x_k, \alpha_k)$, which is the right-hand side of the $k$-th equation of \eqref{eq:alg-CDEs}. We now want to show that the coefficient matrix on the left-hand side of \eqref{eq:tmp} is the coefficient matrix $A_{k - 1}(\beta_{k - 1}, \alpha_{k - 1})$ of the $k$-th equation of \eqref{eq:alg-CDEs}, that is, we want to prove for $k = 2 , \ldots,  d - 1$,
    \begin{equation}
    \label{eq:coef}
        \sum_{x_1 = 1}^{n_1} \cdots \sum_{x_{k - 1} = 1}^{n_{k - 1}} S_{k - 1}(\beta_{k - 1}, x_{1 : k - 1})  \Phi_{k - 1}(x_{1 : k - 1}, \alpha_{k - 1})
        =
        A_{k - 1}(\beta_{k - 1}, \alpha_{k - 1}),
    \end{equation}
    This is implied by Algorithm \ref{alg:1-4}. To see this, for $k = 2$, note that \eqref{eq:coef} amounts to 
    \begin{equation*}
        \sum_{x_1 = 1}^{n_1} s_1(\beta_1, x_1) B_1(x_1, \alpha_{1})
        = A_1(\beta_1, \alpha_{1}), 
    \end{equation*}
    which follows immediately from Algorithm \ref{alg:1-4}. For $2 < k \le d - 1$, \eqref{eq:coef} holds because 
    \begin{equation*}
    \begin{split}
        &\quad \sum_{x_1 = 1}^{n_1} \cdots \sum_{x_{k - 1} = 1}^{n_{k - 1}} S_{k - 1}(\beta_{k - 1}, x_{1 : k - 1})  \Phi_{k - 1}(x_{1 : k - 1}, \alpha_{k - 1}) \\
        & =
        \sum_{x_1 = 1}^{n_1} \cdots \sum_{x_{k - 1} = 1}^{n_{k - 1}} \sum_{\beta_{k - 2} = 1}^{m_{k - 2}} s_{k-1}(\beta_{k - 1}, x_{k - 1}, \beta_{k - 2})S_{k - 2}(\beta_{k - 2}, x_{1 : k - 2}) \Phi_{k - 1}(x_{1 : k-1}, \alpha_{k - 1}) \\
        & = \sum_{x_{k - 1} = 1}^{n_{k - 1}} \sum_{\beta_{k - 2} = 1}^{m_{k - 2}} s_{k-1}(\beta_{k - 1}, x_{k - 1}, \beta_{k - 2}) B_{k - 1}(\beta_{k - 2}, x_{k - 1}, \alpha_{k - 1}) \\
        & = A_{k - 1}(\beta_{k - 1}, \alpha_{k - 1}),
    \end{split}
    \end{equation*}
    where the first equality holds since $S_{k - 1}$ is a contraction of $s_{k - 1}$ and $S_{k - 2}$, the second equality holds because of \eqref{eq:lemma}, and the last equality is given in Algorithm \ref{alg:1-4}. Hence, we have shown that for $k = 2, \ldots, d - 1$, the $k$-th equation of \eqref{eq:alg-CDEs} is indeed obtained by applying $S_{k - 1}$ to both sides of the $k$-th equation of \eqref{eq:CDEs}. Similarly, the last equation of \eqref{eq:alg-CDEs} is obtained by applying $S_{d - 1}$ to both sides of the last equation of \eqref{eq:CDEs}.
    From this it is clear that solutions $G_1, \ldots, G_{d}$ of \eqref{eq:CDEs} formed by $\Phi_1, \ldots, \Phi_{d - 1}$ satisfy \eqref{eq:alg-CDEs}. Now, we use condition (iii) in Theorem \ref{thm:1}; this means that the coefficient matrices $A_1, \ldots, A_{d - 1}$ have full column rank, and thus each equation of \eqref{eq:alg-CDEs} must have a unique solution. Therefore, a unique set of solutions $G_1, \ldots, G_{d}$ of \eqref{eq:CDEs} formed by $\Phi_1, \ldots, \Phi_{d - 1}$ discussed in Proposition \ref{prop:CDEs} gives rise to a unique set of solutions of \eqref{eq:alg-CDEs}. Additionally, as in Proposition \ref{prop:CDEs}, $G_1, \ldots, G_d$ give a TT representation of $p$.
\end{proof}

\section{Application of TT-RS to Markov model}\label{section:discrete markov}

In this section, we demonstrate how model assumptions on $p$ can guide the choice of sketch functions $T_2, \ldots, T_{d}$ and $s_1, \ldots, s_{d - 1}$ to guarantee that the conditions (i)-(iii) of Theorem \ref{thm:1} are satisfied. More precisely, we show that for Markov models, suitable sketch functions exist, and moreover, we give an explicit construction. In Section~\ref{section:exact recovery Markov} we prove that the sketch functions we construct satisfy the requisite conditions. When working with an empirical distribution $\hat p$ which is constructed based on i.i.d.\ samples from some underlying density $p^\star$, TT-RS requires obtaining $B_1,\ldots,B_d$, $A_1,\ldots,A_{d-1}$ by taking expectations over the empirical distribution. Though the variance can be large, in Section~\ref{section:stable recovery Markov}, we show that under certain natural conditions, our choice of sketch functions does not suffer from the ``curse of dimensionality'' when estimating the cores from the empirical distribution. 

Throughout this section, we assume that the input $p$ of TT-RS (Algorithm \ref{alg:1}) is a Markov model, that is, $p$ is a probability density function and satisfies 
\begin{equation}
    \label{eq:markov-condition}
    p(x_1, \ldots, x_d) = p(x_1) p(x_2 | x_1) \cdots p(x_{d} | x_{d - 1}).
\end{equation}
Here, by abuse of notation, for any $i < j$, we denote the marginal density of $(x_i, \ldots, x_j)$ as $p(x_i, \ldots, x_j)$. Depending on the situation, we also use
\begin{equation}
    (\mathcal{M}_{\mathcal{S}}p)(x_\mathcal{S}):=p(x_\mathcal{S}),\quad \mathcal{S}\subset [d]
\end{equation}
to denote the marginalization of $p$ to the variables given by the index set $\mathcal{S}$, which is a $\vert \mathcal{S} \vert$-dimensional function. Also, $p(x_i, \ldots, x_{j} | x_{k})$ denotes the conditional density of $(x_i, \ldots, x_{j})$ given $x_{k}$. For a Markov model $p$, the conditional probabilities $p(x_2 | x_1), \ldots, p(x_d | x_{d - 1})$ in \eqref{eq:markov-condition} are referred to as the transition kernels.

\subsection{Choice of sketch}\label{section:markov sketch} We start with the following simple lemma that shows the low-dimensional nature of the column and row spaces of the unfolding matrices.

\begin{lemma}
\label{lem:markov}
    Suppose $p$ is a Markov model. For any $i \le k < j$,
    \begin{itemize}
        \item[(i)] $p(x_{i : k}; x_{k+1 : j})$ and $p(x_{i : k}; x_{k+1})$ have the same column spaces,
        \item[(ii)] $p(x_{i : k}; x_{k+1 : j})$ and $p(x_{k}; x_{k+1 : j})$ have the same row spaces.
    \end{itemize}
\end{lemma}
\begin{proof}
Since $x_{i : k} \perp x_{k+2 : j} ~|~ x_{k+1}$ (conditional independence), we have that 
\begin{equation*}
    p(x_{i : k}; x_{k+1 : j}) = p(x_{i : k} | x_{k+1}) p(x_{k+2 : j} | x_{k+1}) p(x_{k+1}),
\end{equation*}
which implies that the column space of $p(x_{i : k}; x_{k+1 : j})$ is not affected by $x_{k+2 : j}$. For the same reason, $x_{i : k-1} \perp x_{k+1 : j} ~|~ x_{k}$ implies
\begin{equation*}
    p(x_{i : k}; x_{k+1 : j}) = p(x_{i : k-1} | x_{k}) p(x_{k+1 : j} | x_{k}) p(x_{k}),
\end{equation*}
and hence the row space of $p(x_{i : k}; x_{k+1 : j})$ is not affected by $x_{i : k-1}$.
\end{proof}

An immediate consequence of Lemma \ref{lem:markov} is that each unfolding matrix $p(x_{1 : k}; x_{k + 1 : d})$ may be replaced by $p(x_{1 : k}; x_{k + 1})$ if our main focus is the column space. This motivates a specific choice of sketch functions for a Markov model. For each $k = 1, \ldots,  d - 1$, let $\ell_k = n_{k + 1}$ and define \begin{equation}\label{eq:sketchT_def}
    T_{k+1}(x_{k+1 : d}, \gamma_{k}) = I_{k + 1}(x_{k+1};\gamma_{k}),
\end{equation}
where $I_{k + 1} \colon [n_{k + 1}] \times [n_{k + 1}] \to \mathbb{R}$ such that $I_{k + 1}(x_{k+1};\gamma_{k})$ is the identity matrix. This choice of $T_{k+1}$ yields 
\begin{equation}\label{eq:barpdef_markov}
    \bar{\Phi}_{k}(x_{1 : k}, \gamma_k) = (\mathcal{M}_{1:k+1} p)(x_{1 : k}, \gamma_k).
\end{equation}
In other words, contracting $T_{k+1}$ with the $k$-th unfolding matrix amounts to marginalizing out variables $x_{k + 2}, \ldots, x_{d}$.

Similarly, we let $m_k = n_k$ for each $k = 1, \ldots, d - 1$, and define
\begin{equation}\label{eq:sketchS_def}
    s_{1}(\beta_{1}, x_{1}) = I_1(\beta_{1} ; x_{1}), 
    \quad
    s_{k}(\beta_{k}, x_{k}, \beta_{k-1}) = I_k(\beta_{k}; x_{k}),
\end{equation}
where $I_1 \colon [n_1] \times [n_1] \to \mathbb{R}$ is defined so that $I_1(\beta_1; x_1)$ is the identity matrix, which gives rise to $S_k(\beta_k, x_{1 : k}) = I_k(\beta_{k}, x_{k})$ and 
\begin{equation}\label{eq:tildep_markov}
    \tilde{\Phi}_{1}
    = \mathcal{M}_{\{1,2\}} p, \quad \tilde{\Phi}_{k}
    = \mathcal{M}_{\{k-1,k,k+1\}} p,\ 2 \leq k\leq d-1,
    \quad
    \tilde \Phi_{d} = \mathcal{M}_{\{d-1,d\}} p.
\end{equation}
Again, this choice of left sketch functions leads to $S_{k - 1}$ that marginalizes out variables $x_1, \ldots, x_{k - 2}$. 

In summary, with these sketch functions, \textsc{Sketching} outputs marginals of $p$. Now, it is obvious that Algorithm \ref{alg:1-3} and \ref{alg:1-4} can be done  efficiently; it just performs an SVD on these small marginal matrices and computes both $A_1 = B_1$ and 
\begin{equation*}
    A_{k}(x_{k}, \alpha_{k}) 
    = 
    \sum_{x_{k - 1} = 1}^{n_{k - 1}} B_{k}(x_{k-1}, x_{k}, \alpha_{k})
\end{equation*}
for $k \ge 2$.

\begin{remark}
For the situation where $p$ is a function of continuous variables, in  Appendix~\ref{section:continuous Markov} we discuss how to adapt TT-RS-C (Algorithm~\ref{alg:2}) to the Markov case.
\end{remark}

\subsection{Exact recovery for Markov models}\label{section:exact recovery Markov} In this subsection, we prove that if we use TT-RS (Algorithm \ref{alg:1}) in conjunction with the sketches defined in \eqref{eq:sketchT_def} and \eqref{eq:sketchS_def}, then the resulting algorithm enjoys the exact recovery property. Using Theorem \ref{thm:1}, it suffices to check the choice of sketch functions mentioned in the previous subsection satisfies (i)-(iii) of Theorem \ref{thm:1}. 

\begin{theorem}
    \label{thm:discrete-markov}
    Let $p$ be a discrete Markov model such that the rank (in exact arithmetic) of the $k$-th unfolding matrix of $p$ is $r_k$ for each $k = 1, \ldots, d - 1$. With right and left sketches in \eqref{eq:sketchT_def}, \eqref{eq:sketchS_def}, Algorithm \ref{alg:1} returns $G_1, \ldots, G_{d}$ as cores of $p$.
\end{theorem}
\begin{proof}
It suffices to check that (i)-(iii) of Theorem \ref{thm:1} are satisfied. As noted earlier, for each $k = 1, \ldots, d - 1$, \eqref{eq:barpdef_markov} holds.
Hence, $\bar{\Phi}_k(x_{1 : k}; \gamma_{k})$ and $p(x_{1 : k}; x_{k+1 : d})$ have the same column space by Lemma \ref{lem:markov}. Thus, (i) of Theorem \ref{thm:1} holds. Similarly, for each $k = 2, \ldots, d - 1$, \eqref{eq:tildep_markov} holds, hence, $\tilde{\Phi}_{k}(\beta_{k-1}, x_{k}; \gamma_{k})$ and $\bar{\Phi}_k(x_{1 : k}; \gamma_{k})$ have the same row space. Thus, (ii) of Theorem \ref{thm:1} holds.

Lastly, we claim $A_k(x_{k}; \alpha_{k})$ is rank-$r_k$ for all $k = 1, \ldots, d - 1$ (condition (iii) of Theorem~\ref{thm:1}). Clearly, $A_1(x_1; \alpha_1) = B_1(x_1; \alpha_1)$ is rank-$r_1$ by definition. For $k = 2, \ldots, d - 1$, by definition of $B_k$, we can find $q_{k+1} \colon [n_{k+1}] \times [r_k] \to \bR$ such that the column space of $q_{k+1}(x_{k+1}; \alpha_{k})$ is the same as the row space of $p(x_{k-1}, x_{k}; x_{k+1})$ and 
\begin{equation*}
    B_k(x_{k-1}, x_{k}, \alpha_{k}) = \sum_{x_{k + 1} = 1}^{n_{k + 1}} p(x_{k-1}, x_{k}, x_{k+1}) q_{k+1}(x_{k+1}, \alpha_{k}).
\end{equation*}
Hence, 
\begin{equation*}
    \begin{split}
        A_k(x_k, \alpha_k) 
        &= \sum_{x_{k-1} = 1}^{n_{k - 1}} B_k(x_{k-1}, x_{k}, \alpha_{k}) \\
        & = \sum_{x_{k-1} = 1}^{n_{k - 1}} \sum_{x_{k + 1} = 1}^{n_{k + 1}} p(x_{k-1}, x_{k}, x_{k+1}) q_{k+1}(x_{k+1}, \alpha_{k}). \\
        & = \sum_{x_{k + 1} = 1}^{n_{k + 1}} \left(\sum_{x_{k-1} = 1}^{n_{k - 1}} p(x_{k-1}, x_{k}, x_{k+1})\right) q_{k+1}(x_{k+1}, \alpha_{k}) \\
        & = \sum_{x_{k + 1} = 1}^{n_{k + 1}} p(x_{k}, x_{k+1}) q_{k+1}(x_{k+1}, \alpha_{k}) .
    \end{split}
\end{equation*}
By Lemma \ref{lem:markov}, $p(x_{k-1}, x_{k}; x_{k+1})$ and $p(x_{k}; x_{k+1})$ have the same row space. Therefore, the column space of $q_{k+1}(x_{k+1}; \alpha_{k})$ is the same as the row space of $p(x_{k}; x_{k+1})$, where both are rank-$r_k$. Thus, $A_k(x_{k}; \alpha_{k})$ must be rank-$r_k$ by construction. 
\end{proof}

\subsection{Stable estimation for Markov models}\label{section:stable recovery Markov}

In this section, we present an informal result regarding the stability of the TT-RS algorithm when an empirical distribution $\hat p$ is provided as input instead of the true density $p^{\star}$. The precise statement of the theorem is deferred to Appendix \ref{app:pert_res}. If $\hat{p}$ is taken as the input of Algorithm \ref{alg:1}, the results $\tilde \Phi_1, \ldots, \tilde \Phi_d$ of \textsc{Sketching} have certain variances that get propagated to the final output $G_1,\ldots,G_d$ via the coefficient matrices $A_1, \ldots, A_{d-1}, B_1, \ldots, B_d$. The variances of $\tilde \Phi_1,\ldots, \tilde \Phi_d$ depend critically on the choice of sketch functions. In what follows, we show that the sketches \eqref{eq:sketchT_def} and \eqref{eq:sketchS_def} give a nearly dimension-independent error when estimating the tensor cores if $p^\star$ is a Markov model satisfying the following natural condition.
\begin{condition}
    \label{cond:transition}
    The transition kernels $p^\star(x_2 | x_1), \ldots, p^\star(x_d | x_{d - 1})$ are independent of $d$.
\end{condition}
\begin{theorem}[Informal statement of Theorem \ref{thm:markov-estimation}]
\label{thm:markov-estimation-informal}
Suppose $p^\star$ is a discrete Markov model that satisfies Condition \ref{cond:transition} and admits a TT-representation with rank $(r_1, \ldots, r_{d-1})$. Consider an empirical distribution $\hat p$ constructed based on N i.i.d.\ samples from $p^\star$. Let $\hat G_1,\ldots, \hat G_d$ and $G^\star_1,\ldots, G^\star_d$ be the results of TT-RS with $\hat p$ and $p^\star$ as input, respectively. Then, with high probability,
\begin{equation}
    \frac{\mathrm{dist}(\hat G_k, G_k^\star)}{\normi{G_k^\star}}\leq  O\left(\frac{\sqrt{\log(d)}}{\sqrt{N}}\right) \quad \forall k = 1, \ldots, d,
\end{equation}
where the hidden constant in the ``big-$O$'' notation does not depend on the dimensionality $d$, $\normi{\cdot}$ is some appropriate norm, and  $\mathrm{dist}(\cdot, \cdot)$ is a suitable measure of distance between cores.
\end{theorem}
In Theorem \ref{thm:markov-estimation-informal}, the errors in the cores show $\sqrt{\log(d)}$-dependence which grows very slowly in $d$; the term $\sqrt{\log(d)}$ is a consequence of the union bound required to derive a probabilistic bound on $d$ objects (the cores) simultaneously. We remark, however, that near dimension-independent errors in the pairs $(G_k^\star, \hat{G}_k)$ do not necessarily imply such an error in approximating $p^\star$ by $\hat{G}_1\circ \cdots \circ \hat{G}_d$, the results of TT-RS with $\hat{p}$ as input. Instead, we can derive an error that scales almost linearly in $d$, thereby avoiding the curse of dimensionality. The precise statement is deferred to Appendix \ref{app:pert_res}; here, we provide an informal statement summarizing this result. 
\begin{cor}[Informal statement of Theorem \ref{thm:p_for_markov}]
    \label{cor:total}
    In the setting of Theorem \ref{thm:markov-estimation-informal}, with high probability, 
    \begin{equation*}
        \frac{\|\hat{G}_1\circ \dots \circ \hat{G}_d - G_1^\star \circ \dots \circ G_d^\star \|_\infty}{\normi{G_1^\star}
        \dots \normi{G_d^\star}} \le O\left(\frac{d \sqrt{\log(d)}}{\sqrt{N}}\right),
    \end{equation*}
    where $\|\cdot\|_\infty$ denotes the largest absolute value of the entries of a tensor.
\end{cor}
In Section \ref{section:numerical}, we verify from the experiments that such $d \sqrt{\log(d)}$-dependence of the error indeed suggests near-linear dependence on the dimensionality $d$.

\begin{remark}
    Extensive numerical experiments suggest that Theorem \ref{thm:markov-estimation-informal} and Corollary \ref{thm:p_for_markov} are valid for a broad class of Markov models that may not necessarily satisfy Condition \ref{cond:transition}. See Remark \ref{rmk:constants} for details.
\end{remark}

\subsection{Higher-order Markov models} 
\label{subsec:higher-order}
We conclude this section with a brief discussion on higher-order Markov models. For $m \in \bN$, we call $p$ an order-$m$ Markov model if it is a density and satisfies
\begin{equation*}
    p(x_1, \ldots, x_d) = p(x_1, \ldots, x_m) p(x_{m + 1} | x_1, \ldots, x_m) \cdots p(x_{d} | x_{d - m}, \ldots, x_{d - 1}).
\end{equation*}
What we have presented so far, i.e., the case $m = 1$, can be generalized to any $m \in \bN$ by a suitable replacement of the sketch functions $T_2, \ldots, T_{d}$ and $s_1, s_2, \ldots, s_{d - 1}$. Recall that the sketch functions for the case where $m = 1$ are chosen based on Lemma \ref{lem:markov}, which can be properly generalized to any order-$m$ Markov model. For instance, we can say that $p(x_{i : k}; x_{k+1 : k + j})$ and $p(x_{i : k}; x_{k + 1 : k + m})$ have the same column space for any $j\geq m$. Based on this generalization, the choice of the sketch functions for general $m \in \bN$ is straightforward: they are chosen such that  
\begin{align*}
    \bar{\Phi}_{k} 
    & = \mathcal{M}_{1:(k + m) \wedge d} p, \\
    \tilde{\Phi}_{k}
    & =
    \mathcal{M}_{k-1:(k + m) \wedge d}  p.
\end{align*}
    In particular, using such $\tilde \Phi_k$'s as the input to \textsc{Trimming} and subsequently \textsc{SystemForming}, we obtain an algorithm for a discrete order-$m$ Markov density.

\section{Numerical experiments}\label{section:numerical}

In this section, we illustrate the performance of our algorithm with concrete examples. More specifically, given i.i.d.\ samples of some ground truth density $p^\star$, we construct an empirical density $\hat{p}$ and apply TT-RS (or TT-RS-C) to it to obtain cores $G_1, \ldots, G_d$ such that $p^\star \approx  G_1 \circ \cdots \circ G_{d}=:q$.

\subsection{Ginzburg-Landau distribution}
We consider the following probability density defined on $[a, b]^d$:
\begin{equation*}
    p_{GL}(x_1, \ldots, x_d)
    \propto
    \exp\left(- \beta \sum_{k = 0}^{d} \left(\frac{\lambda}{2} \left(\frac{x_k - x_{k + 1}}{h}\right)^2 + \frac{1}{4 \lambda} (x_k^2 - 1)^2\right)\right),
\end{equation*}
where $x_0 = x_{d+1} = 0$. This is the Boltzmann distribution of a Ginzburg-Landau potential, which is classically used to model phase transitions in physics and also more recently as a test case in generative modeling \cite{gabrie2021adaptive}. Throughout the section, we fix $[a, b] = [-4, 4]$ and $\beta = \lambda = h = 1$.

First, we consider a discretized version of $p$. To discretize $p$, we choose $n$ uniform grid points of $[a, b]$, that is, $\mathcal{Z} = \left\{a+\frac{i}{n-1}(b-a)\right\}^{n-1}_{i=0}$, and define a discretized density $p_{D} \colon [n]^d \to \bR$ as
\begin{equation*}
    p_{D} := [p_{GL}(x_1, \ldots, x_{d})]_{(x_1,\ldots,x_d)\in \mathcal{Z}^d}.
\end{equation*}
Hence, $p_{D}$ is essentially a multi-dimensional array of size $n^d$. Notice that $p_{GL}$ is a Markov model, hence so is $p_{D}$. We obtain $N$ i.i.d.\ samples from $p_{D}$ using a Gibbs sampler and construct an empirical density based on these samples, which form the empirical measure $\hat{p}_{D}$. We apply TT-RS with sketches~\eqref{eq:sketchT_def} and \eqref{eq:sketchS_def} to $\hat{p}_{D}$ and let $q_{D}:=G_1\circ\cdots\circ G_d$ be the contraction of the cores obtained by the algorithm. We compute the following relative $\ell^2$ error:
\begin{equation*}
    \frac{\|p_D - q_D\|_2}{\|p_D\|_2},
\end{equation*}
where $\|f\|_2^2 := \sum_{x_1 = 1}^{n} \cdots \sum_{x_d = 1}^{n} f(x_1, \ldots, x_d)^2$ for any $f \colon [n]^d \to \bR$. We see in Figure~\ref{fig:plots_discrete}(A) that the error decreases with rate $O\left(\frac{1}{\sqrt N}\right)$ as sample size $N$ increases when we fix $d$. Furthermore, when we fix $N$ and let $d$ grow, we see a linear growth in the error (Figure~\ref{fig:plots_discrete}(B)). 

\begin{figure}[ht]
    \centering
    \subfloat[\centering]{\includegraphics[width=0.48\textwidth]{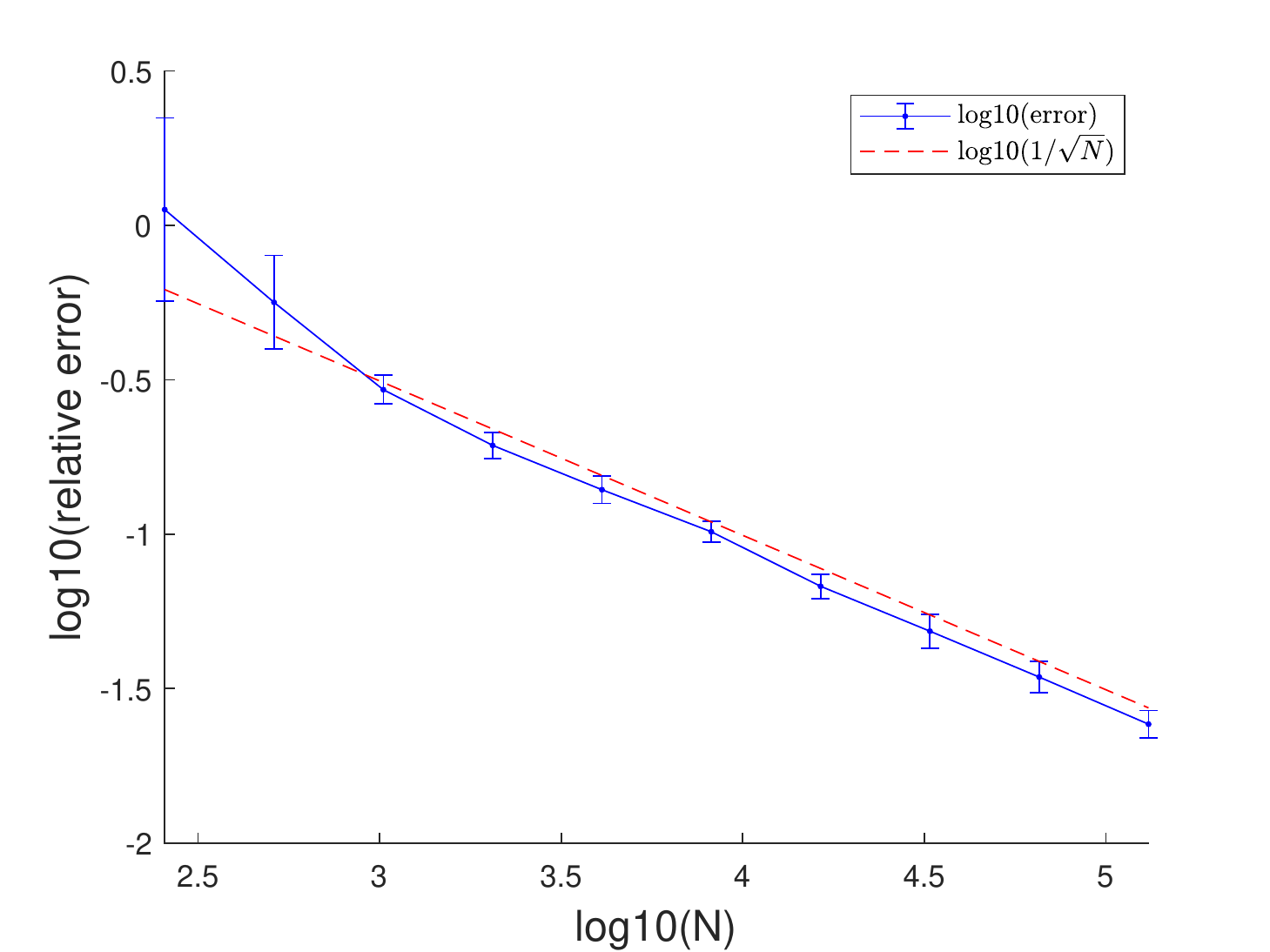}}%
    \hspace{12pt}
    \subfloat[\centering]{\includegraphics[width=0.48\textwidth]{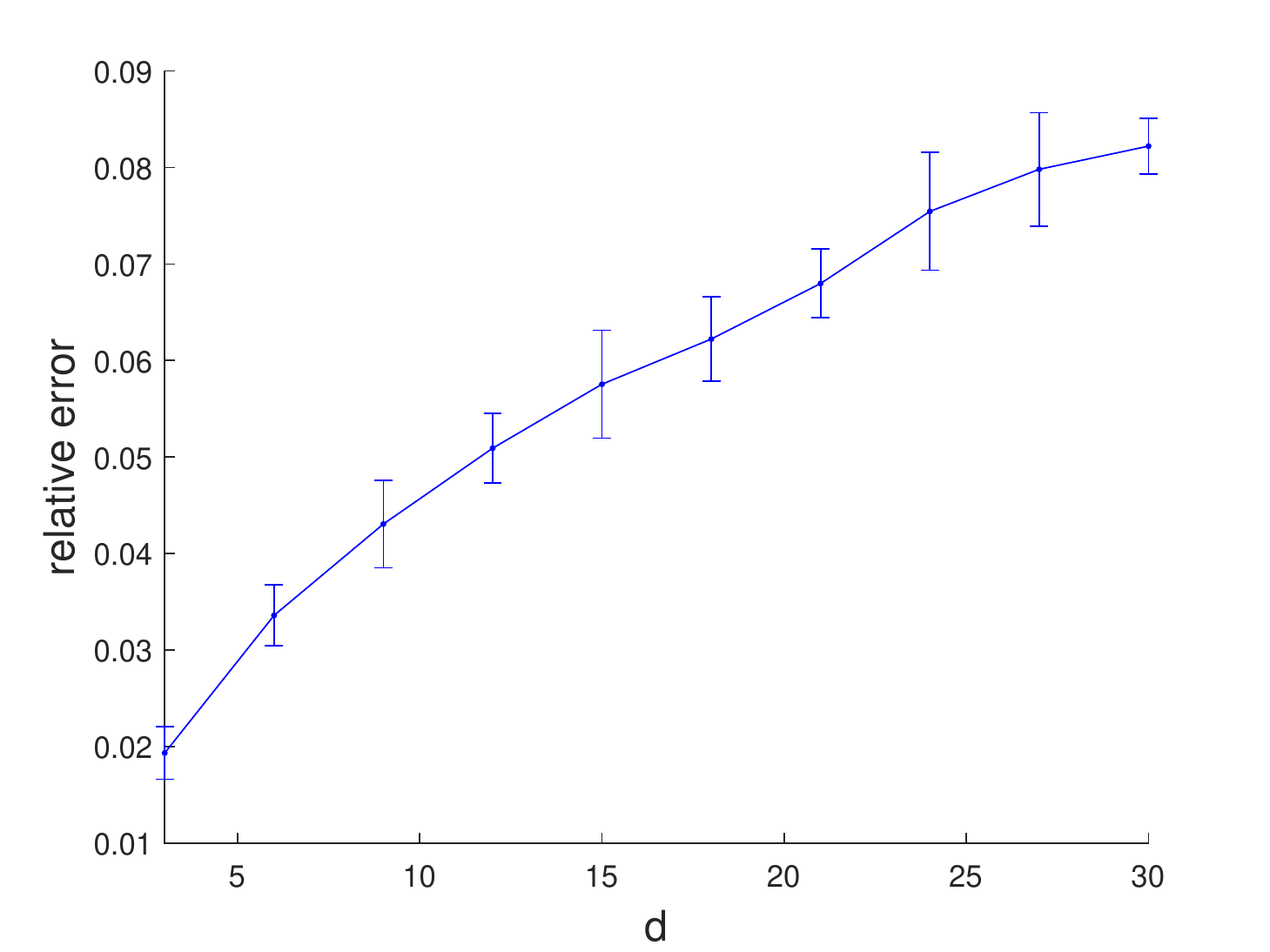}}%
    \caption{Relative $\ell_2$ errors for the discretized case. In (A), we fix $d = 8$ and change the sample size $N \in \{2^{8}, 2^{9}, \ldots, 2^{17}\}$. In (B), we fix the sample size $N = 50000$ and change $d \in \{3, 6, \ldots, 27, 30\}$. In both cases, we use the fixed number of grid points $n = 9$, and TT-RS (Algorithm \ref{alg:1}) is applied with $r_1 = \cdots = r_d = 3$. Each error bar is centered at the average of 20 realizations, with the standard deviation as its vertical length.}
    \label{fig:plots_discrete}
\end{figure}

Next, we repeat the same procedure with a continuous density $p_{GL}$. Now we obtain $N$ i.i.d.\ samples from $p_{GL}$ using the Metropolis-Hastings algorithm and construct an empirical density $\hat{p}$ based on them. Then, we apply Algorithm \ref{alg:markov-continuous} to $\hat{p}$, where we choose the basis functions $\phi_1, \ldots, \phi_M$ as Fourier basis functions on $[a, b]$. Recall that contraction of the resulting cores gives a function $q$ such that 
\begin{equation*}
    q(x_1, \ldots, x_{d}) = \sum_{j_{1} = 1}^{M} \cdots \sum_{j_{d} = 1}^{M} \left(\sum_{\alpha_1 = 1}^{r_1} \cdots \sum_{\alpha_{d - 1} = 1}^{r_{d - 1}} g_1(j_1, \alpha_1) \cdots g_{d}(\alpha_{d-1}, j_{d})\right) \phi_{j_1}(x_1) \cdots \phi_{j_{d}}(x_{d}).
\end{equation*}
Then, we compute the relative $L^2$ error: 
\begin{equation*}
    \text{err}_{t} = \frac{\|p_{GL} - q\|_2}{\|p_{GL}\|_2},
\end{equation*}
where $\|f\|_2^2 = \int_{[a, b]^d} f(x_1, \ldots, x_d)^2 \, d x_1 \cdots d x_d$ for any $f$ defined on $[a, b]^d$. Since $q$ is an element of the function space $\Pi_{M} := \{\phi_{j_1} \otimes \cdots \otimes \phi_{j_d} : j_1, \ldots, j_d \in [M] \}$, we may decompose this $L^2$ error as follows using the orthogonality:
\begin{equation*}
    \text{err}_{t}^2 = \underbrace{\left(\frac{\|p_{GL} - p_{A}\|_2}{\|p_{GL}\|_2}\right)^2}_{=: \text{err}_{a}^2} + \underbrace{\left(\frac{\|p_{A} - q\|_2}{\|p_{GL}\|_2}\right)^2}_{=: \text{err}_{e}^2},
\end{equation*}
where 
\begin{align*}
    p_{A}(x_1, \ldots, x_{d}) & = \sum_{j_{1} = 1}^{M} \cdots \sum_{j_{d} = 1}^{M} \nu(j_1, \ldots, j_d) \phi_{j_1}(x_1) \cdots \phi_{j_{d}}(x_{d}), \\
    \nu(j_1, \ldots, j_d) & = \int p_A(x_1, \ldots, x_d) \phi_{j_1}(x_1) \cdots \phi_{j_{d}}(x_{d}) \, d x_1 \cdots d x_d.
\end{align*}
In other words, $p_{A}$ is the approximation of $p$ within the space $\Pi_{M}$ spanned by the product basis, thus $\text{err}_{a}$ represents an approximation error. Accordingly, we can think of $\text{err}_{e}$ as an estimation error, where the resulting $g_1, \ldots, g_d$ can be thought of as approximate cores of $\nu$. All the integrals above are approximated using the Gauss-Legendre quadrature rule with 50 nodes. 

The resulting $L^2$ errors are shown in Table \ref{tab:results_continuous}. As $M$ increases, the approximation error $\text{err}_{a}$ decreases quickly to 0. On the other hand, larger $M$ leads to a larger estimation error $\text{err}_{e}$ as one needs to estimate a larger size of coefficient tensor $\nu$.

\begin{table}[ht]
\centering
\scalebox{0.85}{
\begin{tabular}{|c | c c c | c c c | c c c|} 
\hline
&\multicolumn{3}{c|}{$d = 5$} & \multicolumn{3}{c|}{$d = 10$} & \multicolumn{3}{c|}{$d = 15$}\\
\hline
$M$ & $\text{err}_a$ & $\text{err}_e$ & $\text{err}_t$ & $\text{err}_a$ & $\text{err}_e$ & $\text{err}_t$ & $\text{err}_a$ & $\text{err}_e$ & $\text{err}_t$ \\
\hline
7 & 0.2693 & 0.0202 (0.0023) & 0.2701 & 0.4144 & 0.0392 (0.0032) & 0.4163 & 0.5104 & 0.0582 (0.0041) & 0.5138 \\
9 & 0.1617 & 0.0334 (0.0018) & 0.1651 & 0.2511 & 0.0621 (0.0027) & 0.2587 & 0.3142 & 0.0908 (0.0041) & 0.3270 \\
11 & 0.0867 & 0.0411 (0.0016) & 0.0960 & 0.1365 & 0.0754 (0.0024) & 0.1559 & 0.1722 & 0.1100 (0.0039) & 0.2044 \\
13 & 0.0400	& 0.0433 (0.0015) & 0.0589 & 0.0655 & 0.0802 (0.0023) & 0.1036 & 0.0837 & 0.1186 (0.0039) & 0.1451 \\
15 & 0.0201 & 0.0446 (0.0015) & 0.0489 & 0.0330 & 0.0833 (0.0023) & 0.0896 & 0.0421 & 0.1246 (0.0038) & 0.1315 \\
\hline
\end{tabular}}
\caption{$L^2$ errors for Ginzburg-Landau Gibbs measure in the continuous case. Sample size $N$ is fixed to $10^6$, and Algorithm \ref{alg:markov-continuous} is applied with $r_1 = \cdots = r_d = 3$. Each $\text{err}_e$ is averaged over 20 realizations, and the number in the parentheses denotes the standard deviation.}
\label{tab:results_continuous}
\end{table}
\subsection{Ising-type model} For our next example we consider the following slight generalization of the one-dimensional Ising model. Define $p \colon \{\pm 1\}^d \to \bR$ by
\begin{equation}
    \label{eq:ising}
    p_I(x_1, \ldots, x_d) \propto \exp\left(-\beta \sum_{i, j = 1}^{d} J_{i j} x_i x_j\right),
\end{equation}
where $\beta > 0$ and the interaction $J_{i j}$ is given by 
\begin{equation*}
    J_{i j} 
    =
    \begin{cases}
        - (1 + |i - j|)^{- 1} & |i - j| \le 2 \\
        0 & \text{otherwise}.
    \end{cases}
\end{equation*}
From this, we can easily see that $p_I$ is an order-$2$ Markov model. For such a model we can apply TT-RS with the sketch functions described in Section \ref{subsec:higher-order}. 

As in the previous section, we obtain $N$ i.i.d.\ samples from $p_I$ using a Gibbs sampler and construct an empirical density based on them, $\hat{p}$. Then, we apply Algorithm TT-RS, with the sketch functions in Section \ref{section:markov sketch} and with the modifications outlined in Section \ref{subsec:higher-order}, to obtain the contraction of the resulting cores $q_1$ and $q_2$, respectively. Then, we compare the two relative $\ell^2$ errors:
\begin{equation*}
    \text{err}_{1} = \frac{\|p_I - q_1\|_2}{\|p_I\|_2} 
    \quad \text{and} \quad  
    \text{err}_{2} = \frac{\|p_I - q_2\|_2}{\|p_I\|_2}.
\end{equation*}

The errors are plotted in Figure \ref{fig:plots_ising}, in which the dashed curves denote the result $\text{err}_{1}$ of TT-RS with sketches as in Section~\ref{section:markov sketch} and the solid curves correspond to $\text{err}_{2}$ from TT-RS with the sketches as in Section \ref{subsec:higher-order}. Clearly, as expected, the error is smaller when using the sketches from Section \ref{subsec:higher-order}.
\begin{figure}[ht]
    \centering
    \subfloat[\centering]{\includegraphics[width=0.49\textwidth]{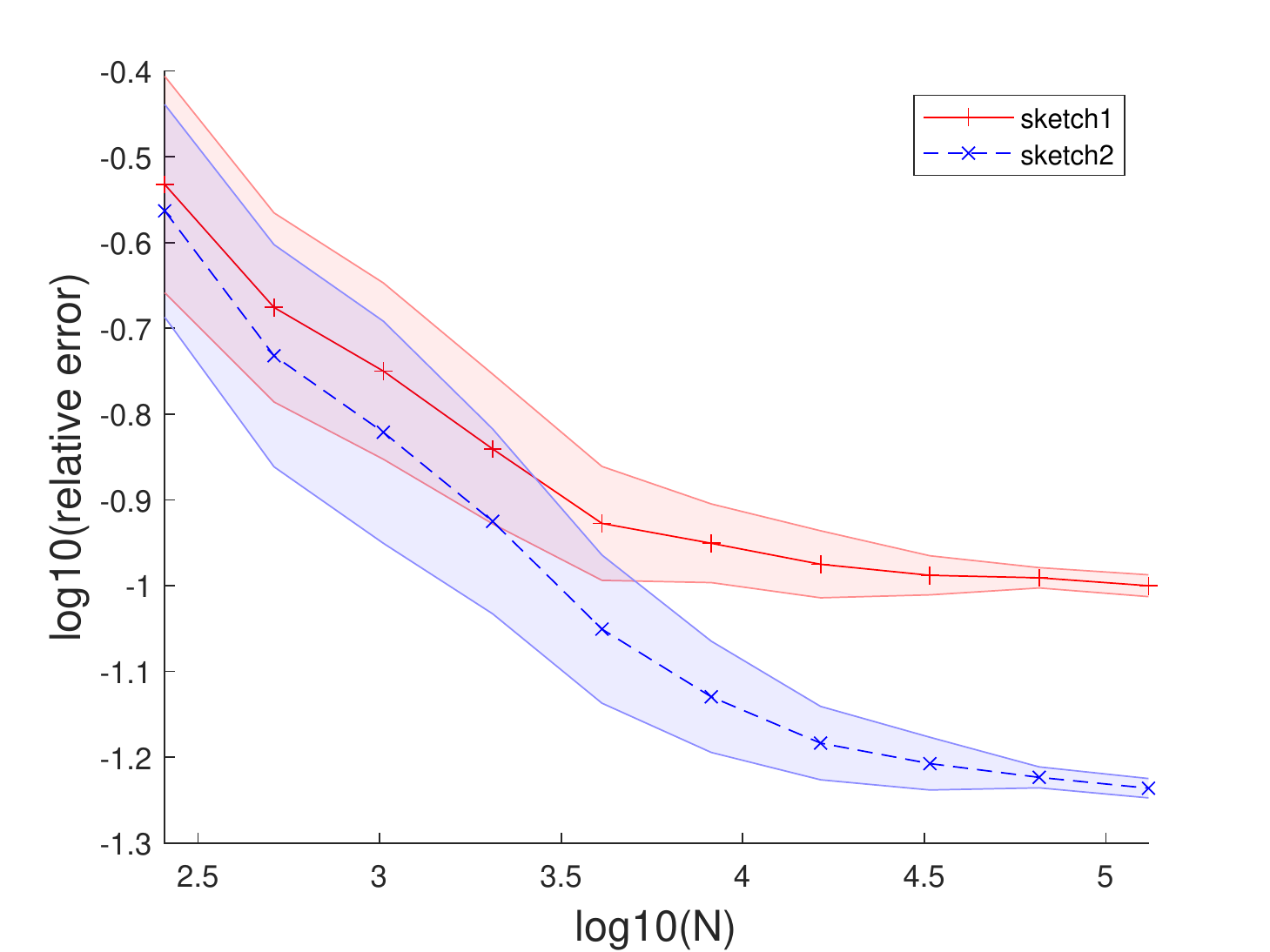}}%
    \subfloat[\centering]{\includegraphics[width=0.49\textwidth]{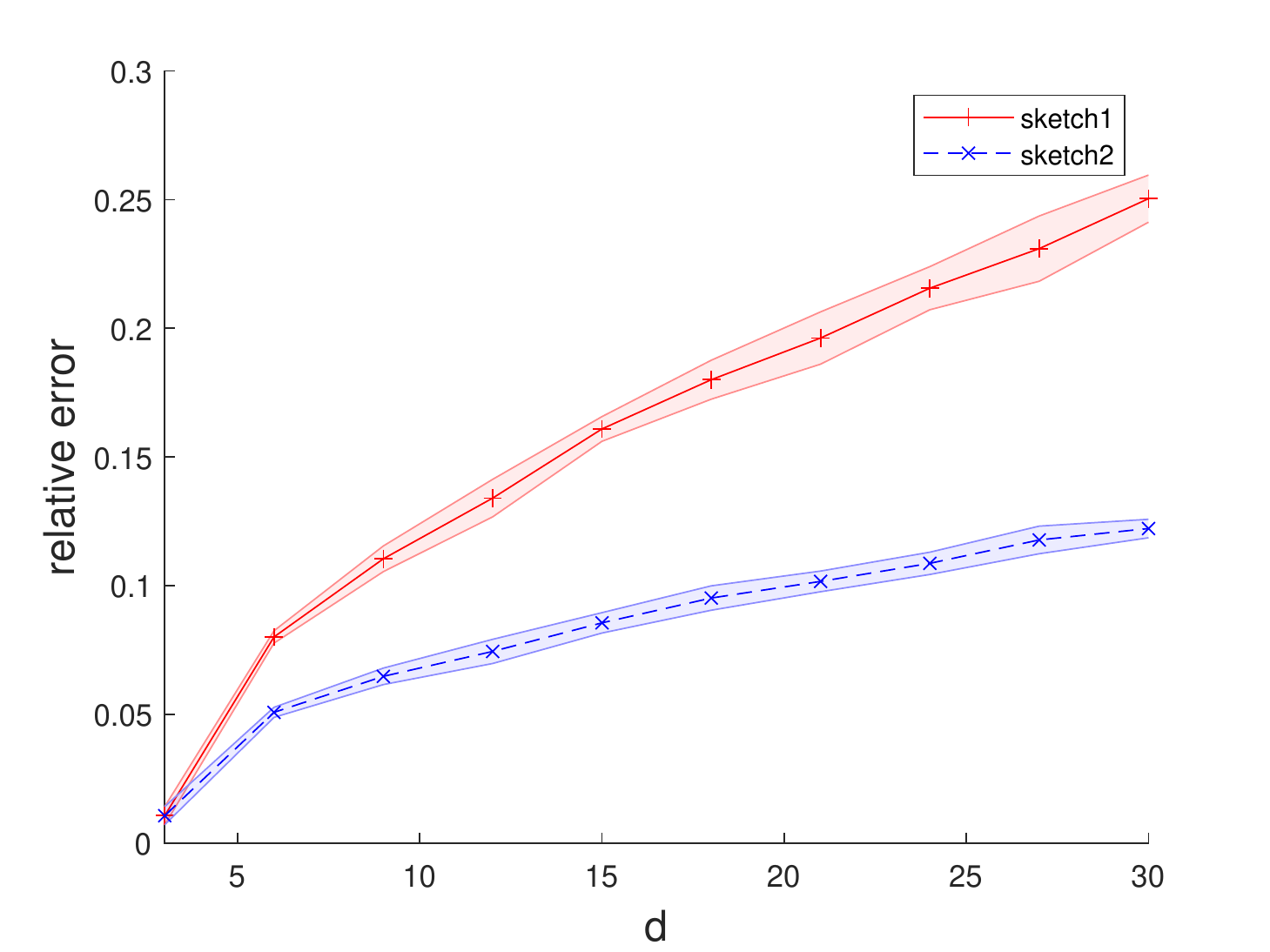}}%
    \caption{Relative $\ell_2$ errors for the order-2 Ising model. In (A), we fix $d = 8$ and change the sample size $N \in \{2^{8}, 2^9, \ldots, 2^{17}\}$. In (B), we fix the sample size $N = 50000$ and change $d \in \{3, 6, \ldots, 27, 30\}$. In both cases, we use $\beta = 0.4$, and TT-RS (Algorithm~\ref{alg:1}) is applied with $(r_1, \ldots, r_d) = (2, 3, \ldots, 3, 2)$. Errors are shown as shaded regions, where both solid and dashed curves connect the averages of errors from 20 realizations, with the standard deviation as the vertical width.}
    \label{fig:plots_ising}
\end{figure}

Lastly, we repeat the same procedure for $p_I$ where $x_k \in \{-2, -1, 0, 1, 2\}$ for  $k=1,\ldots,d$ in \eqref{eq:ising}. The results are shown in Figure \ref{fig:plots_ising5}, which demonstrates that TT-RS with appropriate sketching yields small error in the case of higher-order Markov distributions.

\begin{figure}[ht]
    \centering
    \subfloat[\centering]{\includegraphics[width=0.49\textwidth]{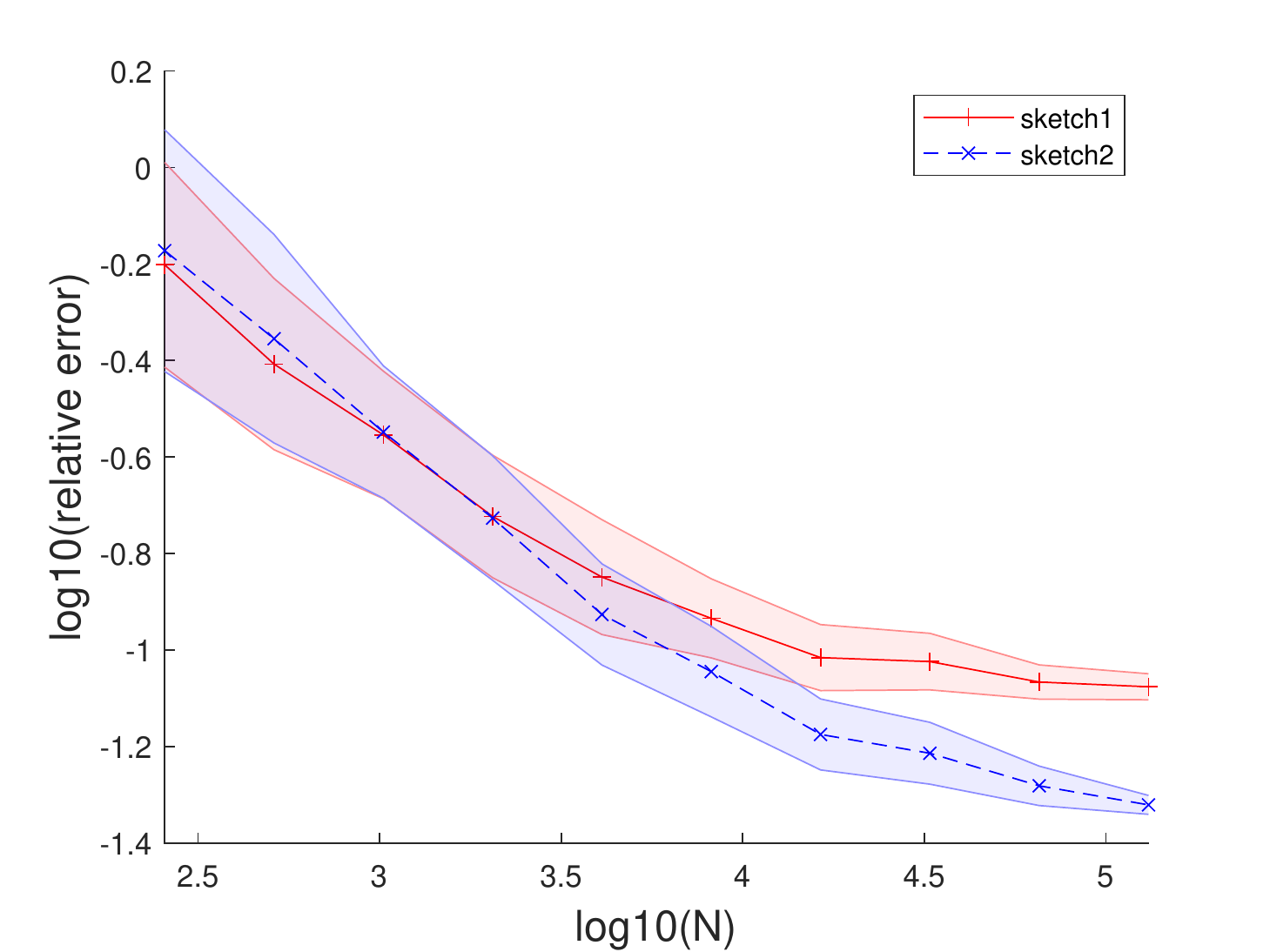}}%
    \subfloat[\centering]{\includegraphics[width=0.49\textwidth]{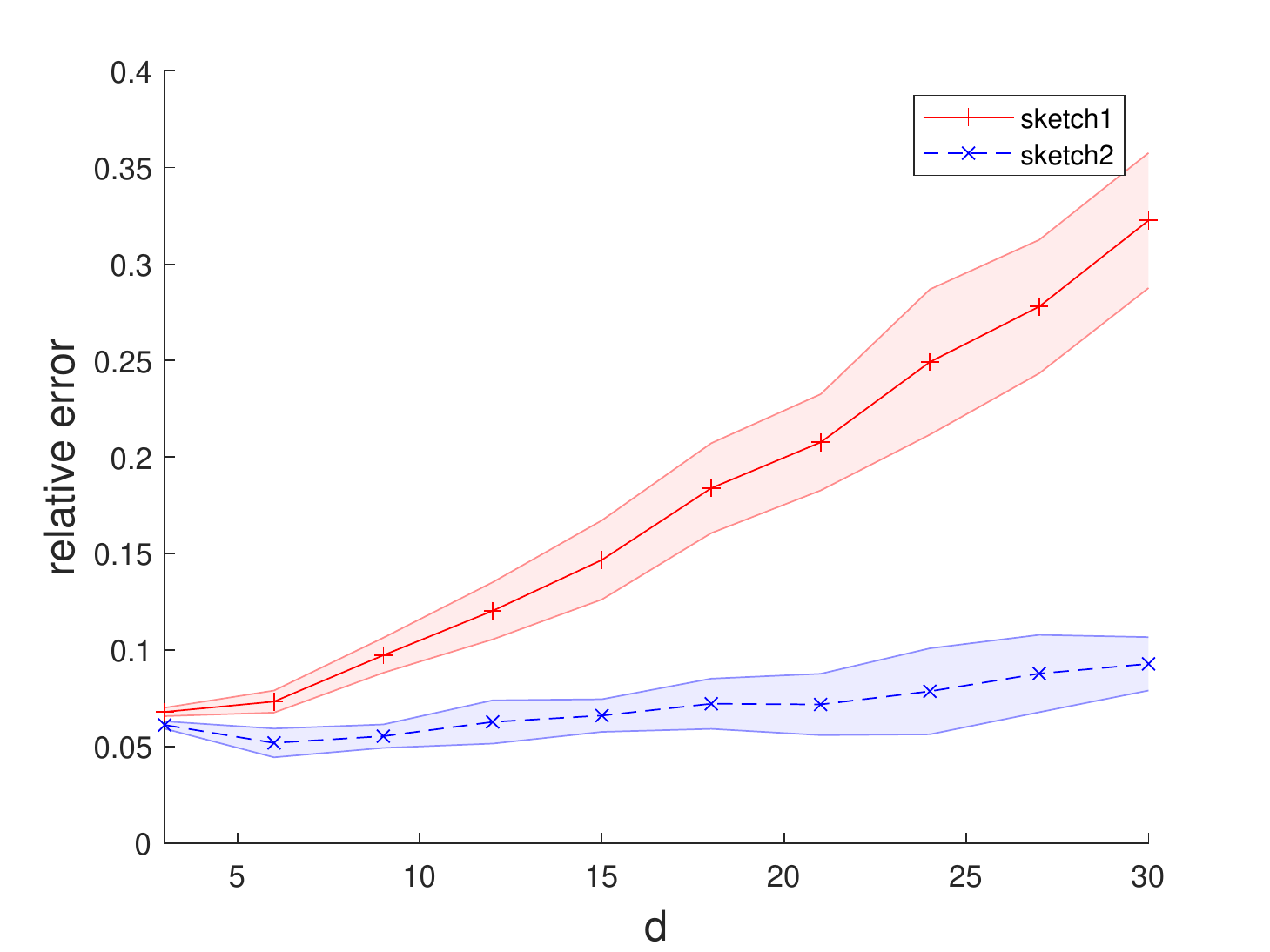}}%
    \caption{Relative $\ell_2$ errors for the order-2 Ising model on $\{-2, -1, 0, 1, 2\}^d$. In (A), we fix $d = 8$ and change the sample size $N \in \{2^8, 2^9, \ldots, 2^{17}\}$. In (B), we fix the sample size $N = 50000$ and change $d \in \{3, 6, \ldots, 27, 30\}$. In both cases, we use $\beta = 0.2$, and TT-RS (Algorithm \ref{alg:1}) is applied with $(r_1, \ldots, r_d) = (2, 3, \ldots, 3, 2)$. Each error bar is centered at the average of 20 realizations, with the standard deviation as its vertical length.}
    \label{fig:plots_ising5}
\end{figure}

\section{Conclusion}\label{section:conclusion}
We have described an algorithm TT-RS which obtains a tensor train representation of a probability density from a collection of its samples. This is done by formulating a sequence of equations, one for each core, which can be solved independently. Additionally, in order to reduce the variance in the coefficient matrices of these equations (which are constructed from the empirical distribution) sketching is required. For Markov (and higher-order Markov) models we give explicit constructions of suitable sketches and provide guarantees on the accuracy of the resulting algorithm.

Lastly, we briefly mention several possible extensions for future research. First, we can apply TT-RS to more complicated models such as hidden Markov models. The ideas that we discussed based on (higher-order) Markov models can be generalized to various models by specifying concrete sketch functions for such models. More generally, future research could focus on adapting TT-RS to tree tensor networks, aiming at generalizing TT-RS to distributions with more general graphical structure. By designing sketch functions for a broader class of models, one can bring TT-RS closer to a wide range of applications and we leave this as future work.

\section*{Acknowledgments}
The work of YH and YK was partially supported by the National Science Foundation under Award No.\ DMS-211563 and the Department of Energy under Award No.\ DE-SC0022232. The work of ML was partially supported by the National Science Foundation under Award No.\ 1903031. The Flatiron Institute is a division of the Simons Foundation.

\appendix

\section{Validity of solving CDEs}\label{app:CDE}
In this section, we give the proof of Proposition~\ref{prop:CDEs}.
\begin{proof}[Proof of Proposition~\ref{prop:CDEs}]
For $k = 2,\ldots,  d - 1$, consider the $k$-th equation in \eqref{eq:CDEs}:
\begin{equation}\label{eq:thm-CDEs}
    \sum_{\alpha_{k - 1} = 1}^{r_{k - 1}} \Phi_{k - 1}(x_{1 : k - 1}; \alpha_{k - 1}) G_{k}(\alpha_{k - 1}; x_{k}, \alpha_{k}) = \Phi_k(x_{1 : k - 1}; x_{k}, \alpha_{k}).
\end{equation}
By definition, $\Phi_{k - 1}(x_{1 : k - 1}; \alpha_{k - 1})$ is the left factor in an exact low-rank factorization of $p$, so $\Phi_{k-1}$  has full column rank and the uniqueness of solutions is guaranteed. To prove a solution also exists, we need to show that columns of $\Phi_k(x_{1 : k - 1}; x_{k}, \alpha_{k})$ are contained within the column space of $\Phi_{k - 1}(x_{1 : k - 1}; \alpha_{k - 1})$.

By the definition of $\Phi_{k - 1}$ and $\Phi_k$, we know there exists $\Psi_{k} \colon [r_{k - 1}] \times [n_k] \times \cdots \times [n_d] \to \bR$ and $\Psi_{k + 1} \colon [r_{k}] \times [n_{k + 1}] \times \cdots \times [n_d] \to \bR$ such that 
\begin{align}\label{eq:thm-factorize_p}
    p(x_{1 : k - 1}; x_{k : d}) & = \sum_{\alpha_{k-1} = 1}^{r_{k - 1}} \Phi_{k-1}(x_{1 : k - 1}; \alpha_{k-1}) \Psi_{k}(\alpha_{k-1}; x_{k : d}), \cr
    p(x_{1 : k}; x_{k + 1 : d}) & = \sum_{\alpha_{k} = 1}^{r_{k}} \Phi_{k}(x_{1 : k}; \alpha_{k}) \Psi_{k + 1}(\alpha_{k}; x_{k + 1 : d}).
\end{align} 
Note that these are rank-$r_{k - 1}$ and rank-$r_k$ decomposition of the $(k - 1)$-th and $k$-th unfolding matrices, respectively. Defining $t_{k + 1} \colon [n_{k + 1}] \times \cdots \times [n_d] \times [r_{k}] \to \bR$ so that $t_{k + 1}(x_{k + 1 : d}; \alpha_{k})$ is the pseudoinverse of $\Psi_{k + 1}(\alpha_{k}; x_{k + 1 : d})$, we obtain
\begin{equation*}
    \Phi_k(x_{1 : k}; \alpha_{k}) = \sum_{x_{k+1} = 1}^{n_{k+1}} \cdots \sum_{x_{d} = 1}^{n_d} p(x_{1 : k}; x_{k + 1 : d})  t_{k+1}(x_{k+1 : d}; \alpha_{k}).
\end{equation*}
Then, one can easily verify that the $k$-th equation~\eqref{eq:thm-CDEs} holds if we let
\begin{equation*}
    G_k(\alpha_{k-1}, x_{k}, \alpha_k) = \sum_{x_{k + 1} = 1}^{n_{k + 1}} \cdots \sum_{x_{d} = 1}^{n_d} \Psi_{k}(\alpha_{k-1}, x_{k}, x_{k + 1 : d}) t_{k + 1}(x_{k+1 : d}, \alpha_{k})
\end{equation*}
along with~\eqref{eq:thm-factorize_p}. Thus, we have not only proved the existence of solutions to the $k$-th equation, but also obtained the exact form of the solution in terms of $\Psi_{k}$ and $t_{k + 1}$. 

Similarly, we can show that the equation in~\eqref{eq:thm-CDEs} for $G_d$ is well-defined. By construction, it then follows that \eqref{eq:prop-core} holds.
\end{proof}

\section{Non-recursive TT-RS: TT-Sketch (TT-S)}\label{app:non-recursive sketch}

\subsection{Role of recursive left sketches and possibility of non-recursive sketches}\label{section:TTS}
In this subsection, we discuss the importance of forming the recursive right sketches $S_1,\ldots,S_{d-1}$ from $s_1,\ldots,s_d$, noting that for $T_2,\ldots,T_d$ there is no such need. The requirement of ``recursiveness'' in the construction of the $S_k$'s is a consequence of the \textsc{Trimming} step, which introduces an arbitrary projection matrix in the factorization of the $k$-th unfolding of $p$. To see this, consider first the case without $\textsc{Trimming}$, i.e., using sketches $T_{k+1}$ with $l_k = r_k$. Then one can use $\bar \Phi_k$ (defined in Algorithm~\ref{alg:1-2}) in the CDEs \eqref{eq:CDEs}, i.e. solve
\begin{equation}\label{eq:CDEs bar}
     \sum_{\alpha_{k - 1} = 1}^{r_{k - 1}} \bar \Phi_{k - 1}(x_{1 : k - 1}; \alpha_{k - 1}) G_{k}(\alpha_{k - 1}; x_{k}, \alpha_{k}) = \bar \Phi_k(x_{1:k-1}; x_{k}, \alpha_{k}) 
\end{equation}
if each $\bar \Phi_{k}$ has rank $r_k$. To reduce the system size further one could simply apply \emph{arbitrary} left sketches as 
\begin{equation}\label{eq:reduced CDEs bar}
\begin{split}
    & \sum_{x_1 = 1}^{n_1} \cdots \sum_{x_{k - 1} = 1}^{n_{k - 1}} S_{k-1}(\beta_{k-1}; x_{1 : k - 1})\sum_{\alpha_{k - 1} = 1}^{r_{k - 1}}  \bar \Phi_{k - 1}(x_{1 : k - 1}; \alpha_{k - 1}) G_{k}(\alpha_{k - 1}; x_{k}, \alpha_{k}) \\
    & = \sum_{x_1 = 1}^{n_1} \cdots \sum_{x_{k - 1} = 1}^{n_{k - 1}} S_{k-1}(\beta_{k-1}; x_{1 : k - 1}) \bar \Phi_k(x_{1:k-1}; x_{k},\alpha_{k}) 
\end{split}
\end{equation}
so long as the reduced CDEs remain well-posed. In this case, one could set $B_k = S_{k-1} \bar \Phi_k$ and $A_{k-1} = S_{k-1}  \bar \Phi_{k-1}$. 

Unfortunately, a complication arises when we use sketches $T_{k+1}$ with $l_k > r_k$. In this case we cannot simply solve \eqref{eq:CDEs bar} or \eqref{eq:reduced CDEs bar} as it gives TT with excessively large rank. We then need to apply a suitable further projection $q_k \in \mathbb{R}^{l_{k-1}\times r_{k-1}}, q_{k+1}\in \mathbb{R}^{l_k\times r_k}$ in \eqref{eq:CDEs bar} and \eqref{eq:reduced CDEs bar}
\begin{equation}
\bar \Phi_{k-1} \rightarrow  \bar \Phi_{k-1} q_{k},\quad \bar \Phi_k \rightarrow  \bar \Phi_k q_{k+1}
\end{equation}
treating $\bar\Phi_{k-1}$ and $\bar\Phi_k$ as matrices of size $n_1\cdots n_{k-1} \times l_{k-1}$ and $n_1\cdots n_{k} \times l_k$, respectively. This is the idea behind \textsc{Trimming}. However, rather than explicitly applying the projection $q_{k}$,  \textsc{Trimming} performs the projections implicitly, i.e., it directly gives
\begin{equation}\label{eq:Bk motivation}
     B_{k}(\beta_{k-1},x_k, \alpha_k) = \sum_{x_1,\ldots, x_{k-1}} S_{k-1}(\beta_{k-1}, x_{1 : k - 1}) \sum_{\gamma_k}  \bar \Phi_{k }(x_{1 : k - 1}, x_{k}, \gamma_{k }) q_{k+1}(\gamma_{k },\alpha_{k})
\end{equation}
via an SVD without obtaining the $q_k$'s. This presents a complication: in order to solve \eqref{eq:reduced CDEs bar}, one needs to form
\begin{equation}\label{eq:Ak motivation}
    A_{k-1}(\beta_{k - 1}, \alpha_{k - 1}) = \sum_{x_1,\ldots, x_{k-1}}S_{k-1}(\beta_{k-1}, x_{1 : k - 1})\sum_{\gamma_{k - 1}}  \bar \Phi_{k-1}(x_{1 : k - 1}, \gamma_{k-1}) q_{k}(\gamma_{k-1},\alpha_{k-1}),
\end{equation}
but all we have access to is $B_{k-1}$ which contains $q_k$ implicitly (note that $A_{k-1}$ is not $B_{k-1}$). It is unclear how to obtain $A_{k-1}$ without knowing $q_k$ explicitly.

There are two remedies for this. The first one is recursive left sketching, and the second one is to obtain projections (the $q_k$'s) directly. The second remedy is more complicated than the first, though it does not require recursive sketching. In this paper, we have focused primarily on the recursive left sketching approach, which allows us to obtain $A_k$'s directly from $B_k$'s. In the next subsection, we provide details of the second remedy in the following subsection.

\subsection{Non-recursive sketches}
Suppose now $S_1,\ldots,S_{d-1}$ in Algorithm~\ref{alg:1-2} are arbitrary sketches that are \emph{non-recursive}, meaning that they are not in the form of
\begin{equation}\label{eq:successive sketch}
    S_{k-1}(\beta_{k-1},x_{1 : k - 1}) = \sum_{\beta_{k-2}} s_{k-1}(\beta_{k-1},x_{k-1},\beta_{k-2}) S_{k-2}(\beta_{k-2},x_{1 : k - 2}).
\end{equation}
Evidently, \textsc{Trimming} gives the following expression for $B_k$ in terms of the sketched unfolding matrix (in Algorithm~\ref{alg:1-3}) and some ``gauge'' $q_{k+1}$
\begin{equation}
    B_k(\beta_{k-1},x_k; \alpha_k) = \sum_{\gamma_k}\tilde \Phi_k(\beta_{k-1},x_k; \gamma_k ) q_{k+1}(\gamma_k,  \alpha_k)
\end{equation}
where 
\begin{equation}
    q_{k+1} = V_k \Sigma_k^{-1}\quad 
\end{equation}
and
\begin{equation}
    \tilde \Phi_k \approx U_k \Sigma_k V_k^\top,\quad U_k \in \mathbb{R}^{m_{k-1} n_{k}\times r_k},\ \Sigma_k \in \mathbb{R}^{r_k\times r_k},\ V_k \in \mathbb{R}^{l_{k}\times r_k}
\end{equation}
being the best rank-$r_k$ approximation of $\tilde \Phi_k \in \mathbb{R}^{m_{k-1} n_k\times l_k}$ (defined in Algorithm~\ref{alg:1-3}) obtained via the SVD. Now, after obtaining the $q_k$'s in this manner, we can use them to construct the $A_{k}$'s in \eqref{eq:Ak motivation}. In this case, we do not need to use $B_k$'s to obtain $A_k$'s, as in the case when using recursive sketches. 

In what follows, we summarize this approach in TT-S (Algorithm~\ref{alg:TTS}) which removes the necessity of recursive sketching. The main difference between TT-S and TT-RS is that TT-S keeps track of the projection matrices $q_2,\ldots q_d$ in \eqref{eq:Ak motivation} obtained via Algorithm~\ref{alg:trimmingTTS} when performing \textsc{Trimming-TT-S} and uses them in Algorithm~\ref{alg:systemformingTTS}. In this way, one eliminates the need for obtaining the $A_{k}$'s via the $B_k$'s from recursive sketching. 
\begin{algorithm}
\caption{TT-S for a discrete function $p$.}
\label{alg:TTS}
\begin{algorithmic}[1]
\REQUIRE $p \colon [n_1] \times \cdots \times [n_d] \to \bR$ and target ranks $r_1, \ldots, r_{d - 1}$.
\REQUIRE $T_{k} \colon [n_{k}] \times \cdots \times [n_d] \times [\ell_{k - 1}] \to \bR$ with $\ell_{k - 1} \ge r_{k - 1}$ for $k = 2, \ldots, d$.
\REQUIRE $S_k \colon [m_k] \times [n_1]\times \cdots\times [n_k]  \to \bR$ for $k = 1, \ldots, d - 1$.
\STATE $\tilde \Phi_1, \ldots, \tilde \Phi_{d}, \bar \Phi_1, \ldots, \bar \Phi_{d - 1}\leftarrow \textsc{Sketching-TT-S}(p, T_2, \ldots, T_d, S_1, \ldots, S_{d - 1})$.
\STATE $B_1, \ldots, B_{d},q_2,\ldots,q_d \leftarrow \textsc{Trimming-TT-S}(\tilde \Phi_1, \ldots, \tilde \Phi_{d}, r_1,\ldots,r_{d-1})$.
\STATE $A_1, \ldots, A_{d - 1} \leftarrow \textsc{SystemForming-TT-S}(\bar \Phi_1, \ldots, \bar \Phi_{d - 1}, q_2,\ldots,q_d,S_1,  \ldots, S_{d - 1})$.
\STATE Solve the following $d$ matrix equations via least-squares for the variables $G_1 \colon [n_1] \times [r_1] \to \bR$, $G_k \colon [r_{k - 1}] \times [n_k] \times [r_{k}] \to \bR$ for $k = 2, \ldots, d - 1$, and $G_{d} \colon [r_{d - 1}] \times [n_d] \to \bR$: \vspace{2mm}
\begin{equation}
    \begin{aligned}
        G_1 & = B_1, \\
        \sum_{\alpha_{k-1} = 1}^{r_{k - 1}} A_{k-1}(\beta_{k-1}; \alpha_{k-1}) G_{k}(\alpha_{k-1}; x_{k}, \alpha_{k}) & = B_{k}(\beta_{k-1}; x_{k}, \alpha_{k}) \quad k = 2,\ldots , d - 1, \\
        \sum_{\alpha_{d-1} = 1}^{r_{d - 1}} A_{d-1}(\beta_{d-1}; \alpha_{d-1}) G_d(\alpha_{d-1}; x_d) & = B_d(\beta_{d-1}; x_{d}).
    \end{aligned}
\end{equation}
\RETURN $G_1, \ldots, G_d$.
\end{algorithmic}
\end{algorithm}

\begin{algorithm}
\caption{\textsc{Sketching-TT-S}.}
\label{alg:sketchingTTS}
\begin{algorithmic}
\REQUIRE $p$, $T_2, \ldots, T_{d}$, and $S_1, \ldots, S_{d - 1}$ as given in Algorithm \ref{alg:TTS}.
\FOR{$k = 1$ to $d - 1$}
    \STATE Right sketching: define $\bar{\Phi}_k \colon [n_1] \times \cdots \times [n_k] \times [\ell_k] \to \bR$ as
    \begin{equation*}
        \bar{\Phi}_k(x_{1 : k}, \gamma_{k}) = \sum_{x_{k+1} = 1}^{n_{k+1}} \cdots \sum_{x_{d} = 1}^{n_d} p(x_{1 : k}, x_{k + 1 : d}) T_{k+1}(x_{k+1 : d}, \gamma_{k}).
    \end{equation*}
    \IF {$k > 1$}
        \STATE Left sketching: define $\tilde{\Phi}_k \colon [m_{k - 1}] \times [n_k] \times [\ell_k] \to \bR$ as
        \begin{equation*}
            \tilde{\Phi}_{k}(\beta_{k-1}, x_{k}, \gamma_{k})
            =
            \sum_{x_1 = 1}^{n_1} \cdots \sum_{x_{k - 1} = 1}^{n_{k - 1}} S_{k-1}(\beta_{k-1}, x_{1 : k-1}) \bar{\Phi}_k(x_{1 : k - 1}, x_{k}, \gamma_{k}).
        \end{equation*}
    \ELSE
        \STATE Define
        \begin{equation*}
            \tilde \Phi_1(x_1,\gamma_1) = \bar \Phi_1(x_1,\gamma_1).
        \end{equation*}
    \ENDIF
\ENDFOR
\STATE Left sketching: define $\tilde{\Phi}_{d} \colon [m_{d - 1}] \times [n_d] \to \bR$ as
\begin{equation*}
    \tilde \Phi_{d}(\beta_{d - 1}, x_{d})
    =
    \sum_{x_1 = 1}^{n_1} \cdots \sum_{x_{d - 1} = 1}^{n_{d - 1}} S_{d-1}(\beta_{d-1}, x_{1 : d-1}) p(x_{1 : d - 1}, x_{d}).
\end{equation*}
\RETURN $\tilde{\Phi}_1, \ldots,  \tilde{\Phi}_{d}$, $\bar{\Phi}_1, \ldots,  \bar{\Phi}_{d - 1}$.
\end{algorithmic}
\end{algorithm}

\begin{algorithm}
\caption{\textsc{Trimming-TT-S}.}
\label{alg:trimmingTTS}
\begin{algorithmic}
\REQUIRE $\tilde{\Phi}_1, \ldots, \tilde{\Phi}_{d}$ from Algorithm \ref{alg:sketchingTTS}. 
\REQUIRE Target ranks $r_1, \ldots, r_{d - 1}$ as given in Algorithm \ref{alg:TTS}.
\FOR{$k = 1$ to $d - 1$}
    \IF {$k = 1$}
        \STATE Let $U_1 \Sigma_1 V^\top_1$, where $U_1\in \mathbb{R}^{n_1\times r_1}, V_1\in \mathbb{R}^{l_1\times r_1}, \Sigma_1\in \mathbb{R}^{r_1\times r_1}$, be the best rank-$r_1$ approximation to the matrix $\tilde \Phi_1(x_1; \alpha_1)$ via SVD. Define $B_1 \colon [n_1] \times [r_1] \to \bR$ where $B_1(x_1; \alpha_1) = U_1(x_1; \alpha_1)$. Furthermore, let $q_{2} = V_1 \Sigma_1^{-1}$.
    \ELSE
        \STATE Let $U_k \Sigma_k V^\top_k$, where $U_k\in \mathbb{R}^{m_{k-1}n_k\times r_k},  V_k\in \mathbb{R}^{l_k\times r_k}, \Sigma_k\in \mathbb{R}^{r_k\times r_k}$, be the best rank-$r_k$ approximation to the matrix $\tilde{\Phi}_{k}(\beta_{k-1}, x_{k}; \gamma_{k})$ via SVD. Define $B_{k} \colon [m_{k-1}] \times [n_k] \times [r_k] \to \bR$ where $B_{k}(\beta_{k - 1}, x_{k}; \alpha_{k}) = U_k(\beta_{k - 1}, x_{k}; \alpha_{k})$. Furthermore, let $q_{k+1} = V_k \Sigma_k^{-1}$.
    \ENDIF
\ENDFOR
\STATE  Let $B_d(\beta_{d-1}, x_d) = \tilde \Phi_d(\beta_{d-1}, x_d)$.
\RETURN $B_1, \ldots, B_{d}, q_2,\ldots,q_d$.
\end{algorithmic}
\end{algorithm}

\begin{algorithm}
\caption{\textsc{SystemForming-TT-S}.}
\label{alg:systemformingTTS}
\begin{algorithmic}
\REQUIRE $\bar \Phi_1, \ldots, \bar \Phi_{d - 1}$ from Algorithm \ref{alg:sketchingTTS}
\REQUIRE$q_2\ldots,q_d$ from Algorithm \ref{alg:trimmingTTS}.
\REQUIRE $S_1, \ldots, S_{d - 1}$ as given in Algorithm \ref{alg:TTS}.
\FOR{$k = 1$ to $d - 1$}
    \STATE Compute $A_k \colon [m_k] \times [r_k] \to \bR$:
    \begin{equation*}
        A_{k}(\beta_{k}, \alpha_{k}) 
        =  \sum_{x_1 = 1}^{n_1} \cdots \sum_{x_{k} = 1}^{n_{k}} 
        S_{k}(\beta_k, x_{1 : k})\sum^{l_k}_{\gamma_k=1} \bar \Phi_{k}(x_{1 : k}, \gamma_k) q_{k+1}(\gamma_k,\alpha_k).
    \end{equation*}
\ENDFOR
\RETURN $A_1,\ldots, A_{d - 1}$.
\end{algorithmic}
\end{algorithm}

\section{Continuous TT-RS}\label{app:cttrs}
In this section, we consider a general function $p \colon X_1 \times \cdots \times X_{d} \to \bR$, where $X_1, \ldots, X_{d} \subset \bR$. It turns out that everything presented in previous sections is still valid if we replace every discrete quantity with its continuous counterpart; concretely, we replace $[n_k]$, $[m_k]$, and $[\ell_k]$ with $X_k$, $\mathcal{B}_{k}$, and $\mathcal{C}_{k}$, respectively, where $\mathcal{B}_{k}$ and $\mathcal{C}_{k}$ are appropriate domains that can be chosen by model assumptions. Accordingly, we also replace all the summation over these sets with appropriate integration; for instance, replace $\sum_{x_{k} = 1}^{n_k}$ and $\sum_{\beta_k = 1}^{m_{k}}$ with $\int_{X_k} \, d x_{k}$ and $\int_{\mathcal{B}_{k}} \, d \beta_{k}$, respectively. As a result, we obtain Algorithms \ref{alg:2}, \ref{alg:2-2},  \ref{alg:2-3}, and \ref{alg:2-4} as continuous counterparts of Algorithms \ref{alg:1}, \ref{alg:1-2}, \ref{alg:1-3}, and \ref{alg:1-4}.

\begin{algorithm}
\caption{TT-RS-C for a continuous function $p$.}
\label{alg:2}
\begin{algorithmic}[1]
\REQUIRE  $ p \colon X_1 \times \cdots \times X_{d} \to \bR$ and target ranks $r_1, \ldots, r_{d - 1}$.
\REQUIRE $T_{k} \colon X_k \times \cdots \times X_{d} \times \mathcal{C}_{k - 1} \to \bR$ for $k = 2, \ldots, d$.
\REQUIRE $s_1 \colon \mathcal{B}_{1} \times X_1 \to \bR$ and $s_k \colon \mathcal{B}_{k} \times X_k \times \mathcal{B}_{k - 1} \to \bR$ for $k = 2, \ldots, d - 1$.
\STATE $\tilde{\Phi}_1, \ldots,  \tilde{\Phi}_{d} \leftarrow \textsc{Sketching-c}(p, T_2, \ldots, T_d, s_1, \ldots, s_{d - 1})$.
\STATE $B_1, \ldots, B_d \leftarrow \textsc{Trimming-c}(\tilde{\Phi}_1, \ldots,  \tilde{\Phi}_{d}, r_1, \ldots, r_{d - 1})$.
\STATE $A_1, \ldots, A_{d-1} \leftarrow \textsc{SystemForming-c}(B_1, \ldots, B_{d - 1}, s_1, \ldots, s_{d - 1})$.
\STATE Solve the following $d$ matrix equations via least-squares for the variables $G_1 \colon X_1 \times [r_1] \to \bR$, $G_k \colon [r_{k - 1}] \times X_k \times [r_{k}] \to \bR$ for $k = 2, \ldots, d - 1$, and $G_{d} \colon [r_{d - 1}] \times X_d \to \bR$. 
\begin{equation}
    \label{eq:alg2-CDEs}
    \begin{aligned}
        G_1 & = B_1, \\
        \sum_{\alpha_{k-1} = 1}^{r_{k - 1}} A_{k-1}(\beta_{k-1}; \alpha_{k-1}) G_{k}(\alpha_{k-1}; x_{k}, \alpha_{k}) & = B_{k}(\beta_{k-1}; x_{k}, \alpha_{k}) \quad k = 2, \ldots, d - 1, \\
        \sum_{\alpha_{d-1} = 1}^{r_{d - 1}} A_{d-1}(\beta_{d-1}; \alpha_{d-1}) G_d(\alpha_{d-1}; x_d) & = B_d(\beta_{d-1}; x_{d}).
    \end{aligned}
\end{equation}
\RETURN $G_1, \ldots, G_d$.
\end{algorithmic}
\end{algorithm}

First, note that the main algorithm for the continuous case, TT-RS-C (Algorithm \ref{alg:2}), has equations \eqref{eq:alg2-CDEs} which are exactly the same as \eqref{eq:alg-CDEs} of Algorithm \ref{alg:1}. Now, \eqref{eq:alg2-CDEs} are infinite-dimensional matrix equations, that is, coefficients and cores are functions. Also, the sketching algorithm for a continuous density (Algorithm \ref{alg:2-2}), which we call \textsc{Sketching-c}, is simply a modification of \textsc{Sketching} by replacing all the summations with integrals properly. We modify \textsc{Trimming} similarly to obtain its continuous counterpart \textsc{Trimming-c}. In this case, \textsc{Trimming-c} should be done by applying functional SVD \cites{functional, suli_mayers_2003} to $\tilde {\Phi}_1, \ldots, \tilde{\Phi}_{d - 1}$ to obtain $B_1, \ldots, B_{d - 1}$, respectively. We demonstrate how such a functional SVD works in the next subsection with a concrete example.

\begin{algorithm}
\caption{\textsc{Sketching-C}.}
\label{alg:2-2}
\begin{algorithmic}
\REQUIRE $p$, $T_2, \ldots, T_{d}$, and $s_1, \ldots, s_{d - 1}$ as given in Algorithm \ref{alg:2}.
\FOR{$k = 1$ to $d - 1$}
    \STATE Right sketching: define $\bar{\Phi}_k \colon X_1 \times \cdots \times X_k \times \mathcal{C}_{k} \to \bR$ as
    \begin{equation*}
        \bar{\Phi}_k(x_{1 : k}, \gamma_{k}) = \int p(x_{1 : k}, x_{k + 1 : d}) T_{k+1}(x_{k+1 : d}, \gamma_{k}) \, d x_{k + 1} \cdots d x_{d}.
    \end{equation*}
    \IF {$k > 1$}
                \STATE Left sketching: define $\tilde{\Phi}_k \colon \mathcal{B}_{k - 1} \times X_k \times \mathcal{C}_{k} \to \bR$ as
        \begin{equation*}
            \tilde{\Phi}_{k}(\beta_{k-1}, x_{k}, \gamma_{k})
            =
            \int S_{k-1}(\beta_{k-1}, x_{1 : k - 1}) \bar{\Phi}_k(x_{1 : k - 1}, x_{k}, \gamma_{k}) \, d x_{1} \cdots d x_{k - 1}.
        \end{equation*}
        \STATE Compute $S_k \colon \mathcal{B}_{k} \times X_1 \times \cdots \times X_k \to \bR$ for the next iteration: 
        \begin{equation*}
            S_k(\beta_k, x_{1 : k}) 
            = 
            \int s_k(\beta_{k}, x_{k}, \beta_{k - 1}) S_{k - 1}(\beta_{k - 1}, x_{1 : k - 1}) \, d \beta_{k - 1}.
        \end{equation*}
    \ELSE
        \STATE Define
        \begin{equation*}
            \tilde \Phi_1(x_1,\gamma_1) = \bar \Phi_1(x_1,\gamma_1).
        \end{equation*}
        \STATE Define sketch function
        \begin{equation*}
            S_1(\beta_1, x_1) = s_1(\beta_1, x_1).
        \end{equation*}
    \ENDIF
\ENDFOR
\STATE Left sketching: define $\tilde \Phi_d \colon \mathcal{B}_{d - 1} \times X_d \to \bR$ as
\begin{equation*}
    \tilde \Phi_d(\beta_{d - 1}, x_{d})
    =
    \int S_{d-1}(\beta_{d-1}, x_{1 : d-1}) p(x_{1 : d - 1}, x_{d}) \, d x_{1} \cdots d x_{d - 1}.
\end{equation*}
\RETURN $\tilde{\Phi}_1, \ldots,  \tilde{\Phi}_{d}$.
\end{algorithmic}
\end{algorithm}


\begin{algorithm}
\caption{\textsc{Trimming-C}.}
\label{alg:2-3}
\begin{algorithmic}
\REQUIRE $\tilde{\Phi}_1, \ldots, \tilde{\Phi}_{d}$ from Algorithm \ref{alg:2-2}. 
\REQUIRE Target ranks $r_1, \ldots, r_{d - 1}$ as given in Algorithm \ref{alg:2}.
\FOR{$k = 1$ to $d - 1$}
    \IF {$k = 1$}
        \STATE Compute the first $r_1$ left singular vectors of $\tilde {\Phi}_1(x_1; \gamma_1)$ and define $B_1 \colon X_1 \times [r_1] \to \bR$ so that these singular vectors are the columns of $B_1(x_1; \alpha_1)$.
    \ELSE
        \STATE Compute the first $r_k$ left singular vectors of $\tilde{\Phi}_{k}(\beta_{k-1}, x_{k}; \gamma_{k})$ and define $B_{k} \colon \mathcal{B}_{k - 1} \times X_{k} \times [r_k] \to \bR$ so that these singular vectors are the columns of $B_{k}(\beta_{k - 1}, x_{k}; \alpha_{k})$.
    \ENDIF 
\ENDFOR
\STATE Let $B_d(\beta_{d-1}, x_d) = \tilde \Phi_d(\beta_{d-1}, x_d)$.
\RETURN $B_1, \ldots, B_{d}$.
\end{algorithmic}
\end{algorithm}


\begin{algorithm}
\caption{\textsc{SystemForming-C}.}
\label{alg:2-4}
\begin{algorithmic}
\REQUIRE $B_1, \ldots, B_{d - 1}$ from Algorithm \ref{alg:2-3}. 
\REQUIRE $s_1, \ldots, s_{d - 1}$ as given in Algorithm \ref{alg:2}.
\FOR{$k = 1$ to $d - 1$}
    \IF {$k = 1$}
        \STATE Compute $A_1 \colon \mathcal{B}_{1} \times [r_1] \to \bR$:
        \begin{equation*}
            A_1(\beta_1, \alpha_{1}) 
            = \int s_1(\beta_1, x_1) B_1(x_1, \alpha_{1}) \, d x_{1}.
        \end{equation*}
    \ELSE
         \STATE Compute $A_k \colon \mathcal{B}_{k} \times [r_k] \to \bR$:
        \begin{equation*}
            A_{k}(\beta_{k}, \alpha_{k}) 
            = 
            \int s_{k}(\beta_k, x_{k}, \beta_{k-1}) B_{k}(\beta_{k-1}, x_{k}, \alpha_{k}) \, d x_{k} d \beta_{k - 1}.
        \end{equation*}
    \ENDIF 
\ENDFOR
\RETURN $A_1,\ldots, A_{d - 1}$.
\end{algorithmic}
\end{algorithm}


\subsection{Applying TT-RS-C to the Markov case}\label{section:continuous Markov}
In this subsection, we assume $p$ is a continuous Markov model, that is, $p$ is a continuous density and satisfies \eqref{eq:markov-condition}. For simplicity, we assume $X_1 = \cdots = X_d = [a, b]$ and $(\phi_{n})_{n \in \bN}$ be a countable orthonormal basis of $L^2([a, b])$ such that $\phi_1$ is a constant function, say $\phi_{1}(x) \equiv c$. Due to orthogonality, 
\begin{equation*}
    \int_{a}^{b} \phi_{n}(x) \, d x = 0
\end{equation*}
for all $n \ge 2$. Suppose each marginal density of $p$ is well approximated using the first $M$ basis functions $\phi_1, \ldots, \phi_M$. Based on Lemma \ref{lem:markov}, we now show that we can choose concrete sketch functions $T_2, \ldots, T_{d}$ and $s_1, \ldots, s_{d - 1}$ so that Algorithm \ref{alg:2} exactly recovers the cores of $p$, when provided with $\hat p = p$.

First, let $\mathcal{B}_{k} = \mathcal{C}_{k} = [M]$ for $k = 1, \ldots, d - 1$, where $r_1, \ldots, r_{d - 1} \le M$. Then, we define $T_{k+1} \colon X_{k+1} \times \cdots \times X_{d} \times \mathcal{C}_{k} \to \bR$ as
\begin{equation*}
    T_{k+1}(x_{k+1 : d}, \gamma_{k}) = \phi_{\gamma_{k}}(x_{k+1})
\end{equation*}
which gives
\begin{equation*}
    \bar{\Phi}_{k}(x_{1 : k}, \gamma_{k}) = \int p(x_{1 : k}, x_{k+1}) \phi_{\gamma_{k}}(x_{k+1}) \, d x_{k+1}. 
\end{equation*}
$T_{k + 1}$ marginalizes out $x_{k + 2}, \ldots, x_{d}$ as in the discrete case and it replaces the variable $x_{k + 1}$ with the index $\gamma_{k} \in [M]$ based on the fact that the marginal density can be approximated well by the fist $M$ basis functions. 

Similarly, define $S_{1} \colon \mathcal{B}_{1} \times X_{1} \to \bR$ such that $\mathcal{B}_{1} = X_{1}$
\begin{equation*}
    s_{1}(\beta_{1}, x_{1}) = \phi_{\beta_{1}}(x_{1}), 
    \quad
    s_{k}(\beta_{k}, x_{k}, \beta_{k-1}) = \phi_{\beta_{k}}(x_{k}) \delta(\beta_{k-1} - 1),
\end{equation*}
where $\delta$ is the Dirac delta function. Then,
\begin{equation*}
    S_k(\beta_k, x_{1 : k}) = \phi_{\beta_{k}}(x_{k}) \phi_1(x_{k-1}) \cdots \phi_1(x_1) = c^{k - 1} \phi_{\beta_{k}}(x_{k}),
\end{equation*}
thus
\begin{equation*}
    \begin{split}
        \tilde{\Phi}_{k}(\beta_{k - 1}, x_{k}, \gamma_{k})
        & = \int S_{k-1}(\beta_{k-1}, x_{1 : k-1}) \bar{\Phi}_k(x_{1 : k - 1}, x_{k}, \gamma_{k}) \, d x_1 \cdots d x_{k-1} \\
        & = \int c^{k - 2} \phi_{\beta_{k-1}}(x_{k-1}) p(x_{1 : k}, x_{k+1}) \phi_{\gamma_{k}}(x_{k+1}) \, d x_{k+1} d x_1 \cdots d x_{k-1}\\
        & = c^{k-2} \int \phi_{\beta_{k-1}}(x_{k-1}) p(x_{k-1}, x_k, x_{k+1}) \phi_{\gamma_{k}}(x_{k+1}) \, d x_{k+1} d x_{k - 1}.
    \end{split}
\end{equation*}
and
\begin{equation*}
    \begin{split}
        \tilde{\Phi}_{d}(\beta_{d - 1}, x_{d})
        & =
        \int S_{d-1}(\beta_{d-1}, x_{1 : d -1}) p(x_{1 : d - 1}, x_{d}) \, d x_1 \cdots d x_{d-1} \\
        & = c^{d-2} \int \phi_{\beta_{d - 1}}(x_{d-1}) p(x_{d-1}, x_{d}) \, d x_{d-1}.
    \end{split}
\end{equation*}
In other words, $S_{k - 1}$ marginalizes out variables $x_1, \ldots, x_{k - 2}$ as in the discrete case; furthermore, it replaces the variable $x_{k - 1}$ with the index $\beta_{k - 1}$ by integration against basis functions. 

Using the results $\tilde{\Phi}_1, \ldots, \tilde{\Phi}_{d}$ from \textsc{Sketching-C}, we now explain how to implement \textsc{Trimming-C} via functional SVD. The idea is to use basis expansion with respect to each node $x_k \in X_k$ and then apply SVD. For instance, consider $\tilde{\Phi}_1(x_1, \gamma_1)$. For large enough $M \in \bN$, we have
\begin{equation*}
    \tilde{\Phi}_1(x_1, \gamma_{1}) \approx \sum_{\beta_1 = 1}^{M} \nu_1(\beta_1, \gamma_{1}) \phi_{\beta_1}(x_1),
\end{equation*}
where $\nu_1 \colon [M] \times [M] \to \bR$ is given as
\begin{equation*}
    \nu_1(\beta_1, \gamma_{1}) = \int \tilde{\Phi}_1(x_1, \gamma_{1}) \phi_{\beta_1}(x_1) \, d x_1.
\end{equation*} 
Now, we can apply SVD to a matrix $\nu_1(\beta_1; \gamma_{1})$; compute the first $r_k$ left singular vectors of $\nu_1(\beta_1; \gamma_{1})$ and define $\tilde B_{1} \colon [M] \times [r_1] \to \bR$ so that these singular vectors are the columns of $\tilde B_1(\beta_1; \alpha_1)$. Then, we define $B_1 \colon X_1 \times [r_1] \to \bR$ as
\begin{equation*}
    B_1(x_1, \alpha_{1}) := \sum_{\beta_1 = 1}^{M} \tilde B_1(\beta_1, \alpha_{1}) \phi_{\beta_1}(x_1).
\end{equation*}
Then, 
\begin{equation*}
    A_1(\beta_1,\alpha_1) = \int  B_1(x_1, \alpha_{1}) \phi_{\beta_1}(x_1) dx_1 = \tilde B_1(\beta_1,\alpha_1).
\end{equation*}
Similarly, for $k = 2, \ldots, d - 1$, we have
\begin{equation*}
    \tilde{\Phi}_{k}(\beta_{k-1}, x_{k}, \gamma_{k})
    \approx
    \sum_{j_k = 1}^{M} \nu_{k}(\beta_{k-1}, j_{k}, \gamma_{k}) \phi_{j_k}(x_{k}),
\end{equation*}
where $\nu_{k} \colon [M] \times [M] \times [M] \to \bR$ is given as
\begin{equation*}
    \nu_{k}(\beta_{k-1}, j_{k}, \gamma_{k}) = \int \tilde{\Phi}_{k}(\beta_{k-1}, x_{k}, \gamma_{k}) \phi_{j_k}(x_{k}) \, d x_k.
\end{equation*}
We compute the first $r_k$ left singular vectors of $\nu_{k}(\beta_{k-1}, j_{k}; \gamma_{k})$ and define $\tilde{B}_{k} \colon [M] \times [M] \times [r_k] \to \bR$ so that these singular vectors are the columns of $\tilde{B}_{k}(\beta_{k - 1}, j_{k}; \alpha_{k})$. Then, we define $B_{k} \colon [M] \times X_{k} \times [r_k] \to \bR$ as
\begin{equation*}
    B_{k}(\beta_{k-1}, x_{k}, \alpha_{k}) 
    =
    \sum_{j_k = 1}^{M} \tilde{B}_{k}(\beta_{k-1}, j_{k}, \alpha_{k}) \phi_{j_k}(x_{k}),
\end{equation*}
which yields $A_{k} \colon [M] \times [r_k] \to \bR$ as
\begin{equation*}
\begin{split}
    A_{k}(\beta_{k}, \alpha_{k})
    &:= \int \sum_{\beta_{k - 1} = 1}^{M} s_{k}(\beta_k, x_{k}, \beta_{k-1}) B_{k}(\beta_{k-1}, x_{k}, \alpha_{k}) \, d x_{k} \\
    & = \int \phi_{\beta_{k}}(x_k) B_k(1, x_k, \alpha_k) \, d x_k \\
    & = \sum_{j_k = 1}^{M} \tilde{B}_{k}(1, j_{k}, \alpha_{k}) \int \phi_{\beta_{k}}(x_k) \phi_{j_k}(x_{k}) \, d x_k \\
    & = \tilde{B}_{k}(1, \beta_{k}, \alpha_k).
\end{split}
\end{equation*}
Lastly, we apply basis expansion to $B_{d}$ as well. Define
\begin{equation*}
    \tilde{B}_{d}(\beta_{d-1}, j_{d}) = \int B_{d}(\beta_{d - 1}, x_{d}) \phi_{j_{d}}(x_{d}) \, d x_{d}
\end{equation*}
so that 
\begin{equation*}
    B_{d}(\beta_{d-1}, x_{d})
    \approx
    \sum_{j_{d} = 1}^{M} \tilde{B}_{d}(\beta_{d-1}, j_{d}) \phi_{j_{d}}(x_{d}).
\end{equation*}
Now, one can easily verify that solving \eqref{eq:alg2-CDEs} for $G_1, \ldots, G_d$ amounts to solving 
\begin{equation}
\label{eq:coefficient-CDEs}
    \begin{aligned}
        g_1 & = \tilde{B}_1, \\
        \sum_{\alpha_{k-1} = 1}^{r_{k - 1}} A_{k-1}(\beta_{k-1}, \alpha_{k-1}) g_{k}(\alpha_{k-1}, j_{k}, \alpha_{k}) & = \tilde{B}_{k}(\beta_{k-1}, j_{k}, \alpha_{k}) \quad k = 2, \ldots, d - 1, \\
        \sum_{\alpha_{d-1} = 1}^{r_{d - 1}} A_{d-1}(\beta_{d-1}, \alpha_{d-1}) g_d(\alpha_{d-1}, j_d) & = \tilde{B}_d(\beta_{d-1}, j_{d})
    \end{aligned}
\end{equation}
for the variables $g_1 \colon [M] \times [r_1] \to \bR$, $g_k \colon [r_{k - 1}] \times [M] \times [r_{k}] \to \bR$ for $k = 2, \ldots, d - 1$, and $g_{d} \colon [r_{d - 1}] \times [M] \to \bR$ and letting 
\begin{align*}
    G_1(x_1, \alpha_1) & = \sum_{j_1 = 1}^{M} g_1(j_1, \alpha_1) \phi_{j_1}(x_1), \\
    G_k(\alpha_{k - 1}, x_{k}, \alpha_k) & = \sum_{j_k = 1}^{M} g_k(\alpha_{k - 1}, j_{k}, \alpha_{k}) \phi_{j_k}(x_k) \quad k = 2, \ldots, d - 1, \\
    G_d(\alpha_{d - 1}, x_{d}) & = \sum_{j_{d} = 1}^{M} g_{d}(\alpha_{d - 1}, j_{d}) \phi_{j_d}(x_{d}).
\end{align*}
In this case, the resulting TT-format is 
\begin{align*}
    & \sum_{\alpha_1 = 1}^{r_1} \cdots \sum_{\alpha_{d - 1} = 1}^{r_{d - 1}} G_1(x_1, \alpha_1) \cdots G_{d}(\alpha_{d-1}, x_{d}) \\
    & =
    \sum_{j_{1} = 1}^{M} \cdots \sum_{j_{d} = 1}^{M} \left(\sum_{\alpha_1 = 1}^{r_1} \cdots \sum_{\alpha_{d - 1} = 1}^{r_{d - 1}} g_1(j_1, \alpha_1) \cdots g_{d}(\alpha_{d-1}, j_{d})\right) \phi_{j_1}(x_1) \cdots \phi_{j_{d}}(x_{d}).
\end{align*}
We summarize the case of specializing Algorithm~\ref{alg:2} to the case of Markov density in Algorithm~\ref{alg:markov-continuous}. We note that one should be able to prove a result similar to Theorem \ref{thm:discrete-markov} under mild assumptions.

\begin{algorithm}
\caption{Main algorithm for a continuous Markov density.}
\label{alg:markov-continuous}
\begin{algorithmic}[1]
\REQUIRE $ p \colon [a, b]^d \to \bR$ and target ranks $r_1, \ldots, r_{d - 1}$.
\REQUIRE Orthonormal functions $\phi_1, \ldots, \phi_M$ in $L^2([a, b])$ with $\phi_1 \equiv c$ and $M \ge r_1, \ldots, r_{d - 1}$.
\FOR{$k = 1$ to $d - 1$}
    \IF {$k = 1$}
        \STATE Define $\nu_1 \colon [M] \times [M] \to \bR$ as
        \begin{equation*}
            \nu_1(\beta_1, \gamma_{1}) = \int \int p(x_1, x_{2}) \phi_{\beta_1}(x_1) \phi_{\gamma_{1}}(x_{2}) \, d x_1 d x_{2}.
        \end{equation*} 
        \STATE Compute the first $r_1$ left singular vectors of $\nu_1(\beta_1; \gamma_{1})$ and define $\tilde B_1 \colon [M] \times [r_1] \to \mathbb{R}$ so that these singular vectors are the columns of $\tilde B_1(\beta_1; \alpha_1)$.
        \STATE Define $A_{1} \colon [M] \times [r_1] \to \bR$ so that
        \begin{equation*}
            A_1 = \tilde B_1.
        \end{equation*}
    \ELSIF{$k < d$}
        \STATE Define $\nu_{k} \colon [M] \times [M] \times [M] \to \bR$ as
        \begin{equation*}
            \nu_{k}(\beta_{k - 1}, j_{k}, \gamma_{k})
            =
            c^{k-2} \int  p(x_{k-1}, x_k, x_{k+1}) \phi_{\beta_{k-1}}(x_{k-1}) \phi_{j_k}(x_{k}) \phi_{\gamma_{k}}(x_{k+1}) \, d x_{k - 1} d x_k d x_{k+1}.  
        \end{equation*}
        \STATE Compute the first $r_k$ left singular vectors of $\nu_{k}(\beta_{k - 1}, j_{k}; \gamma_{k})$ and define $\tilde{B}_{k} \colon [M] \times [M] \times [r_k] \to \bR$ so that these singular vectors are the columns of $\tilde{B}_{k}(\beta_{k - 1}, j_{k}; \alpha_{k})$.
        \STATE Define $A_k \colon [M] \times [r_k] \to \bR$ so that
        \begin{equation*}
            A_{k}(\beta_{k}, \alpha_{k})
            = \tilde{B}_{k}(1, \beta_{k}, \alpha_k).
        \end{equation*}
    \ELSE
        \STATE Define $\tilde{B}_{d} \colon [M] \times [M] \to \bR$ as
        \begin{equation*}
            \tilde{B}_{d}(\beta_{d-1}, j_{d}) = c^{d-2} \int \phi_{\beta_{d - 1}}(x_{d-1}) p(x_{d-1}, x_{d}) \phi_{j_{d}}(x_{d}) \, d x_{d-1} d x_{d}.
        \end{equation*}
    \ENDIF 
\ENDFOR
\STATE Solve the following $d$ matrix equations via least-squares for the variables $g_1 \colon [M] \times [r_1] \to \bR$, $g_k \colon [r_{k - 1}] \times [M] \times [r_{k}] \to \bR$ for $k = 2, \ldots, d - 1$, and $g_{d} \colon [r_{d - 1}] \times [M] \to \bR$. 
\begin{equation}
\label{eq:continuous-markov-CDEs}
    \begin{aligned}
        g_1 & = \tilde B_1, \\
        \sum_{\alpha_{k-1} = 1}^{r_{k - 1}} A_{k-1}(\beta_{k-1}; \alpha_{k-1}) g_{k}(\alpha_{k-1}; j_{k}, \alpha_{k}) & = \tilde{B}_{k}(\beta_{k-1}; j_{k}, \alpha_{k}) \quad k = 2, \ldots, d - 1, \\
        \sum_{\alpha_{d-1} = 1}^{r_{d - 1}} A_{d-1}(\beta_{d-1}; \alpha_{d-1}) g_d(\alpha_{d-1}; j_d) & = \tilde{B}_d(\beta_{d-1}; j_{d}).
    \end{aligned}
\end{equation}
\RETURN $G_1, \ldots, G_d$ by letting
\begin{align*}
    G_1(x_1, \alpha_1) & = \sum_{j_1 = 1}^{M} g_1(j_1, \alpha_1) \phi_{j_1}(x_1), \\
    G_k(\alpha_{k - 1}, x_{k}, \alpha_k) & = \sum_{j_k = 1}^{M} g_k(\alpha_{k - 1}, j_{k}, \alpha_{k}) \phi_{j_k}(x_k) \quad k = 2, \ldots, d - 1, \\
    G_d(\alpha_{d - 1}, x_{d}) & = \sum_{j_{d} = 1}^{M} g_{d}(\alpha_{d - 1}, j_{d}) \phi_{j_d}(x_{d}).
\end{align*}
\end{algorithmic}
\end{algorithm}
\clearpage

\section{Perturbation results}\label{app:pert_res}
This section provides perturbation results of Algorithm \ref{alg:1}. First, we prove that small perturbation on the coefficients and the right-hand sides of \eqref{eq:alg-CDEs} of Algorithm \ref{alg:1} leads to small perturbations of the cores. Using this result we show that Algorithm \ref{alg:1} with sketches \eqref{eq:sketchT_def} and \eqref{eq:sketchS_def} is robust against small perturbations for a discrete Markov density $p$. From this, we prove that Algorithm \ref{alg:1} with sketches \eqref{eq:sketchT_def} and \eqref{eq:sketchS_def} applied to the empirical density $\hat{p}$, which is constructed based on $N$ i.i.d.\ samples from a discrete density $p^\star$, recovers $p^\star$ with high probability given $N$ is large enough; a concrete sample complexity is then derived.

\subsection{Preliminaries}
In what follows, for a given vector $x$ we let $\|x\|$ and $\|x\|_\infty$ denote its Euclidean norm and its supremum norm, respectively. For a matrix $A$, we denote its spectral norm, Frobenius norm, and the $r$-th singular value by $\|A\|$, $\|A\|_F$, and $\sigma_r(A)$, respectively. With some abuse of notation, we also let $\|A\|_\infty$ denote the largest absolute value of the entries of $A$. Lastly, the orthogonal group in dimension $r$ is denoted by $O(r)$.

We also introduce the following norms for 3-tensors.
\begin{defn}
    For any 3-tensor $G \in \mathbb{R}^{n_1 \times n_2 \times n_3}$, or equivalently, $G \colon [n_1] \times [n_2] \times [n_3] \to \mathbb{R}$, we define the norm
    $$\normi{G} := \max_{i_2 \in [n_2]} \|G(\cdot, i_2, \cdot)\|.$$
    Here, $G(\cdot, i_2, \cdot) \in \mathbb{R}^{n_1 \times n_2}$ denotes a matrix, and $\| G(\cdot,i_2,\cdot)\|$ denotes its spectral norm. Also, we define $\|G\|_\infty$ by
    $$\|G\|_\infty = \max_{(i_1,i_2,i_3) \in [n_1] \times [n_2] \times [n_3]} |G(i_1,i_2,i_3)|.$$
\end{defn}

\begin{remark}\label{rmk:3-tensor-norm}
Such a norm $\normi{\cdot}$ is useful for bounding the norm of a contraction of cores. Throughout the section, we will analyze cores obtained by our algorithm: $G_1 \colon [n_1] \times [r_1] \to \bR$, $G_k \colon [r_{k - 1}] \times [n_k] \times [r_{k}] \to \bR$ for $k = 2, \ldots, d - 1$, and $G_{d} \colon [r_{d - 1}] \times [n_d] \to \bR$. For ease of exposition, for the specific matrices $G_1$ and $G_d$ produced by the algorithm (and any perturbations of them), set $\normi{G_1} = \max_{x_1 \in [n_1]} \|G(x_1, \cdot)\|$ and $\normi{G_d} = \max_{x_d \in [n_d]} \|G(\cdot, x_d)\|$. Then, one can easily verify that 
\begin{equation*}
    \|G_1 \circ \cdots \circ G_d\|_\infty \le \normi{G_1} \cdots \normi{G_d},
\end{equation*}
where $\|G_1 \circ \cdots \circ G_d\|_\infty$ denotes the supremum norm of the function $(G_1 \circ \cdots \circ G_d) \colon [n_1] \times \cdots \times [n_d] \to \mathbb{R}$. In summary, the supremum norm of the contraction is easily bounded by the product of $\normi{\cdot}$'s.
\end{remark}

We start with the following basic perturbation result on a linear system $Ax = b$. 
\begin{lemma}[Theorem 3.48 of \cite{wendland_2017}]
\label{lem:matrix_perturbation}
For $A \in \bR^{m \times n}$, suppose $\mathrm{rank}(A) = n \le m$. Let $\Delta A \in \bR^{m \times n}$ be a perturbation such that $\|A^\dagger\| \|\Delta A\| < 1$. Then, $\mathrm{rank}(A + \Delta A) = n$. Moreover, let $x$ and $x + \Delta x$ be least-squares solutions to linear systems $A x = b$ and $(A + \Delta A) x = b + \Delta b$, respectively. Then, 
\begin{equation*}
    \frac{\|\Delta x\|}{\|x\|} \le \frac{\|A\| \|A^\dagger\|}{1 - \|A^\dagger\| \|\Delta A\|} \left[\frac{\|\Delta A\|}{\|A\|} \left(1 + \kappa(A) \frac{\|A x - b\|}{\|A\| \|x\|}\right) + \frac{\|\Delta b\|}{\|A\| \|x\|}\right].
\end{equation*} 
\end{lemma}

Using this we prove the following lemma which bounds the perturbation of solutions of the tensor equation $A \circ X = B,$ where $A$ is a matrix, and both $X$ and $B$ are three-tensors. The contraction here is performed over the second index of $A$ and the first index of $X$. 

\begin{lemma}
\label{lem:total-perturbation}
For $A \in \bR^{m \times n}$ suppose $\mathrm{rank}(A) = n \le m$. Let $\Delta A \in \bR^{m \times n}$ be a perturbation such that $\|A^\dagger\| \|\Delta A\| < 1$. Then, $\mathrm{rank}(A + \Delta A) = n$. Let $B\in \mathbb{R}^{m\times \ell_1 \times \ell_2}$ and $\Delta B$ be its perturbation. Also, let $X \in \bR^{n \times \ell_1 \times \ell_2}$ and $X + \Delta X$ be least-squares solutions to the tensor equations $A\circ X = B$ and $(A + \Delta A) \circ X = B + \Delta B$, respectively. Suppose the column space of $B \in \bR^{m \times \ell_1\times \ell_2}$ is contained in that of $A$, then
\begin{equation*}
    \normi{\Delta X} \le \frac{\sqrt{2 m \ell_2}\|A^\dagger\|}{1 - \|A^\dagger\| \|\Delta A\|} \left( \|\Delta A\| \normi{X} +  \|\Delta B\|_\infty\right).
\end{equation*}
 In particular, if $\normi{X}\ge\chi >0$ for some constant $\chi,$ and $\Delta A$ satisfies  $\|A^\dagger\| \|\Delta A\| \le 1 / 2,$ then 
\begin{equation*}
    \frac{\normi{\Delta X}}{\normi{X}} \le {\sqrt{8 m \ell_2}  \|A^\dagger\|}\,\left(\|\Delta A\| + \|\Delta B\|_\infty\chi^{-1}\right).
\end{equation*}
\end{lemma}
\begin{proof}
    For any $i=(i_1,i_2)$ with $1\le i_1 \le \ell_1$ and $1 \le i_2 \le \ell_2$, we set $x_i= X(\cdot, i_1,i_2)$ and $b_i = B(\cdot,i_1,i_2)$ to be ``columns'' of $X$ and $B$ respectively. For each equation, since $b_i$ is contained in the column space of $A$, the previous lemma implies that 
    \begin{equation*}
        \|\Delta x_i\| 
        \le \frac{\|A\| \|A^\dagger\|}{1 - \|A^\dagger\| \|\Delta A\|} \left(\frac{\|\Delta A\|}{\|A\|} \|x_{i}\| + \frac{\|\Delta b_{i}\|}{\|A\|}\right)
        =
        \underbrace{\frac{\|A^\dagger\|}{1 - \|A^\dagger\| \|\Delta A\|}}_{=: C} \left(\|\Delta A\| \|x_i\| + \|\Delta b_{i}\|\right).
    \end{equation*}
    Now, for each $1\le i_1\le \ell_1,$
    \begin{equation*}
    \begin{split}
          \| \Delta X(\cdot, i_1, \cdot)\|_F
          & =\sum_{j=1}^{n}\sum_{i_2=1}^{\ell_2}\left| \Delta X(j, i_1, i_2)\right|^2\\ &=\sum_{i_2=1}^{\ell_2} \|\Delta x_{(i_1,i_2)}\|^2 \\
        & \le \sum_{i_2=1}^{\ell_2} C^2 (\|\Delta A\|\, \|x_{(i_1,i_2)}\| + \|\Delta b_{(i_1,i_2)}\|)^2 \\
        & \le \sum_{i_2=1}^{\ell_2} 2C^2 (\|\Delta A\|^2 \|x_{(i_1,i_2)}\|^2 + \|\Delta b_{(i_1,i_2)}\|^2) \\
        & = 2 C^2 (\|\Delta A\|^2 \|X(\cdot, i_1, \cdot)\|_F^2 + \|\Delta B(\cdot, i_1, \cdot)\|_F^2) \\
        & \le 2 C^2 (\ell_2 \|\Delta A\|^2 \|X(\cdot, i_1, \cdot)\|^2 + m \ell_2 \|\Delta B\|_\infty^2).
    \end{split}
    \end{equation*}
    Thus,
    \begin{equation*}
    \begin{split}
        \normi{\Delta X}
        & = \max_{i_1} \|\Delta X(\cdot, i_1, \cdot)\| \\
        & \le \max_{i_1} \|\Delta X(\cdot, i_1, \cdot)\|_F \\
        & \le \left(2 C^2 (\ell_2 \|\Delta A\|^2 \max_{i_1} \|X(\cdot, i_1, \cdot)\|^2 + m \ell_2 \|\Delta B\|_\infty^2)\right)^{1 / 2} \\
        & = \left(2 C^2 (\ell_2 \|\Delta A\|^2 \normi{X}^2 + m \ell_2 \|\Delta B\|_\infty^2)\right)^{1 / 2} \\
        & \le \sqrt{2 m \ell_2} C \left( \|\Delta A\| \normi{X} +  \|\Delta B\|_\infty\right),
    \end{split}
    \end{equation*}
    from which the rest of the result follows immediately.
\end{proof}

\begin{lemma}
    \label{lem:contraction_bound}
    Let $G_1 \colon [n_1] \times [r_1] \to \bR$, $G_k \colon [r_{k - 1}] \times [n_k] \times [r_{k}] \to \bR$ for $k = 2, \ldots, d - 1$, and $G_{d} \colon [r_{d - 1}] \times [n_d] \to \bR$. Denote their corresponding perturbations by $\Delta G_k$. Suppose that there exist $\delta_k > 0,$ $k=1,\dots,d$ such that $\normi{\Delta G_k} \le \delta_k \normi{G_k}$ for all $k = 1, \ldots, d$. Set 
    \begin{equation*}
        \Delta (G_1 \circ \cdots \circ G_d) := (G_1 + \Delta G_1) \circ \cdots \circ (G_d + \Delta G_d) - G_1 \circ \cdots \circ G_d.
    \end{equation*}
    Then
    \begin{equation*}
        \|\Delta (G_1 \circ \cdots \circ G_d)\|_\infty
        \le 
        \normi{G_1} \cdots \normi{G_d}\left(\sum_{k=1}^d \delta_k \right) \exp \left(\sum_{k=1}^d \delta_k \right).
    \end{equation*}
    \end{lemma}
    
    The following corollary is an immediate consequence of the previous lemma.
    \begin{cor}\label{cor:eps}
    Under the same assumptions as the previous lemma, let $\epsilon \in (0,1)$ be given. If  $\delta := \max_{1 \le k \le d} \delta_k \le \epsilon / (3 d)$  then
    \begin{equation*}
        \frac{\|\Delta (G_1 \circ \cdots \circ G_d)\|_\infty}{\normi{G_1} \cdots \normi{G_d}}
        \le \epsilon.
    \end{equation*}
\end{cor}
\begin{proof}[Proof of Lemma \ref{lem:contraction_bound}]
    For ease of exposition, we set $G_k^* = G_k +\Delta G_k$ for $k=1,\dots,d$. Next, we observe that
    \begin{equation}\label{eqn:telescope}
    \begin{split}
        \Delta (G_1 \circ \dots \circ G_d) =& (G_1^* \circ \dots \circ G_d^*) - (G_1 \circ G_2^* \circ \dots \circ G_d^*) \\
        &+(G_1\circ G_2^* \circ \dots \circ G_d^*) - (G_1\circ G_2 \circ G_3^* \circ \dots \circ G_d^*) \\
        & + \dots \\
        & + (G_1 \circ \dots \circ G_{d-1} \circ G_d^*) -(G_1 \circ \dots \circ G_d).
    \end{split}
    \end{equation}
    The first line on the right-hand side of the previous equation reduces to $\Delta G_1 \circ G_2^* \circ \dots G_d^*$. As in Remark \ref{rmk:3-tensor-norm}, 
    \begin{equation*}
        \|\Delta G_1 \circ G_2^\ast \circ \cdots \circ G_d^\ast\|_\infty
        \le
        \normi{\Delta G_1} \normi{G_2^\ast} \cdots \normi{G_d^\ast}.
    \end{equation*}
    Furthermore, for $k = 1, \ldots, d$,
    $$\normi{G_k+\Delta G_k} \le \normi{G_k}+\normi{\Delta G_k} \le (1+\delta_k)\normi{G_k},$$
    and hence
    $$ \|\Delta G_1 \circ G_2^\ast \circ \cdots \circ G_d^\ast\|_\infty
    \le \delta_1 \prod_{k = 2}^{d} (1 + \delta_k) \normi{G_1} \cdots \normi{G_d}.$$
    The other lines on the right-hand side of \eqref{eqn:telescope} can be bounded similarly. Thus, summing over all the terms on the right-hand side of \eqref{eqn:telescope}, we find
    $$\|\Delta (G_1 \circ \dots \circ G_d)\|_\infty \le \normi{G_1} \cdots \normi{G_d} \left(\sum_{k=1}^d \delta_k \right)\exp\left(\sum_{k=1}^d \delta_k \right),$$
    where we have used the fact that $1+x<\exp(x)$. 
\end{proof}

\begin{remark}
We note that in the previous lemma, the bounds we obtain are quite pessimistic, since they do not account for possible cancellations in contractions of multiple $G_k$'s. The product of $\normi{G_k}$'s could instead be replaced by the more cumbersome, but sharper, expression
$$ \max_k \max_{\sigma_\ell,\sigma_r\in\{0,1\}} \max_{x_1,\dots,x_d} \|G^{\sigma_\ell}_1\circ \dots \circ G_{k-1}^{\sigma_\ell}\| \cdot \normi{G_k}\cdot \|G^{\sigma_r}_{k+1} \circ \dots \circ G^{\sigma_r}_d\|, $$
where $G^0_k = G_k$ and $G^1_k = G^*_k.$
\end{remark}

\subsection{Perturbation results}

We have seen from Theorem \ref{thm:1} that under certain mild assumptions Algorithm \ref{alg:1} produces a well-defined set of matrix equations \eqref{eq:alg-CDEs}. The following result shows that small perturbations of the coefficients and the right-hand sides of \eqref{eq:alg-CDEs} result in small perturbations of the output of Algorithm \ref{alg:1}.

\begin{lemma}
\label{lem:perturbation}
Under the assumptions of Theorem \ref{thm:1}, let $G_1, \ldots, G_{d}$ be the solutions to \eqref{eq:alg-CDEs}. Given $\delta \in (0, 1)$, suppose that the coefficients and right-hand sides of \eqref{eq:alg-CDEs} are perturbed such that
\begin{equation*}
    \|\Delta A_{1}\|, \ldots, \|\Delta A_{d - 1}\|, \|\Delta B_1\|_\infty, \ldots, \|\Delta B_{d}\|_\infty
    \le \delta \beta 
    =: \delta \left(\sqrt{8 r \max(m, n)} c_A \left(1+\frac{1}{c_G}\right)\right)^{-1}
\end{equation*}
where the constants are defined as follows:
\begin{itemize}    
    \item $r = \max_{1 \le k \le d - 1} r_{k}$,
    \item $m = \max_{1 \le k \le d - 1} m_{k}$,
    \item $n = \max_{1 \le k \le d} n_{k}$,
    \item $c_{G} = \min_{1 \le k \le d} \normi{G_k}$,
    \item $c_A = 1 \vee \max_{1 \le k \le d-1} \|A_k^\dagger\|$.
\end{itemize}
Then, the perturbed version of \eqref{eq:alg-CDEs} has $G_k + \Delta G_k$ as least-squares solutions such that
\begin{equation*}
   \frac{\normi{\Delta G_k}}{\normi{G_k}} \le \delta.
\end{equation*}
\end{lemma}
\begin{proof}
    First, we compute a perturbation bound for the solution of the first equation: notice that $\sqrt{8 r \max(m, n)} c_A \ge \sqrt{n r}$, which implies $\beta \le \frac{c_G}{\sqrt{n r}}$, hence
    \begin{equation*}
        \frac{\normi{\Delta G_1}}{\normi{G_1}} 
        = \frac{\normi{\Delta B_1}}{\normi{G_1}} 
        \le \frac{\sqrt{n r} \|\Delta B_1\|_\infty}{\normi{G_1}} 
        \le\frac{\sqrt{n r} \beta \delta}{c_G}
        \le \frac{\sqrt{n r}}{c_G} \frac{c_G}{\sqrt{n r}} \delta
        = \delta.
    \end{equation*}
    Next, we observe that
    $$\|\Delta A_k\| \le \beta \le \frac{1}{2c_A} \le \frac{1}{2 \|A_k^\dagger\|},$$
    from which it follows that $\|\Delta A_k\|\, \|A_k^\dagger\| \le 1/2$ for all $k=1,\dots,d-1,$ and therefore we may apply Lemma \ref{lem:total-perturbation}. In particular,
    \begin{align*}
        \frac{\normi{\Delta G_k}}{\normi{G_k}} 
        & \le \sqrt{8 m_{k - 1} r_k} \|A_{k - 1}^\dagger\|\left(\|\Delta A_{k - 1}\| + \frac{\|\Delta B_k\|_\infty}{c_G}\right) \\
        & \le \sqrt{8 m r} c_A \left(1 + \frac{1}{c_G}\right) \beta \delta \\
        & \le \delta.
    \end{align*}
\end{proof}

Next, we analyze the effect of a perturbation $\Delta p$ of the input $p$ of Algorithm \ref{alg:1}. Having established Lemma \ref{lem:perturbation}, it suffices to quantify $\Delta A_k$ and $\Delta B_k$ in terms of $\Delta p$. First, the perturbation on $\tilde{\Phi}_k$ from \textsc{Sketching} is obvious; we may roughly say $\Delta \tilde{\Phi}_k \approx S_{k - 1} \circ \Delta p \circ T_{k + 1}$. Now that $B_k$ is obtained as the left singular vectors of $\tilde{\Phi}_k$ in \textsc{Trimming}, we invoke Wedin's theorem \cite{wedin_1972} to quantify $\Delta B_k$ in terms of $\Delta \tilde{\Phi}_k$. To this end, we first introduce the following distance comparing two 3-tensors up to rotation, which is common in spectral analysis of linear algebra, see Chapter 2 of \cite{ccfm_2021}.

\begin{defn}
    For any 3-tensors $\hat{G}, G \in \mathbb{R}^{r_1 \times n \times r_2}$, we define 
    \begin{equation*}
        \mathrm{dist}(\hat{G}, G) 
        := \min_{R_1 \in O(r_1), R_2 \in O(r_2)} \normi{\hat{G} - R_1 \circ G \circ R_2}.
    \end{equation*}
    Here, $R_1 \circ G \circ R_2$ denotes a 3-tensor formed by contracting the second index of $R_1$ and the first index of $G$ and contracting the first index of $R_2$ and the third index of $G$.
\end{defn}

Using this distance, we compare the $\hat{G}_1, \ldots, \hat{G}_d$, the which result from applying Algorithm \ref{alg:1} to $\hat{p} = p + \Delta p$ as input, with $G_1, \ldots, G_d$, the results of Algorithm \ref{alg:1} with $p$ as input. We will restrict our analysis to the case where $p$ is a Markov model and Algorithm \ref{alg:1} is implemented with sketches \eqref{eq:sketchT_def} and \eqref{eq:sketchS_def} as in Section \ref{section:discrete markov}.

\begin{remark} As in Remark \ref{rmk:3-tensor-norm}, we define $\mathrm{dist}(\cdot, \cdot)$ for the first and last cores as well. Accordingly, we set
\begin{align*}
    \mathrm{dist}(\hat{G}_1, G_1) & = \min_{R \in O(r_1)} \normi{\hat{G}_1 - G_1 R}, \\
    \mathrm{dist}(\hat{G}_d, G_d) & = \min_{R \in O(r_{d - 1})} \normi{\hat{G}_d - R G_d},
\end{align*}
where $G_1, \hat{G}_1 \colon [n_1] \times [r_1] \to \bR$ and $G_{d}, \hat{G}_d \colon [r_{d - 1}] \times [n_d] \to \bR$ are the first and last cores produced by the algorithm, respectively. Here $\normi{\cdot}$ on the right-hand sides of the previous equations are the norms defined for the first and last cores introduced in Remark \ref{rmk:3-tensor-norm}.
\end{remark}

\begin{prop}
\label{prop:perturbation-markov}
Under the assumptions of Theorem \ref{thm:discrete-markov}, let $G_1, \ldots, G_d$ be the cores of $p$ obtained as solutions to \eqref{eq:alg-CDEs}. Suppose we apply Algorithm \ref{alg:1} to the perturbed input $\hat{p} = p + \Delta p$ with sketches \eqref{eq:sketchT_def} and \eqref{eq:sketchS_def} as in Theorem \ref{thm:discrete-markov}; the results are denoted as $\hat{G}_1, \ldots, \hat{G}_d$. Suppose further that for some fixed $\delta \in (0, 1)$,   
\begin{equation}
\label{eq:Delta-p-condition}
\begin{split}
    & \|\Delta p(x_1; x_2)\|_\infty, 
    \|\Delta p(x_{1}, x_{2}; x_{3})\|_\infty, \ldots, \|\Delta p(x_{d - 2}, x_{d - 1}; x_{d})\|_\infty
    \|\Delta p(x_{d - 1}; x_{d})\|_\infty \\
    & \le \frac{c_P}{2n^2(1+c_P)}\left( \sqrt{8 r n} c_A \left(1+\frac{1}{c_G}\right)\right)^{-1} \delta 
    =: \gamma \delta 
\end{split}
\end{equation}
where the constants are defined as follows:
\begin{itemize}    
    \item $n = \max_{1 \le k \le d} n_{k}$,
    \item $c_P = \sigma_{r_1}(p(x_1; x_2)) \wedge \min_{k = 2, \ldots, d - 1} \sigma_{r_{k}}(p(x_{k - 1}, x_{k}; x_{k + 1}))$,
    \item $c_{G} = \min_{1\le k \le d} \normi{G_k}$,
    \item $c_A = 1 \vee \max_{1 \le k \le d - 1} \|A_k^\dagger\|$.
\end{itemize}
Then, for $k = 1, \ldots, d$,
\begin{equation*}
    \frac{\mathrm{dist}(\hat{G}_k, G_k)}{\normi{G_k}} \le \delta.
\end{equation*}
\end{prop}
\begin{proof}
We apply Algorithm \ref{alg:1} to $p$ and $\hat{p}$ with sketches \eqref{eq:sketchT_def} and \eqref{eq:sketchS_def} as in Theorem \ref{thm:discrete-markov}; the resulting coefficient matrices and right-hand sides of \eqref{eq:alg-CDEs} are
\begin{equation*}
    A_1, \ldots, A_{d - 1}, B_1, \ldots, B_{d - 1}, p(x_{d - 1}, x_{d})
    \quad \text{and} \quad
    \hat{A}_1, \ldots, \hat{A}_{d - 1}, \hat{B}_1, \ldots, \hat{B}_{d - 1}, \hat{p}(x_{d - 1}, x_{d}),
\end{equation*}
respectively. Our goal is to quantify their differences.

Recall that $B_1$ and $\hat{B}_1$ are the first $r_1$ left singular vectors of $p(x_1; x_2)$ and $\hat{p}(x_1; x_2)$, respectively. We apply Wedin's theorem presented in Theorem 2.9 of \cite{ccfm_2021}; if $\|\Delta p(x_1; x_2)\| < \sigma_{r_1}(p(x_1; x_2))$, we can find $R_1 \in O(r_1)$ such that 
\begin{equation*}
    \hat{B}_1(x_1; \alpha_1) = \sum_{a_1 = 1}^{r_1} B_1(x_1; a_1) R_1(a_1; \alpha_1) + E_1(x_1; \alpha_1),
\end{equation*}
and
\begin{equation*}
    \|E_1(x_1; \alpha_1)\| \le \frac{\sqrt{2} \|\Delta p(x_1; x_2)^\top B_1(x_1; a_1)\|}{\sigma_{r_1}(p(x_1; x_2)) - \|\Delta p(x_1; x_2)\|}.
\end{equation*}
In particular, if $\|\Delta p(x_1; x_2)\| \le (1 - 1 / \sqrt{2}) \sigma_{r_1}(p(x_1; x_2))$, using $\|B_1(x_1; a_1)\| = 1$, we have
\begin{equation*}
    \|E_1(x_1; \alpha_1)\| \le \frac{2 \|\Delta p(x_1; x_2)\|}{\sigma_{r_1}(p(x_1; x_2))}.
\end{equation*}

Similarly, for $k = 2, \ldots, d - 1$, if $\|\Delta p(x_{k - 1}, x_{k}; \alpha_{k})\| \le (1 - 1 / \sqrt{2}) \sigma_{r_k}(p(x_{k - 1}, x_{k}; x_{k + 1}))$, we can find $R_k \in O(r_k)$ such that 
\begin{equation*}
    \hat{B}_{k}(x_{k - 1}, x_{k}; \alpha_k) = \sum_{a_k = 1}^{r_{k}} B_{k}(x_{k - 1}, x_{k}; a_{k}) R_{k}(a_k; \alpha_{k}) + E_{k}(x_{k - 1}, x_{k}; \alpha_{k}),
\end{equation*}
and 
\begin{equation*}
    \|E_{k}(x_{k - 1}, x_{k}; \alpha_{k})\| \le \frac{2 \|\Delta p(x_{k-1}, x_{k}; x_{k + 1})\|}{\sigma_{r_k}(p(x_{k - 1}, x_{k}; x_{k + 1}))}.
\end{equation*}
Accordingly, for $k = 2, \ldots, d - 1$,
\begin{equation*}
    \hat{A}_{k}(x_{k}; \alpha_{k}) = \sum_{a_k = 1}^{r_{k}} A_k(x_k; a_k) R_{k}(a_k; \alpha_k) + \sum_{x_{k - 1} = 1}^{n_{k - 1}} E_{k}(x_{k - 1}, x_{k}; \alpha_{k}).
\end{equation*}

Conceptually speaking, we see that the perturbation in the coefficients and right-hand sides of equations \eqref{eq:alg-CDEs} for $G_1, \dots, G_d$ consist of two parts: a rotation and an additive error. We will see that though the rotations affect the individual $G_k$'s, they do not change the final contraction $G_1 \circ \dots \circ G_d,$ and hence do not directly contribute to the pointwise error in the compressed representation of the density. To that end, we define the rotated quantities $\Phi_1^\ast,$ $A_k^\ast,$ and $B_k^\ast$ as follows:
\begin{align*}
    B_1^\ast(x_1, \alpha_1) & := \sum_{a_1 = 1}^{r_1} B_1(x_1, a_1) R_1(a_1, \alpha_1) =: A_1^\ast(x_1, \alpha_1), \\
    B_{k}^\ast(x_{k-1}, x_{k}, \alpha_{k}) & := \sum_{a_k = 1}^{r_k} B_{k}(x_{k-1}, x_{k}, a_{k}) R_{k}(a_{k}, \alpha_{k}), \\
    A_{k}^\ast(x_k, \alpha_k) & := \sum_{x_{k - 1} = 1}^{n_{k - 1}} B_{k}^\ast(x_{k-1}, x_{k}, \alpha_{k}) = \sum_{a_k = 1}^{r_{k}} A_{k}(x_k, a_k) R_{k}(a_k, \alpha_k).
\end{align*}
Now, consider the following equations:
\begin{equation}
\label{eq:rotated}
    \begin{aligned}
        G_1^\ast &= B_1^\ast, \\
        \sum_{\alpha_{k - 1} = 1}^{r_{k - 1}} A_{k - 1}^\ast(x_{k - 1}, \alpha_{k - 1}) G_k^\ast(\alpha_{k - 1}, x_k, \alpha_k) &= B_k^\ast(x_{k-1}, x_k, \alpha_k) \quad k = 2, \ldots, d - 1 \\
        \sum_{\alpha_{d-1} = 1}^{r_{d - 1}} A_{d-1}^\ast(x_{d-1}, \alpha_{d-1}) G_d^\ast(\alpha_{d-1}, x_d) &= p(x_{d-1}, x_{d}).
    \end{aligned}
\end{equation}
These equations can be viewed as the rotated version of the original equations for $G_1,\dots,G_d$. In fact, the solutions are also simply rotated from the original solutions $G_1, \ldots, G_d$ as follows:\footnote{More simply, $G_1^\ast = G_1 R_1$, $G_k^\ast = R_{k - 1}^\top \circ G_k \circ R_k$ for $k = 2, \ldots, d - 1$, and $G_d^\ast = R_{d - 1}^\top G_d$.}
\begin{align*}
    G_1^\ast 
    & = \sum_{a_1 = 1}^{r_1} G_1(x_1, a_1) R_1(a_1, \alpha_1), \\
    G_k^\ast(\alpha_{k - 1}, x_{k}, \alpha_{k}) 
    & = \sum_{a_{k - 1} = 1}^{r_{k -1}} \sum_{a_{k} = 1}^{r_{k}} R_{k - 1}(a_{k - 1}, \alpha_{k - 1}) G_k(a_{k - 1}, x_{k}, a_{k}) R_k(a_{k}, \alpha_{k}) \quad k = 2, \ldots, d - 1, \\
    G_d^\ast(\alpha_{d - 1}, x_{d})
    & = \sum_{a_{d - 1} = 1}^{r_{d - 1}} R_{d - 1}(a_{d - 1}, \alpha_{d - 1}) G_d(a_{d - 1}, x_{d}).
\end{align*}
By definition, it is obvious that $\normi{G_k} = \normi{G_k^\ast}$ for all $k = 1, \ldots, d$ and $G_1\circ \dots \circ G_d = G_1^\ast\circ \dots \circ G_d^\ast$.

We now address the effect of the additive error. As a result of the above discussion, running our algorithm with input $\hat{p}$ amounts to a perturbed version of \eqref{eq:rotated}, where the coefficients and the right-hand sides are perturbed as follows:
\begin{align*}
    & \hat{B}_k = B_k^\ast + \Delta B_k^\ast, \quad \hat{A}_k = A_k^\ast + \Delta A_k^\ast \quad k = 1, \ldots, d - 1, \\
    & \hat{p}(x_{d - 1}; x_{d}) = p(x_{d - 1}; x_{d}) + \Delta p(x_{d - 1}; x_{d}).
\end{align*}
By construction, $\Delta B_1^\ast = \Delta A_1^\ast = E_1$, 
\begin{equation*}
    \Delta B_k^\ast = E_k, \quad \Delta A_k^\ast(x_{k}; \alpha_{k}) = \sum_{x_{k - 1} = 1}^{n_{k - 1}} E_k(x_{k - 1}, x_{k}; \alpha_{k}) \quad k = 2, \ldots, d - 1.
\end{equation*}
We now look for suitable bounds on $\hat{G}_k - G_k^\ast$ for $k = 1, \ldots, d$. In light of Lemma \ref{lem:perturbation}, it suffices to construct suitable bounds for $\|\Delta A_1^\ast\|, \ldots, \|\Delta A_{d-1}^\ast\|$, $\|\Delta B_1^\ast\|_\infty, \ldots, \|\Delta B_{d-1}^\ast\|_\infty$, and $\|\Delta p(x_{d - 1}; x_{d})\|_\infty$. In particular, we claim
\begin{equation}
    \label{eq:bounds-temporary}
    \|E_1\|, \|\Delta A_{2}^\ast\| \dots, \|\Delta A_{d-1}^\ast\|, \|E_1\|_\infty, \dots \|E_{d-1}\|_\infty, \|\Delta p(x_{d - 1}; x_{d})\|_\infty
    \le \beta \delta,
\end{equation}
where $\beta$ is as in Lemma \ref{lem:perturbation}, namely, 
\begin{equation*}
    \beta = \left( \sqrt{8 r n} c_A \left(1+\frac{1}{c_G}\right)\right)^{-1}.
\end{equation*}
Here, we use the fact that $m_k = n_k$ and $\max_{1 \le k \le d - 1} r_k \le n$. Essentially, we have
\begin{equation*}
    \gamma = \frac{c_P}{2n^2(1+c_P)} \beta.
\end{equation*} 
By definition of $G_k^\ast$ and $A_k^\ast$, it is obvious that $c_G = \min_{1 \le k \le d} \normi{G_k} = \min_{1 \le k \le d} \normi{G^\ast_k}$ and $c_A = \max_{1 \le k \le d - 1} \|A_k^\dagger\| = \max_{1 \le k \le d - 1} \|(A_k^\ast)^\dagger\|$. Hence, by Lemma \ref{lem:perturbation}, it suffices to check \eqref{eq:bounds-temporary} to prove that for $k = 1, \ldots, d$,
\begin{equation}
\label{eq:ttttt}
    \frac{\normi{\hat{G}_k - G_k^\ast}}{\normi{G_k^\ast}} \le \delta.
\end{equation}

Let us verify \eqref{eq:bounds-temporary}. First,
\begin{align*}
    & \|\Delta p(x_{d - 1}; x_{d})\|_\infty
    \le \gamma \delta \le \beta \delta.
\end{align*}
Moreover, as we showed above,
\begin{align*}
    & \|E_1\|_\infty 
    \le \|E_1\| 
    \le \frac{2 \|\Delta p(x_1; x_2)\|}{\sigma_{r_1}(p(x_1; x_2))}
    \le \frac{2 n \|\Delta p(x_1; x_2)\|_\infty}{\sigma_{r_1}(p(x_1; x_2))}
    \le \frac{2n}{c_P} \gamma \delta \le \beta \delta.
\end{align*}
For $k = 2, \ldots, d - 1$, we verify $\|\Delta A_k^\ast\| \le n^{1 / 2} \|E_k(x_{k - 1}, x_{k}; \alpha_{k})\|$. Note that $\Delta A_k^\ast = P_k E_k(x_{k - 1}, x_{k}; \alpha_{k})$; here $P_k \in \mathbb{R}^{n_k \times n_k n_{k - 1}} = [I_k, \ldots, I_k]$, where $I_k \in \mathbb{R}^{n_k \times n_k}$ is the identity matrix. Hence, $\|\Delta A_k^\ast\| \le \|P_k\| \|E_k(x_{k - 1}, x_{k}; \alpha_{k})\| \le n^{1 / 2} \|E_k(x_{k - 1}, x_{k}; \alpha_{k})\|$ because $\|P_k\| = \sqrt{n_{k - 1}} \le n^{1 / 2}$. Therefore,
\begin{equation*}
\begin{split}
    \|\Delta A_k^\ast\|, \|E_k\|_\infty
    & \le n^{1 / 2} \|E_{k}(x_{k - 1}, x_{k}; \alpha_{k})\| \\
    & \le \frac{2 n^{1 / 2} \|\Delta p(x_{k-1}, x_{k}; x_{k + 1})\|}{\sigma_{r_k}(p(x_{k - 1}, x_{k}; x_{k + 1}))} \\
    & \le \frac{2 n^{2} \|\Delta p(x_{k-1}, x_{k}; x_{k + 1})\|_\infty}{\sigma_{r_k}(p(x_{k - 1}, x_{k}; x_{k + 1}))} \\
    & \le \frac{2n^{2} \gamma}{c_P} \delta \le \beta \delta.
\end{split}
\end{equation*}
Hence, \eqref{eq:bounds-temporary} is satisfied, thus \eqref{eq:ttttt} holds. By definition of $\mathrm{dist}(\cdot, \cdot)$ and $\normi{\cdot}$, we have for $k = 1, \ldots, d$,
\begin{equation*}
    \frac{\mathrm{dist}(\hat{G}_k, G_k)}{\normi{G_k}} 
    \le \frac{\normi{\hat{G}_k - G_k^\ast}}{\normi{G_k}} 
    = \frac{\normi{\hat{G}_k - G_k^\ast}}{\normi{G_k^\ast}} \le \delta.
\end{equation*}
\end{proof}

The following result on the error of the contraction follows immediately from the previous Proposition, combined with Corollary \ref{cor:eps}.
\begin{theorem}\label{prop:markov-perturbation-total}
Under the assumptions of Theorem \ref{thm:discrete-markov}, let $G_1, \ldots, G_d$ be the cores of $p$ obtained as solutions to \eqref{eq:alg-CDEs}. Suppose we apply Algorithm \ref{alg:1} to the perturbed input $\hat{p} = p + \Delta p$ with sketches \eqref{eq:sketchT_def} and \eqref{eq:sketchS_def} as in Theorem \ref{thm:discrete-markov}; the results are denoted as $\hat{G}_1, \ldots, \hat{G}_d$. Suppose further that for some fixed $\epsilon \in (0, 1)$,   
\begin{equation*}
\begin{split}
    & \|\Delta p(x_1; x_2)\|_\infty, 
    \|\Delta p(x_{1}, x_{2}; x_{3})\|_\infty, \ldots, \|\Delta p(x_{d - 2}, x_{d - 1}; x_{d})\|_\infty
    \|\Delta p(x_{d - 1}; x_{d})\|_\infty \\
    & \le \frac{c_P}{6d n^2(1+c_P)}\left( \sqrt{8 r n} c_A \left(1+\frac{1}{c_G}\right)\right)^{-1} \epsilon,
\end{split}
\end{equation*}
where the constants $n, c_P, c_{G}, c_{A}$ are as in Proposition \ref{prop:perturbation-markov}. Then,
$$\frac{\|\hat{G}_1\circ \dots \circ \hat{G}_d - G_1 \circ \dots \circ G_d\|_\infty}{\normi{G_1}\dots \normi{G_d}} \le \epsilon. $$
\end{theorem}

\subsection{Estimation error analysis}
Lastly, we present a precise version of Theorem \ref{thm:markov-estimation-informal}. Recall that our main interest is to apply Algorithm \ref{alg:1} to an empirical density $\hat{p}$ constructed based on $N$ i.i.d.\ samples from some underlying density $p^\star$; letting $\hat{G}_1, \ldots, \hat{G}_d$ be the results of Algorithm \ref{alg:1} applied to $\hat{p}$, we hope to claim $p^\star \approx \hat{G}_1 \circ \cdots \circ \hat{G}_d$. 

Using the previous perturbation result (Proposition \ref{prop:perturbation-markov}), we will quantify a difference between $\hat{G}_k$ and $G_k^\star$, where $G_1^\star, \ldots, G_d^\star$ are the results of Algorithm \ref{alg:1} applied to $p^\star$. The only technicality here is that the perturbed input is not arbitrary, but given as an empirical density. Therefore, the perturbation $\hat{p} - p^\star$ can be represented in terms of the sample size $N$. The following lemma derives a concrete bound on $\hat{p} - p^\star$ using simple concentration inequalities.

\begin{lemma}
\label{lem:concentration}
Let $p^\star \colon [n_1] \times \cdots \times [n_d] \to \mathbb{R}$ be a density. Suppose $\hat{p}$ is an empirical density based on $N$ i.i.d.\ samples from $p^\star$. Let $\Delta p^\star = \hat{p} - p^\star$ and $n = \max_{1 \le k \le d} n_k$, then for any $\eta \in (0, 1)$, the following inequalities hold with probability at least $1 - \eta$:
\begin{align*}
    \|\Delta p^\star(x_1; x_2)\|_\infty & \le \sqrt{\frac{\log(2 n^2 d / \eta)}{2 N}}, \\
    \|\Delta p^\star(x_{k-1}, x_{k}; x_{k + 1})\|_\infty & \le \sqrt{\frac{\log(2 n^3 d / \eta)}{2 N}} \quad k = 2, \ldots, d - 1, \\    
    \|\Delta p^\star(x_{d - 1}; x_{d})\|_\infty & \le \sqrt{\frac{\log(2 n^2 d / \eta)}{2 N}}.
\end{align*}
\end{lemma}
\begin{proof}
Since $N \hat{p}$ is the sum of $N$ independent Bernoulli random variables, concentration inequalities imply that for any fixed $x_1 \in [n_1]$ and $x_2 \in [n_2]$ and $t \ge 0$,
\begin{equation*}
    \mathbb{P}(|\Delta p^\star(x_1, x_2)| > t) \le 2 e^{- 2 N t^2}.
\end{equation*}
Due to the union bound, $\|\Delta p^\star(x_1; x_2)\|_\infty \le t$ holds with probability at least $1 - 2 n^2 e^{- 2 N t^2}$. Equivalently, 
\begin{equation*}
    \|\Delta p^\star(x_1; x_2)\|_\infty \le \sqrt{\frac{\log(2 n^2 / \eta)}{2 N}}
\end{equation*}
holds with probability at least $1 - \eta$. Similarly, for $k = 2, \ldots, d - 1$, \begin{equation*}
    \|\Delta p^\star(x_{k-1}, x_{k}; x_{k + 1})\|_\infty 
    \le \sqrt{\frac{\log(2 n^3 / \eta)}{2 N}}
\end{equation*}
holds with probability at least $1 - \eta$. Due to the union bound,
\begin{align*}
    \|\Delta p^\star(x_1; x_2)\|_\infty 
    & \le \sqrt{\frac{\log(2 n^2 d / \eta)}{2 N}}, \\
    \|\Delta p^\star(x_{k-1}, x_{k}; x_{k + 1})\|_\infty 
    & \le \sqrt{\frac{\log(2 n^3 d / \eta)}{2 N}} \quad k = 2, \ldots, d - 1, \\    
    \|\Delta p^\star(x_{d - 1}; x_{d})\|_\infty 
    & \le \sqrt{\frac{\log(2 n^2 d / \eta)}{2 N}}
\end{align*}
hold with probability at least $1 - \eta$.
\end{proof}

Hence, we have proved that the perturbation $\hat{p} - p^\star$ is bounded above by $O(1 / \sqrt{N})$. Now, by comparing this bound with the right-hand sides of \eqref{eq:Delta-p-condition}, we obtain a complexity. Again, we will restrict our analysis to the case where $p^\star$ is a Markov model and Algorithm \ref{alg:1} is implemented with sketches \eqref{eq:sketchT_def} and \eqref{eq:sketchS_def} as in Section \ref{section:discrete markov}.

\begin{theorem}
\label{thm:markov-estimation}
Let $p^\star \colon [n_1] \times \cdots \times [n_d] \to \mathbb{R}$ be a Markov density satisfying Condition \ref{cond:transition} such that the rank of the $k$-th unfolding matrix of $p^\star$ is $r_k$ for each $k = 1, \ldots, d - 1$. Let $G_1^\star, \ldots, G_d^\star$ be the cores of $p^\star$ obtained by applying Algorithm \ref{alg:1} to $p^\star$ with sketches \eqref{eq:sketchT_def} and \eqref{eq:sketchS_def} as in Theorem \ref{thm:discrete-markov}; $A_1^\star, \ldots, A_{d - 1}^\star$ are the resulting coefficient matrices in \eqref{eq:alg-CDEs}.

Now, let $\hat{p}$ be an empirical density based on $N$ i.i.d.\ samples from $p^\star$. Let $\hat{G}_1, \ldots, \hat{G}_d$ be the results of applying Algorithm \ref{alg:1} to $\hat{p}$ with sketches \eqref{eq:sketchT_def} and \eqref{eq:sketchS_def} as in Theorem \ref{thm:discrete-markov}. Given $\delta \in (0, 1)$ and $\eta \in (0, 1)$, suppose 
\begin{equation}
\label{eq:sample-comlexity}
    N \ge 16 c_A^2 \left(1+\frac{1}{c_G}\right)^2 \left(1+\frac{1}{c_P}\right)^2 \frac{n^5 r \log(2 n^3 d / \eta)}{\delta^2},
\end{equation}
where 
\begin{itemize}    
    \item $n = \max_{1 \le k \le d} n_{k}$,
    \item $c_P = \sigma_{r_1}(p^\star(x_1; x_2)) \wedge \min_{k = 2, \ldots, d - 1} \sigma_{r_{k}}(p^\star(x_{k - 1}, x_{k}; x_{k + 1}))$,
    \item $c_{G} = \min_{1\le k \le d} \normi{G_k^\star}$,
    \item $c_A = 1 \vee \max_{1 \le k \le d - 1} \|(A_k^\star)^\dagger\|$.
\end{itemize}
Then, 
\begin{equation*}
    \frac{\mathrm{dist}(\hat{G}_k, G_k^\star)}{\normi{G_k^\star}} \le \delta \quad \forall k = 1, \ldots, d
\end{equation*}
with probability at least $1 - \eta$.
\end{theorem}
\begin{proof}
Due to Proposition \ref{prop:perturbation-markov}, it suffices to show that $N$ satisfies 
\begin{equation*}
    \sqrt{\frac{\log(2 n^3 d / \eta)}{2 N}} 
    \le \frac{c_P}{2n^2(1+c_P)}\left( \sqrt{8 r n} c_A \left(1+\frac{1}{c_G}\right)\right)^{-1} \delta,
\end{equation*}
which is equivalent to \eqref{eq:sample-comlexity}.
\end{proof}

In addition, using Proposition \ref{prop:markov-perturbation-total}, we obtain the following sample complexity for bounding the error of the contraction.

\begin{theorem}\label{thm:p_for_markov}
Let $p^\star \colon [n_1] \times \cdots \times [n_d] \to \mathbb{R}$ be a Markov density satisfying Condition \ref{cond:transition} such that the rank of the $k$-th unfolding matrix of $p^\star$ is $r_k$ for each $k = 1, \ldots, d - 1$. Let $G_1^\star, \ldots, G_d^\star$ be the cores of $p^\star$ obtained by applying Algorithm \ref{alg:1} to $p^\star$ with sketches \eqref{eq:sketchT_def} and \eqref{eq:sketchS_def} as in Theorem \ref{thm:discrete-markov}; $A_1^\star, \ldots, A_{d - 1}^\star$ are the resulting coefficient matrices in \eqref{eq:alg-CDEs}.

Now, let $\hat{p}$ be an empirical density based on $N$ i.i.d.\ samples from $p^\star$. Let $\hat{G}_1, \ldots, \hat{G}_d$ be the results of applying Algorithm \ref{alg:1} to $\hat{p}$ with sketches \eqref{eq:sketchT_def} and \eqref{eq:sketchS_def} as in Theorem \ref{thm:discrete-markov}. Given $\epsilon \in (0, 1)$ and $\eta \in (0, 1)$, suppose 

\begin{equation*}
    N \ge 144 c_A^2 \left(1+\frac{1}{c_G}\right)^2 \left(1+\frac{1}{c_P}\right)^2 \frac{d^2 n^5 r \log(2 n^3 d / \eta)}{\epsilon^2},
\end{equation*}
where 
\begin{itemize}    
    \item $n = \max_{1 \le k \le d} n_{k}$,
    \item $c_P = \sigma_{r_1}(p^\star(x_1; x_2)) \wedge \min_{k = 2, \ldots, d - 1} \sigma_{r_{k}}(p^\star(x_{k - 1}, x_{k}; x_{k + 1}))$,
    \item $c_{G} = \min_{1\le k \le d} \normi{G_k^\star}$,
    \item $c_A = 1 \vee \max_{1 \le k \le d - 1} \|(A_k^\star)^\dagger\|$.
\end{itemize}
Then,
$$\frac{\|\hat{G}_1\circ \dots \circ \hat{G}_d - G_1^\star \circ \dots \circ G_d^\star\|_\infty}{\normi{G_1^\star}\dots \normi{G_d^\star}} \le \epsilon$$
with probability at least $1 - \eta$.
\end{theorem}

\begin{remark}
    \label{rmk:constants}
    In Theorems \ref{thm:markov-estimation} and \ref{thm:p_for_markov}, notice that the constants $c_P, c_G, c_A$ are independent of $d$; to see this, observe that they are determined by the marginals of $p^\star$, namely, $p^\star(x_1; x_2)$ and $p^\star(x_{k - 1}, x_k; x_{k + 1})$, which are independent of $d$ under Condition \ref{cond:transition}. Therefore, we obtain Theorem \ref{thm:markov-estimation-informal} and Corollary \ref{cor:total}, where the upper bounds hide those constants under the ``big-$O$'' notation as they are independent of $d$. Meanwhile, notice that Theorems \ref{thm:markov-estimation} and \ref{thm:p_for_markov} are valid for $p^\star$ that may not satisfy Condition \ref{cond:transition}; in such a case, the constants $c_P, c_G, c_A$ may depend on $d$ in principle. Extensive numerical experiments, however, suggest that the constants $c_P, c_G, c_A$ are often nearly independent of $d$ for a broad class of Markov models that may not satisfy Condition \ref{cond:transition}, such as the Ginzburg-Landau model used in Section \ref{section:numerical}.
\end{remark}
\bibliographystyle{apalike}
\bibliography{ref}
\end{document}